\documentclass[a4paper]{article}%
\usepackage{amsmath}
\usepackage{amsfonts}
\usepackage{amssymb}
\usepackage{graphicx}%
\providecommand{\U}[1]{\protect\rule{.1in}{.1in}}
\newtheorem{theorem}{Theorem}

\newtheorem{lemma}[theorem]{Lemma}

\newtheorem{proposition}[theorem]{Proposition}

\newenvironment{proof}[1][Proof]{\noindent\textbf{#1.} }{\ \rule{0.5em}{0.5em}}
\begin{document}

\title{Construction of Toda flow\\via Sato-Segal-Wilson theory}
\author{Shuo Zhang$^{1}$, \ Shinichi Kotani$^{2}$, \ Jiahao Xu$^{3}$ \ }
\date{}
\maketitle

\begin{abstract}
A Toda flow is constructed on a space of bounded initial data through
Sato-Segal-Wilson theory. The flow is described by the Weyl functions of the
underlying Jacobi operators. This is a continuation of the previous work on
the KdV flow.

\end{abstract}

\section{Introduction}

The Toda\footnotetext[1]{Dept. of Math. Nanjing Univ. Nanjing China
shuozhang@smail.nju.edu.cn} lattice\footnotetext[2]{Dept. of Math. \ Osaka
Univ. \ Toyonaka Japan \ skotani@outlook.com} introduced\footnotetext[3]{Dept.
of Math. Nanjing Univ. Nanjing China 178881559@qq.com} by M. Toda in 1967 is a
simple model for a one-dimensional crystal in solid state physics, and it is
known to be one of the earliest examples of a non-linear completely integrable
system. Originally it was an infinite dimensional system of equations%
\[%
\begin{array}
[c]{l}%
\overset{\cdot}{p}_{n}(t)=e^{-\left(  q_{n}(t)-q_{n-1}(t)\right)
}-e^{-\left(  q_{n+1}(t)-q_{n}(t)\right)  }\\
\overset{\cdot}{q}_{n}(t)=p_{n}(t)
\end{array}
\]
with $n\in\mathbb{Z}$, and later it was rewritten in an equivalent form%
\begin{equation}%
\begin{array}
[c]{l}%
\overset{\cdot}{a}_{n}(t)=a_{n}(t)\left(  b_{n}(t)-b_{n-1}(t)\right) \\
\overset{\cdot}{b}_{n}(t)=2\left(  a_{n+1}(t)^{2}-a_{n}(t)^{2}\right)
\end{array}
\label{1}%
\end{equation}
by Flaschka variables%
\[
a_{n}(t)=\dfrac{1}{2}e^{-\left(  q_{n}(t)-q_{n-1}(t)\right)  /2}\text{,
\ \ }b_{n}(t)=-\dfrac{1}{2}p_{n}(t)\text{.}%
\]
This version is useful because it relates Toda lattice with a Jacobi operator%
\[
\left(  H_{q}u\right)  _{n}=a_{n+1}u_{n+1}+a_{n}u_{n-1}+b_{n}u_{n}\text{,}%
\]
where $a_{n}>0$, $b_{n}\in\mathbb{R}$ and $q\equiv\left\{  a_{n}%
,b_{n}\right\}  _{n\in\mathbb{Z}}$. The remarkable fact here is that if the
coefficients $q$ are replaced by a solution $q\left(  t\right)  \equiv\left\{
a_{n}\left(  t\right)  ,b_{n}\left(  t\right)  \right\}  _{n\in\mathbb{Z}}$ to
(\ref{1}), then the arising Jacobi operator $H_{q(t)}$ and the initial Jacobi
operator $H_{q(0)}$ have the same spectrum. Flaschka discovered in 1973 that
$H_{q(t)}$ satisfies the following operator equation%
\begin{equation}
\partial_{t}H_{q(t)}=\left[  P\left(  t\right)  ,H_{q(t)}\right]  \text{
\ }\left(  \equiv P\left(  t\right)  H_{q(t)}-H_{q(t)}P\left(  t\right)
\right)  \label{2}%
\end{equation}
with a skew symmetric operator $P(t)$%
\[
\left(  P\left(  t\right)  u\right)  _{n}=\dfrac{1}{2}\left(  a_{n}%
(t)u_{n+1}-a_{n-1}(t)u_{n-1}\right)  \text{,}%
\]
which establishes the unitary equivalence of $H_{q(t)}$ and $H_{q(0)}$. If
$P(t)$ is replaced by other suitable skew symmetric operators, we obtain
infinitely many solutions to a hierarchy of non-linear equations by (\ref{2}),
which is called Toda hierarchy. The pair $\left\{  H_{q(t)},P(t)\right\}  $ is
called a Lax pair and the operator $H_{q(t)}$ is the underlying operator for
this hierarchy of equations. There are many works treating Toda hierarchy and
main results obtained in the last century can be found in the book by G.
Teschl \cite{t}.

Recently in 2018 there appeared another point of view for the Toda hierarchy
by C. Remling \cite{r} and D. C. Ong-C. Remling \cite{or}. Remling considered
the Lax equation in a form
\begin{equation}
\partial_{t}H_{q(t)}=\left[  p\left(  H_{q(t)}\right)  _{a},H_{q(t)}\right]
\label{3}%
\end{equation}
with real polynomial $p$, where $X_{a}$ for a self-adjoint bounded operator
$X$ on $\ell^{2}\left(  \mathbb{Z}\right)  $ is defined by%
\begin{equation}
\left(  X_{a}\right)  _{jk}=\left\{
\begin{array}
[c]{cc}%
X_{jk} & \text{if }j<k\\
0 & \text{if }j=k\\
-X_{jk} & \text{if }j>k
\end{array}
\right.  \text{ \ with }X_{jk}=\left\langle \delta_{j},X\delta_{k}%
\right\rangle \text{.} \label{4}%
\end{equation}
He introduced a notion of cocycles or transfer matrices which map Weyl
functions for Jacobi operators to Weyl functions. This notion played a crucial
role in the study of spectral problems for ergodic Jacobi operators, so it is
expected that the cocycle property of the Toda flow is also significant to
investigate global behavior of solutions to the Toda lattice. Remling-Ong
expanded the Toda hierarchy from polynomials to entire functions preserving
the cocycle property.

The purpose of this article is to apply Sato's theory and provide another
approach to the study of the Toda hierarchy. In 1980 M. Sato \cite{sa}
obtained solutions to many integrable systems by constructing flows on an
infinite dimensional Grassmann manifold. Later in 1985 G. Segal-G. Wilson
\cite{sew} interpreted Sato's theory by Hardy space on the unit circle of the
complex plane $\mathbb{C}$. They considered Grassmann manifolds on
$L^{2}\left(  \left\vert z\right\vert =1\right)  $ consisting of closed
subspaces which are invariant under multiplication by $z^{n}$, and they showed
the KdV hierarchy is obtained when $n=2$. Segal-Wilson's approach was employed
by S. Kotani \cite{k} to construct general non-decaying solutions to the KdV
equation including many almost-periodic initial data. In the present article
we apply Segal-Wilson's method to the Toda hierarchy by following Kotani's work.

For the KdV case $z^{2}$ was employed as the spectral parameter for
Schr\"{o}dinger operators, which are the underlying operators for the KdV
hierarchy. In the Toda case we use $z+z^{-1}$ as its spectral parameter since
$\left\{  z^{n}\right\}  _{n\in\mathbb{Z}}$ forms a system of generalized
eigen-functions for the discrete Laplacian $u_{n+1}+u_{n-1}$. The bounded
domain $D_{+}$ in $\mathbb{C}$ which is fundamental in the following argument
is chosen so that it satisfies%
\begin{equation}%
\begin{array}
[c]{l}%
D_{+}\ni z\text{ \ }\longrightarrow\text{ \ }z^{-1}\text{, \ }\overline{z}\in
D_{+}\text{ (}\overline{z}\text{ denotes the complex conjugate of }z\text{)}\\
D_{+}\text{ contains the spectrum of Jacobi operators}%
\end{array}
\text{.}\label{5}%
\end{equation}
Throughout the paper we assume the spectrum \textrm{sp}$H_{q}$ of a Jacobi
operator $H_{q}$ is bounded, which is equivalent to the boundedness of the
coefficients $q=\left\{  a_{n},b_{n}\right\}  _{n\in\mathbb{Z}}$. Let%
\[
\phi\left(  z\right)  =z+z^{-1}\text{.}%
\]
If \textrm{sp}$H_{q}\subset\left[  -\lambda_{0},\lambda_{0}\right]  $ holds
for $\lambda_{0}>2$ ($\left[  -2,2\right]  $ is the spectrum for the discrete
Laplacian), then%
\[
\Sigma_{\lambda_{0}}\equiv\phi^{-1}\left(  \left[  -\lambda_{0},\lambda
_{0}\right]  \right)  =\left\{  \left\vert z\right\vert =1\right\}
\cup\left[  -\ell,-\ell^{-1}\right]  \cup\left[  \ell^{-1},\ell\right]
\]
with $\ell=\left(  \lambda_{0}+\sqrt{\lambda_{0}^{2}-4}\right)  /2$.
Therefore, $D_{+}$ should satisfy%
\[
\Sigma_{\lambda_{0}}\subset D_{+}\text{.}%
\]
Since $D_{+}$ is bounded, the origin $0$ is not an element of $D_{+}$,
otherwise the symmetry in (\ref{5}) implies $\infty=0^{-1}\in D_{+}$. We
denote%
\[%
\begin{array}
[c]{l}%
C=\partial D_{+}\text{ (the boundary of }D_{+}\text{), \ \ }D_{-}%
=\mathbb{C}\backslash\left(  D_{+}\cup C\right)  \\
C=C_{1}\cup C_{2}\text{ (}C_{1}\text{ is the outer curve and }C_{2}\text{ is
the inner curve of }C\text{)}%
\end{array}
\text{,}%
\]
and assume $C_{1}$, $C_{2}$ are closed smooth simple curves. Denote%
\[
D_{-}=D_{-}^{1}\cup D_{-}^{2}\text{ \ \ (2 disjoint domains),}%
\]
where $D_{-}^{1}$ is unbounded and $D_{-}^{2}$ containing $0$ is bounded.

The basic space is the Hilbert space $L^{2}\left(  C\right)  $, which is a
direct sum of 2 closed subspaces $H_{\pm}$:%
\[%
\begin{array}
[c]{l}%
H_{+}=L^{2}\text{-closure of }\left\{  \text{all rational functions with no
poles in }D_{+}\right\} \\
H_{-}=L^{2}\text{-closure of }\left\{
\begin{array}
[c]{l}%
\text{all functions }f\text{ on }D_{-}\text{ such that }\left.  f\right\vert
_{D_{-}^{j}}\text{ are}\\
\text{rational functions with no poles on }D_{-}^{j}\\
\text{for }j=1\text{, }2\text{ and }f(z)=o(1)\text{ as }z\rightarrow\infty
\end{array}
\right\}
\end{array}
\text{.}%
\]
The Hardy space $H_{+}$ is generated by $\left\{  z^{n}\right\}
_{n\in\mathbb{Z}}$, and%
\[
r\left(  z\right)  =\left\{
\begin{array}
[c]{cc}%
z^{m} & \text{for \ }z\in D_{-}^{1}\\
z^{n} & \text{for \ }z\in D_{-}^{2}%
\end{array}
\right.  \text{ with }m\leq-1\text{ and }n\geq0
\]
is an element of $H_{-}$. The projections $\mathfrak{p}_{\pm}$ from
$L^{2}\left(  C\right)  $ onto $H_{\pm}$ are obtained by%
\[%
\begin{array}
[c]{ll}%
\left(  \mathfrak{p}_{+}f\right)  \left(  z\right)  =\dfrac{1}{2\pi i}%
{\displaystyle\int_{C}}
\dfrac{f\left(  \lambda\right)  }{\lambda-z}d\lambda & \text{for }z\in D_{+}\\
\left(  \mathfrak{p}_{-}f\right)  \left(  z\right)  =\dfrac{1}{2\pi i}%
{\displaystyle\int_{C}}
\dfrac{f\left(  \lambda\right)  }{z-\lambda}d\lambda & \text{for }z\in D_{-}%
\end{array}
\text{,}%
\]
where integrals on $C$ are defined by%
\[%
{\displaystyle\int_{C}}
f\left(  \lambda\right)  d\lambda=%
{\displaystyle\int_{C_{1}}}
f\left(  \lambda\right)  d\lambda+%
{\displaystyle\int_{C_{2}}}
f\left(  \lambda\right)  d\lambda\text{.}%
\]
$C_{1}$ is oriented anti-clockwise and $C_{2}$ clockwise. It is known that
$\mathfrak{p}_{\pm}$ are bounded on $L^{2}\left(  C\right)  $ (see \cite{v})
and satisfy%
\[
\mathfrak{p}_{\pm}^{2}=\mathfrak{p}_{\pm}\text{, \ \ \ }\mathfrak{p}%
_{+}+\mathfrak{p}_{-}=I\text{ \ (the identity operator on }L^{2}\left(
C\right)  \text{).}%
\]
In our Toda case Sato-Segal-Wilson theory is performed on a Glassmann manifold
$Gr^{\left(  toda\right)  }$ consisting of closed subspaces $W$ of
$L^{2}\left(  C\right)  $ satisfying%
\begin{equation}
\phi\left(  z\right)  W\subset W\text{.} \label{6}%
\end{equation}
Additionally we assume also the comparability of $W$ with $H_{+}$, that is%
\begin{equation}
\mathfrak{p}_{+}:W\longrightarrow H_{+}\text{ \ \ is bijective.} \label{7}%
\end{equation}
Generally a bounded vector function%
\[
\boldsymbol{a}\left(  \lambda\right)  =\left(  a_{1}\left(  \lambda\right)
,a_{2}\left(  \lambda\right)  \right)  \text{ \ on }C
\]
gives such $W$ by%
\begin{equation}
W_{\boldsymbol{a}}\equiv\boldsymbol{a}H_{+}=\left\{  \boldsymbol{a}u\text{;
\ }u\in H_{+}\right\}  \text{,} \label{8}%
\end{equation}
where the product $\boldsymbol{a}u$ is defined by%
\begin{equation}
\left(  \boldsymbol{a}u\right)  \left(  \lambda\right)  =a_{1}\left(
\lambda\right)  u\left(  \lambda\right)  +a_{2}\left(  \lambda\right)
\lambda^{-1}u\left(  \lambda^{-1}\right)  \text{.} \label{9}%
\end{equation}
The invariance $\phi\left(  \lambda\right)  =\phi\left(  \lambda^{-1}\right)
$ implies%
\[
\phi\left(  \lambda\right)  \left(  \boldsymbol{a}u\right)  \left(
\lambda\right)  =a_{1}\left(  \lambda\right)  \left(  \phi u\right)  \left(
\lambda\right)  +a_{2}\left(  \lambda\right)  \lambda^{-1}\left(  \phi
u\right)  \left(  \lambda^{-1}\right)  =\left(  \boldsymbol{a}\left(  \phi
u\right)  \right)  \left(  \lambda\right)  \text{,}%
\]
hence $\phi W_{\boldsymbol{a}}\subset W_{\boldsymbol{a}}$ holds due to $\phi
H_{+}\subset H_{+}$, which shows (\ref{6}). Moreover, defining an operator on
$H_{+}$ by%
\[
T\left(  \boldsymbol{a}\right)  u=\mathfrak{p}_{+}\left(  \boldsymbol{a}%
u\right)  \text{ \ for }u\in H_{+}\text{,}%
\]
one sees that the property (\ref{7}) for $W_{\boldsymbol{a}}$ is equivalent to%
\[
T\left(  \boldsymbol{a}\right)  :H_{+}\rightarrow H_{+}\text{ \ is bijective.}%
\]
The operator $T\left(  \boldsymbol{a}\right)  $ is called Toeplitz operator
with symbol $\boldsymbol{a}$. Hereafter we consider $W_{\boldsymbol{a}}$
instead of a general $W$.

Sato introduced tau-function to represent his flows on a Grassmann manifold.
In our context it is defined as follows. Set%
\begin{equation}
\boldsymbol{A}^{inv}\left(  C\right)  =\left\{  \boldsymbol{a}=\left(
a_{1},a_{2}\right)  \text{; }\sup_{\lambda\in C}\left\Vert \boldsymbol{a}%
\left(  \lambda\right)  \right\Vert <\infty\text{ and }T\left(  \boldsymbol{a}%
\right)  \text{ is invertible on }H_{+}\right\}  \label{10}%
\end{equation}
with $\left\Vert \boldsymbol{a}\right\Vert =\sqrt{\left\vert a_{1}\right\vert
^{2}+\left\vert a_{2}\right\vert ^{2}}$ and a group $\Gamma$ acting on the
space of symbols%
\begin{equation}
\Gamma=\left\{
\begin{array}
[c]{c}%
g=re^{h}\text{; }r\text{ is a rational function with no poles nor zeros in
}\Sigma_{\lambda_{0}}\\
\text{and }h\text{ is analytic in a neighborhood of }\Sigma_{\lambda_{0}}%
\end{array}
\right\}  \text{.} \label{11}%
\end{equation}
The domain $D_{+}$ for each $g\in\Gamma$ is chosen so that $g$ has no poles
nor zeros in $\overline{D}_{+}$ (the closure). For $\boldsymbol{a}%
\in\boldsymbol{A}^{inv}\left(  C\right)  $, $g\in\Gamma$ define%
\begin{equation}
\tau_{\boldsymbol{a}}\left(  g\right)  =\det\left(  g^{-1}T\left(
g\boldsymbol{a}\right)  T\left(  \boldsymbol{a}\right)  ^{-1}\right)  \text{.}
\label{12}%
\end{equation}
Since $g^{-1}T\left(  g\boldsymbol{a}\right)  T\left(  \boldsymbol{a}\right)
^{-1}-I$ is a trace class operator on $H_{+}$, one can define $\tau
_{\boldsymbol{a}}\left(  g\right)  $. Since we are trying to construct the
Toda flow on a subclass of $\boldsymbol{A}^{inv}\left(  C\right)  $, the
property $g\boldsymbol{a}\in\boldsymbol{A}^{inv}\left(  C\right)  $ is crucial
and this is valid if and only if $\tau_{\boldsymbol{a}}\left(  g\right)
\neq0$. Therefore, it is one of the main issues in this paper to find a class
of symbols $\boldsymbol{a}$ satisfying $\tau_{\boldsymbol{a}}\left(  g\right)
\neq0$ for sufficiently enough $g\in\Gamma$. Since we are interested only in
real valued solutions, we assume $\overline{\boldsymbol{a}}=\boldsymbol{a}$,
where generally the conjugation $\overline{f}$ is defined by
\[
\overline{f}\left(  z\right)  =\overline{f\left(  \overline{z}\right)
}\text{.}%
\]
Then, it is easily seen that $\tau_{\boldsymbol{a}}\left(  g\right)
\in\mathbb{R}$ if $\overline{\boldsymbol{a}}=\boldsymbol{a}$ and $g\in
\Gamma_{\operatorname{real}}$ with%
\begin{equation}
\Gamma_{\operatorname{real}}=\left\{  g\in\Gamma\text{; \ }g=\overline
{g}\right\}  \text{.} \label{13}%
\end{equation}
If $\tau_{\boldsymbol{a}}\left(  z^{n}\right)  >0$ holds for any
$n\in\mathbb{Z}$, then the coefficients $\left\{  a_{n}\left(  \boldsymbol{a}%
\right)  ,b_{n}\left(  \boldsymbol{a}\right)  \right\}  _{n\in\mathbb{Z}}$ are
obtained by%
\begin{equation}
a_{n}\left(  \boldsymbol{a}\right)  =\sqrt{\dfrac{\tau_{\boldsymbol{a}}\left(
z^{n}\right)  \tau_{\boldsymbol{a}}\left(  z^{n-2}\right)  }{\tau
_{\boldsymbol{a}}\left(  z^{n-1}\right)  ^{2}}}\text{, \ \ \ }b_{n}\left(
\boldsymbol{a}\right)  =\left.  \partial_{\varepsilon}\log\dfrac
{\tau_{\boldsymbol{a}}\left(  z^{n}e^{\varepsilon z}\right)  }{\tau
_{\boldsymbol{a}}\left(  z^{n-1}e^{\varepsilon z}\right)  }\right\vert
_{\varepsilon=0}\text{.} \label{14}%
\end{equation}
This formula resembles with the conventional representation of $\left\{
a_{n},b_{n}\right\}  _{n\in\mathbb{Z}}$ by the moment problem, but the present
one is a representation on the whole $\mathbb{Z}$.

Set%
\[
Q_{\lambda_{0}}=\left\{  q=\left\{  a_{n},b_{n}\right\}  _{n\in\mathbb{Z}%
}\text{; \ }a_{n}>0\text{, }b_{n}\in\mathbb{R}\text{ and \textrm{sp}}%
H_{q}\subset\left[  -\lambda_{0},\lambda_{0}\right]  \right\}
\]
and let $m_{\pm}$ be the Weyl functions of $H_{q}$ (see Appendix). Define an
analytic function $m$ on $\mathbb{C}\backslash\Sigma_{\lambda_{0}}$ by%
\[
m\left(  z\right)  =\left\{
\begin{array}
[c]{ll}%
z+z^{-1}+a_{1}^{2}m_{+}\left(  z+z^{-1}\right)  & \text{if }z\in
\mathbb{C}\backslash\Sigma_{\lambda_{0}}\text{ and }\left\vert z\right\vert
>1\\
-a_{0}^{2}m_{-}\left(  z+z^{-1}\right)  +b_{0} & \text{if }z\in\mathbb{C}%
\backslash\Sigma_{\lambda_{0}}\text{ and }\left\vert z\right\vert <1
\end{array}
\right.  \text{.}%
\]
It is known that $q$ is completely recovered from $m$. If we define a symbol
$\boldsymbol{m}$%
\[
\boldsymbol{m}\left(  z\right)  =\left(  \dfrac{zm(z)-1}{z^{2}-1},z^{2}%
\dfrac{z-m\left(  z\right)  }{z^{2}-1}\right)  \text{,}%
\]
then it is proved that $\tau_{\boldsymbol{m}}\left(  g\right)  >0$ for any
$g\in\Gamma_{\operatorname{real}}\mathcal{\ }$and one can define $q\left(
g\boldsymbol{m}\right)  =\left\{  a_{n}\left(  g\boldsymbol{m}\right)
,b_{n}\left(  g\boldsymbol{m}\right)  \right\}  _{n\in\mathbb{Z}}$ by
(\ref{14}). $q\left(  g\boldsymbol{m}\right)  \in Q_{\lambda_{0}}$ is also
shown if $\boldsymbol{m}$ is generated by $\left\{  m_{\pm},a_{0},a_{1}%
,b_{0}\right\}  $ of $q\in Q_{\lambda_{0}}$, hence one can define%
\[
\mathrm{Toda}\left(  g\right)  q=q\left(  g\boldsymbol{m}\right)  \in
Q_{\lambda_{0}}\text{ \ for }q\in Q_{\lambda_{0}}\text{.}%
\]
Our main theorem is

\begin{theorem}
\label{t1}$\left\{  \mathrm{Toda}\left(  g\right)  \right\}  _{g\in
\Gamma_{\operatorname{real}}}$ define a flow on $Q_{\lambda_{0}}$ and yields
the Toda hierarchy by choosing $g_{t}=e^{tp}$ for any real polynomial $p$.
Especially $\mathrm{Toda}\left(  e^{tz}\right)  q$ provides a Toda lattice.
\end{theorem}

This theorem implies that one can construct the Toda flow on the space of
bounded initial data $\left\{  a_{n},b_{n}\right\}  _{n\in\mathbb{Z}}$, and
the group $\Gamma_{\operatorname{real}}$ is more general than entire functions
treated by Ong-Remling. One can define the cocycles for the present flow.

One of the advantages for this construction is the possibility to extend the
flow to a flow on an unbounded initial data $\left\{  a_{n},b_{n}\right\}
_{n\in\mathbb{Z}}$. For this purpose one has to replace the domain $D_{+}$
with a domain containing $\left(  -\infty,0\right)  \cup\left(  0,\infty
\right)  $ and trace the whole argument below. This will be a future work.

\section{Basic properties of $W_{\boldsymbol{a}}$}

For a bounded vector symbol $\boldsymbol{a}$ on $C$ the space
$W_{\boldsymbol{a}}$ is defined in (\ref{8}). In this section we show basic
properties of $W_{\boldsymbol{a}}$ for $\boldsymbol{a}\in\boldsymbol{A}%
^{inv}\left(  C\right)  $.

Define%
\[
Rf(z)=z^{-1}f\left(  z^{-1}\right)  \text{.}%
\]
Since $C$ is invariant under the transformation $z\rightarrow z^{-1}$, $R$
acts on functions on $C$. We have

\begin{lemma}
\label{l1}Identities%
\[
\mathfrak{p}_{\pm}R=R\mathfrak{p}_{\pm}\text{ on }L^{2}\left(  C\right)
\text{, \ \ }T\left(  \boldsymbol{a}\right)  R=RT\left(  \widetilde
{\boldsymbol{a}}\right)  \text{ on }H_{+}%
\]
are valid, where $\widetilde{\boldsymbol{a}}\left(  \lambda\right)  =\left(
a_{1}\left(  \lambda^{-1}\right)  ,a_{2}\left(  \lambda^{-1}\right)  \right)
$. Especially this implies $\widetilde{\boldsymbol{a}}\in\boldsymbol{A}%
^{inv}\left(  C\right)  $ if $\boldsymbol{a}\in\boldsymbol{A}^{inv}\left(
C\right)  $.
\end{lemma}

\begin{proof}
For $f\in L^{2}\left(  C\right)  $ note%
\begin{align*}
\int_{C}f\left(  \lambda^{-1}\right)  d\lambda &  =\int_{C_{1}}f\left(
\lambda^{-1}\right)  d\lambda+\int_{C_{2}}f\left(  \lambda^{-1}\right)
d\lambda\\
&  =-\int_{C_{1}}f\left(  \lambda\right)  \lambda^{-2}d\lambda-\int_{C_{2}%
}f\left(  \lambda\right)  \lambda^{-2}d\lambda=-\int_{C}f\left(
\lambda\right)  \lambda^{-2}d\lambda\text{.}%
\end{align*}
Then one has%
\[
\mathfrak{p}_{+}\left(  Rf\right)  \left(  z\right)  =\dfrac{1}{2\pi i}%
\int_{C}\dfrac{\lambda^{-1}f\left(  \lambda^{-1}\right)  }{\lambda-z}%
d\lambda=z^{-1}\dfrac{1}{2\pi i}\int_{C}\dfrac{f\left(  \lambda\right)
}{\lambda-z^{-1}}d\lambda=R\left(  \mathfrak{p}_{+}f\right)  \left(  z\right)
\text{.}%
\]
The second identity follows from $\boldsymbol{a}R=R\widetilde{\boldsymbol{a}}$
and the first identity.\medskip
\end{proof}

For $\boldsymbol{a}\in\boldsymbol{A}^{inv}\left(  C\right)  $ and
$n\in\mathbb{Z}$ define a sequence of functions of $H_{-}$ by%
\[
\varphi_{\boldsymbol{a}}^{\left(  n\right)  }(z)=\mathfrak{p}_{-}\left(
\boldsymbol{a}T\left(  \boldsymbol{a}\right)  ^{-1}z^{n}\right)  \in
H_{-}\text{.}%
\]

\begin{lemma}
\label{l2}$\varphi_{\boldsymbol{a}}^{\left(  n\right)  }$ for $\boldsymbol{a}%
\in\boldsymbol{A}^{inv}\left(  C\right)  $ satisfies the following
identities:\newline%
\[%
\begin{array}
[c]{l}%
\text{(i) \ }\varphi_{\widetilde{\boldsymbol{a}}}^{\left(  n\right)
}(z)=z^{-1}\varphi_{\boldsymbol{a}}^{\left(  -n-1\right)  }(z^{-1})\\
\text{(ii)}%
\begin{array}
[c]{l}%
\left(  z+z^{-1}\right)  \left(  z^{n}+\varphi_{\boldsymbol{a}}^{\left(
n\right)  }\right)  =z^{n+1}+\varphi_{\boldsymbol{a}}^{\left(  n+1\right)
}+z^{n-1}+\varphi_{\boldsymbol{a}}^{\left(  n-1\right)  }\\
\text{ \ \ \ \ \ \ \ \ \ \ \ \ \ \ \ \ \ \ \ }+\varphi_{\widetilde
{\boldsymbol{a}}}^{\left(  -n-1\right)  }(0)\left(  1+\varphi_{\boldsymbol{a}%
}^{\left(  0\right)  }\right)  +\varphi_{\boldsymbol{a}}^{\left(  n\right)
}(0)\left(  z^{-1}+\varphi_{\boldsymbol{a}}^{\left(  -1\right)  }\right)
\end{array}
\end{array}
\text{.}%
\]
(ii) shows that linear combinations of $\left\{  1+\varphi_{\boldsymbol{a}%
}^{\left(  0\right)  }\text{, }z^{-1}+\varphi_{\boldsymbol{a}}^{\left(
-1\right)  }\right\}  $ generate $W_{\boldsymbol{a}}$.
\end{lemma}

\begin{proof}
Lemma \ref{l1} implies%
\begin{align*}
\varphi_{\widetilde{\boldsymbol{a}}}^{\left(  n\right)  }  &  =\mathfrak{p}%
_{-}\left(  \widetilde{\boldsymbol{a}}T\left(  \widetilde{\boldsymbol{a}%
}\right)  ^{-1}z^{n}\right)  =\mathfrak{p}_{-}\left(  \widetilde
{\boldsymbol{a}}RT\left(  \boldsymbol{a}\right)  ^{-1}R^{-1}z^{n}\right)
=\mathfrak{p}_{-}\left(  R\boldsymbol{a}T\left(  \boldsymbol{a}\right)
^{-1}z^{-n-1}\right) \\
&  =R\mathfrak{p}_{-}\left(  \boldsymbol{a}T\left(  \boldsymbol{a}\right)
^{-1}z^{-n-1}\right)  =R\varphi_{\boldsymbol{a}}^{\left(  -n-1\right)
}\text{,}%
\end{align*}
which is (i).

To show (ii) first note $z^{n}+\varphi_{\boldsymbol{a}}^{\left(  n\right)
}\in W_{\boldsymbol{a}}$, hence $\left(  z+z^{-1}\right)  \left(
z^{n}+\varphi_{\boldsymbol{a}}^{\left(  n\right)  }\right)  \in
W_{\boldsymbol{a}}$ is valid. Now decompose this quantity into elements of
$H_{\pm}$, namely%
\[
\left(  z+z^{-1}\right)  \left(  z^{n}+\varphi_{\boldsymbol{a}}^{\left(
n\right)  }\right)  =u+v
\]
with%
\[
\left\{
\begin{array}
[c]{l}%
u=z^{n+1}+z^{n-1}+\lim_{z\rightarrow\infty}z\varphi_{\boldsymbol{a}}^{\left(
n\right)  }(z)+\varphi_{\boldsymbol{a}}^{\left(  n\right)  }(0)z^{-1}\in
H_{+}\\
v=z\varphi_{\boldsymbol{a}}^{\left(  n\right)  }(z)-\lim_{z\rightarrow\infty
}z\varphi_{\boldsymbol{a}}^{\left(  n\right)  }(z)+z^{-1}\left(
\varphi_{\boldsymbol{a}}^{\left(  n\right)  }(z)-\varphi_{\boldsymbol{a}%
}^{\left(  n\right)  }(0)\right)  \in H_{-}%
\end{array}
\text{.}\right.
\]
Here note due to (i)%
\[
\lim_{z\rightarrow\infty}z\varphi_{\boldsymbol{a}}^{\left(  n\right)
}(z)=\varphi_{\widetilde{\boldsymbol{a}}}^{\left(  -n-1\right)  }(0)\text{.}%
\]
The above $u\in H_{+}$ has the origin in $W_{\boldsymbol{a}}$ of%
\[
z^{n+1}+\varphi_{\boldsymbol{a}}^{\left(  n+1\right)  }+z^{n-1}+\varphi
_{\boldsymbol{a}}^{\left(  n-1\right)  }+\varphi_{\widetilde{\boldsymbol{a}}%
}^{\left(  -n-1\right)  }(0)\left(  1+\varphi_{\boldsymbol{a}}^{\left(
0\right)  }\right)  +\varphi_{\boldsymbol{a}}^{\left(  n\right)  }(0)\left(
z^{-1}+\varphi_{\boldsymbol{a}}^{\left(  -1\right)  }\right)  \text{,}%
\]
hence the bijectivity of $\mathfrak{p}_{+}:W_{\boldsymbol{a}}\rightarrow
H_{+}$ implies the identity (ii).\bigskip
\end{proof}

Set%
\begin{equation}%
\begin{array}
[c]{l}%
\Delta_{\boldsymbol{a}}\left(  \zeta\right) \\
=\dfrac{\left(  1+\varphi_{\boldsymbol{a}}^{\left(  0\right)  }\left(
\zeta\right)  \right)  \left(  \zeta+\varphi_{\boldsymbol{a}}^{\left(
-1\right)  }\left(  \zeta^{-1}\right)  \right)  -\left(  1+\varphi
_{\boldsymbol{a}}^{\left(  0\right)  }\left(  \zeta^{-1}\right)  \right)
\left(  \zeta^{-1}+\varphi_{\boldsymbol{a}}^{\left(  -1\right)  }\left(
\zeta\right)  \right)  }{\zeta-\zeta^{-1}}%
\end{array}
\text{.} \label{15}%
\end{equation}
Later we need non-vanishing of $\Delta_{\boldsymbol{a}}\left(  \zeta\right)  $.

\begin{lemma}
\label{l3}$\Delta_{\boldsymbol{a}}\left(  \zeta\right)  \neq0$ holds for
$\boldsymbol{a}\in\boldsymbol{A}^{inv}\left(  C\right)  $ and $\zeta\in D_{-}$.
\end{lemma}

\begin{proof}
Set $u_{0}=T\left(  \boldsymbol{a}\right)  ^{-1}1$, $u_{-1}=T\left(
\boldsymbol{a}\right)  ^{-1}z^{-1}$. Since%
\[
\varphi_{\boldsymbol{a}}^{\left(  0\right)  }=\mathfrak{p}_{-}\left(
\boldsymbol{a}u_{0}\right)  \text{, \ \ \ \ }\varphi_{\boldsymbol{a}}^{\left(
-1\right)  }=\mathfrak{p}_{-}\left(  \boldsymbol{a}u_{-1}\right)
\]
holds, one has%
\[
\boldsymbol{a}u_{0}\left(  z\right)  =1+\varphi_{\boldsymbol{a}}^{\left(
0\right)  }\left(  z\right)  \text{,}\ \ \ \boldsymbol{a}u_{-1}\left(
z\right)  =z^{-1}+\varphi_{\boldsymbol{a}}^{\left(  -1\right)  }\left(
z\right)  \text{.}%
\]
Hence, observing $r\left(  z\right)  =\left(  \zeta-z\right)  ^{-1}\left(
\zeta-z^{-1}\right)  ^{-1}$ satisfies $r\left(  z\right)  =r\left(
z^{-1}\right)  $, we have a decomposition into $H_{\pm}$:%
\begin{align*}
\boldsymbol{a}ru_{0}\left(  z\right)   &  =r\boldsymbol{a}u_{0}\left(
z\right)  =r\left(  z\right)  \left(  1+\varphi_{\boldsymbol{a}}^{\left(
0\right)  }\left(  z\right)  \right) \\
&  =\left(  \dfrac{\varphi_{\boldsymbol{a}}^{\left(  0\right)  }\left(
z\right)  -\varphi_{\boldsymbol{a}}^{\left(  0\right)  }\left(  \zeta
^{-1}\right)  }{\left(  \zeta-z^{-1}\right)  \left(  \zeta-z\right)  }%
-\dfrac{\varphi_{\boldsymbol{a}}^{\left(  0\right)  }\left(  \zeta\right)
-\varphi_{\boldsymbol{a}}^{\left(  0\right)  }\left(  \zeta^{-1}\right)
}{\left(  \zeta-\zeta^{-1}\right)  \left(  \zeta-z\right)  }\right) \\
&  \text{ \ \ \ \ \ \ \ \ \ }+\left(  \dfrac{\varphi_{\boldsymbol{a}}^{\left(
0\right)  }\left(  \zeta\right)  -\varphi_{\boldsymbol{a}}^{\left(  0\right)
}\left(  \zeta^{-1}\right)  }{\left(  \zeta-\zeta^{-1}\right)  \left(
\zeta-z\right)  }+\dfrac{1+\varphi_{\boldsymbol{a}}^{\left(  0\right)
}\left(  \zeta^{-1}\right)  }{\left(  \zeta-z\right)  \left(  \zeta
-z^{-1}\right)  }\right)  \text{,}%
\end{align*}
which yields%
\begin{align*}
\left(  T\left(  \boldsymbol{a}\right)  ru_{0}\right)  \left(  z\right)   &
=\dfrac{\varphi_{\boldsymbol{a}}^{\left(  0\right)  }\left(  \zeta\right)
-\varphi_{\boldsymbol{a}}^{\left(  0\right)  }\left(  \zeta^{-1}\right)
}{\left(  \zeta-\zeta^{-1}\right)  \left(  \zeta-z\right)  }+\dfrac
{1+\varphi_{\boldsymbol{a}}^{\left(  0\right)  }\left(  \zeta^{-1}\right)
}{\left(  \zeta-z\right)  \left(  \zeta-z^{-1}\right)  }\\
&  =\left(  z^{-1}\dfrac{1+\varphi_{\boldsymbol{a}}^{\left(  0\right)
}\left(  \zeta^{-1}\right)  }{\zeta-z^{-1}}+\zeta\dfrac{1+\varphi
_{\boldsymbol{a}}^{\left(  0\right)  }\left(  \zeta\right)  }{\zeta-z}\right)
\dfrac{1}{\zeta^{2}-1}\text{.}%
\end{align*}
Similarly we have%
\[
\left(  T\left(  \boldsymbol{a}\right)  ru_{-1}\right)  \left(  z\right)
=\left(  z^{-1}\dfrac{\zeta+\varphi_{\boldsymbol{a}}^{\left(  -1\right)
}\left(  \zeta^{-1}\right)  }{\zeta-z^{-1}}+\zeta\dfrac{\zeta^{-1}%
+\varphi_{\boldsymbol{a}}^{\left(  -1\right)  }\left(  \zeta\right)  }%
{\zeta-z}\right)  \dfrac{1}{\zeta^{2}-1}\text{.}%
\]
If we regard $-\zeta\left(  \zeta-z\right)  ^{-1}$ is unknown in these two
identities, one obtains%
\[
T\left(  \boldsymbol{a}\right)  \left(  \left(  \zeta+\varphi_{\boldsymbol{a}%
}^{\left(  -1\right)  }\left(  \zeta^{-1}\right)  \right)  ru_{0}-\left(
1+\varphi_{\boldsymbol{a}}^{\left(  0\right)  }\left(  \zeta^{-1}\right)
\right)  ru_{-1}\right)  =\dfrac{\Delta_{\boldsymbol{a}}\left(  \zeta\right)
}{\zeta-z}\text{.}%
\]
If $\Delta_{\boldsymbol{a}}\left(  \zeta\right)  =0$ holds for some $\zeta\in
D_{-}$, then%
\[
\left(  \zeta+\varphi_{\boldsymbol{a}}^{\left(  -1\right)  }\left(  \zeta
^{-1}\right)  \right)  u_{0}-\left(  1+\varphi_{\boldsymbol{a}}^{\left(
0\right)  }\left(  \zeta^{-1}\right)  \right)  u_{-1}=0
\]
follows. Applying $T\left(  \boldsymbol{a}\right)  $ yields%
\[
\left(  \zeta+\varphi_{\boldsymbol{a}}^{\left(  -1\right)  }\left(  \zeta
^{-1}\right)  \right)  -\left(  1+\varphi_{\boldsymbol{a}}^{\left(  0\right)
}\left(  \zeta^{-1}\right)  \right)  z^{-1}=0\text{,}%
\]
which implies%
\[
\zeta+\varphi_{\boldsymbol{a}}^{\left(  -1\right)  }\left(  \zeta^{-1}\right)
=1+\varphi_{\boldsymbol{a}}^{\left(  0\right)  }\left(  \zeta^{-1}\right)
=0\text{.}%
\]
Similarly one has $\zeta^{-1}+\varphi_{\boldsymbol{a}}^{\left(  -1\right)
}\left(  \zeta\right)  =1+\varphi_{\boldsymbol{a}}^{\left(  0\right)  }\left(
\zeta\right)  =0$. Hence $T\left(  \boldsymbol{a}\right)  ru_{0}=0$ and
$u_{0}=0$ follow. But this implies $1=T\left(  \boldsymbol{a}\right)  u_{0}%
=0$.\bigskip
\end{proof}

Assuming $z^{n}\boldsymbol{a}\in\boldsymbol{A}^{inv}\left(  C\right)  $ for
any $n\in\mathbb{Z}$ we define%
\begin{equation}
f_{n}=\boldsymbol{a}T\left(  z^{n}\boldsymbol{a}\right)  ^{-1}1\in
W_{\boldsymbol{a}}\text{.} \label{16}%
\end{equation}
Then

\begin{lemma}
\label{l4}$1+\varphi_{z^{n}\boldsymbol{a}}^{\left(  0\right)  }\left(
0\right)  \neq0$ is valid for any $n\in\mathbb{Z}$ and $\left\{
f_{n}\right\}  _{n\in\mathbb{Z}}$ satisfies a recurrence relation
\[
\dfrac{1+\varphi_{z^{n}\boldsymbol{a}}^{\left(  0\right)  }(0)}{1+\varphi
_{z^{n+1}\boldsymbol{a}}^{\left(  0\right)  }\left(  0\right)  }%
f_{n+1}+\left(  \varphi_{\widetilde{z^{n}\boldsymbol{a}}}^{\left(  -1\right)
}\left(  0\right)  -\varphi_{\widetilde{z^{n-1}\boldsymbol{a}}}^{\left(
-1\right)  }\left(  0\right)  \right)  f_{n}+f_{n-1}=\left(  z+z^{-1}\right)
f_{n}\text{.}%
\]

\end{lemma}

\begin{proof}
Observe%
\[
z^{n}f_{n}=z^{n}\boldsymbol{a}T\left(  z^{n}\boldsymbol{a}\right)
^{-1}1=1+\varphi_{z^{n}\boldsymbol{a}}^{\left(  0\right)  }\text{.}%
\]
Moreover $z^{n}f_{n+1}$, $z^{n}f_{n-1}\in W_{z^{n}\boldsymbol{a}}$ have
decompositions into $H_{\pm}$:%
\begin{align*}
z^{n}f_{n+1}  &  =z^{-1}\left(  1+\varphi_{z^{n+1}\boldsymbol{a}}^{\left(
0\right)  }\left(  0\right)  \right)  +z^{-1}\left(  \varphi_{z^{n+1}%
\boldsymbol{a}}^{\left(  0\right)  }-\varphi_{z^{n+1}\boldsymbol{a}}^{\left(
0\right)  }\left(  0\right)  \right) \\
z^{n}f_{n-1}  &  =\left(  z+\lim_{z\rightarrow\infty}z\varphi_{z^{n-1}%
\boldsymbol{a}}^{\left(  0\right)  }\left(  z\right)  \right)  +\left(
z\varphi_{z^{n-1}\boldsymbol{a}}^{\left(  0\right)  }-\lim_{z\rightarrow
\infty}z\varphi_{z^{n-1}\boldsymbol{a}}^{\left(  0\right)  }\left(  z\right)
\right) \\
&  =\left(  z+\varphi_{\widetilde{z^{n-1}\boldsymbol{a}}}^{\left(  -1\right)
}\left(  0\right)  \right)  +\left(  z\varphi_{z^{n-1}\boldsymbol{a}}^{\left(
0\right)  }-\varphi_{\widetilde{z^{n-1}\boldsymbol{a}}}^{\left(  -1\right)
}\left(  0\right)  \right)  \text{.}%
\end{align*}
Therefore, we have identities:%
\begin{equation}
\left\{
\begin{array}
[c]{l}%
z^{n}f_{n}=1+\varphi_{z^{n}\boldsymbol{a}}^{\left(  0\right)  }\\
z^{n}f_{n+1}=\left(  1+\varphi_{z^{n+1}\boldsymbol{a}}^{\left(  0\right)
}\left(  0\right)  \right)  \left(  z^{-1}+\varphi_{z^{n}\boldsymbol{a}%
}^{\left(  -1\right)  }\right) \\
z^{n}f_{n-1}=z+\varphi_{z^{n}\boldsymbol{a}}^{\left(  1\right)  }%
+\varphi_{\widetilde{z^{n-1}\boldsymbol{a}}}^{\left(  -1\right)  }\left(
0\right)  \left(  1+\varphi_{z^{n}\boldsymbol{a}}^{\left(  0\right)  }\right)
\end{array}
\text{.}\right.  \label{17}%
\end{equation}
If $1+\varphi_{z^{n+1}\boldsymbol{a}}^{\left(  0\right)  }\left(  0\right)
=0$ for some $n\in\mathbb{Z}$, then the second identity implies%
\[
1+\varphi_{z^{n+1}\boldsymbol{a}}^{\left(  0\right)  }=f_{n+1}=0
\]
identically, which shows $\varphi_{z^{n+1}\boldsymbol{a}}^{\left(  0\right)
}=-1$. This is impossible since $\varphi_{z^{n+1}\boldsymbol{a}}^{\left(
0\right)  }\in H_{-}$ and $-1\in H_{+}$, hence $1+\varphi_{z^{n+1}%
\boldsymbol{a}}^{\left(  0\right)  }\left(  0\right)  \neq0$ holds for any
$n\in\mathbb{Z}$. Applying (ii) of Lemma \ref{l2} for $n=0$ and $z^{n}%
\boldsymbol{a}$ one has%
\[
z+\varphi_{z^{n}\boldsymbol{a}}^{\left(  1\right)  }=\left(  z+z^{-1}%
-\varphi_{\widetilde{z^{n}\boldsymbol{a}}}^{\left(  -1\right)  }\left(
0\right)  \right)  \left(  1+\varphi_{z^{n}\boldsymbol{a}}^{\left(  0\right)
}\right)  -\left(  1+\varphi_{z^{n}\boldsymbol{a}}^{\left(  0\right)
}(0)\right)  \left(  z^{-1}+\varphi_{z^{n}\boldsymbol{a}}^{\left(  -1\right)
}\right)  \text{.}%
\]
This combined with (\ref{17}) completes the proof.
\end{proof}

\section{A subclass $\boldsymbol{M}\left(  C\right)  $ of $\boldsymbol{A}%
^{inv}\left(  C\right)  $}

The invertibility of $T\left(  \boldsymbol{a}\right)  $ is crucial since it is
equivalent to the non-existence of singularities of the flow. In this section
we introduce a subclass $\boldsymbol{M}\left(  C\right)  $ of $\boldsymbol{A}%
^{inv}\left(  C\right)  $.

Let $\boldsymbol{M}\left(  C\right)  $ be the set of all symbols
$\boldsymbol{m}=(m_{1},m_{2})$ satisfying%
\begin{equation}%
\begin{array}
[c]{ll}%
\text{(i)} & m_{j}\text{ is analytic in a neighborhood of }\overline{D}%
_{-}\text{, and }m_{j}=\overline{m}_{j}\text{ for }j=1\text{, }2\text{,}\\
\text{(ii)} & m_{1}\left(  0\right)  =\widetilde{m}_{1}\left(  0\right)
=1\text{, \ }m_{2}(0)=\widetilde{m}_{2}(0)=0,\\
\text{(iii)} & m_{1}\left(  z\right)  \widetilde{m}_{1}\left(  z\right)
-m_{2}\left(  z\right)  \widetilde{m}_{2}\left(  z\right)  \neq0\text{ on
}\overline{D}_{-}\text{,}%
\end{array}
\label{18-1}%
\end{equation}
where $\widetilde{f}\left(  z\right)  =f\left(  z^{-1}\right)  $. For
$\boldsymbol{m}=\left(  m_{1},m_{2}\right)  $, \ $\boldsymbol{n}=\left(
n_{1},n_{2}\right)  \in\boldsymbol{M}\left(  C\right)  $ define%
\[
\boldsymbol{m}\cdot\boldsymbol{n}=\left(  m_{1}n_{1}+m_{2}\widetilde{n}%
_{2},m_{1}n_{2}+m_{2}\widetilde{n}_{1}\right)  \text{.}%
\]
Since we have%
\begin{align*}
&  \left(  \boldsymbol{m}\cdot\boldsymbol{n}\right)  _{1}\left(  z\right)
\left(  \widetilde{\boldsymbol{m}\cdot\boldsymbol{n}}\right)  _{1}\left(
z\right)  -\left(  \boldsymbol{m}\cdot\boldsymbol{n}\right)  _{2}\left(
z\right)  \left(  \widetilde{\boldsymbol{m}\cdot\boldsymbol{n}}\right)
_{2}\left(  z\right) \\
&  =\left(  m_{1}\left(  z\right)  \widetilde{m}_{1}-m_{2}\widetilde{m}%
_{2}\right)  \left(  n_{1}\widetilde{n}_{1}-n_{2}\widetilde{n}_{2}\right)
\neq0\text{,}%
\end{align*}
clearly $\boldsymbol{m}\cdot\boldsymbol{n}\in\boldsymbol{M}\left(  C\right)  $
is valid. One has

\begin{lemma}
\label{l5}$\boldsymbol{M}\left(  C\right)  $ is a group by the operation
$\boldsymbol{m}\cdot\boldsymbol{n}$, whose identity $\boldsymbol{1}$ is
$\left(  1,0\right)  $ and inverse is%
\[
\boldsymbol{m}^{-1}=\left(  \dfrac{\widetilde{m}_{1}}{m_{1}\widetilde{m}%
_{1}-m_{2}\widetilde{m}_{2}},\dfrac{-m_{2}}{m_{1}\widetilde{m}_{1}%
-m_{2}\widetilde{m}_{2}}\right)  \text{.}%
\]
For $\boldsymbol{m}$, $\boldsymbol{n}\in\boldsymbol{M}\left(  C\right)  $ it
holds that%
\begin{equation}
T\left(  \boldsymbol{m}\cdot\boldsymbol{n}\right)  =T\left(  \boldsymbol{m}%
\right)  T\left(  \boldsymbol{n}\right)  \text{,} \label{18}%
\end{equation}
and hence $\boldsymbol{M}\left(  C\right)  \subset\boldsymbol{A}^{inv}\left(
C\right)  $.
\end{lemma}

\begin{proof}
The associative law and the form of the inverse are clear. The identity
(\ref{18}) is verified as follows. Let $\boldsymbol{m}=\left(  m_{1}%
,m_{2}\right)  $, \ $\boldsymbol{n}=\left(  n_{1},n_{2}\right)  \in
\boldsymbol{M}\left(  C\right)  $. Then, $n_{1}H_{-}$, $n_{2}H_{-}\subset
H_{-}$ are valid, hence, for $u\in H_{+}$%
\[
\mathfrak{p}_{+}\left(  m_{i}n_{j}u\right)  =\mathfrak{p}_{+}\left(
m_{i}\mathfrak{p}_{+}n_{j}u\right)  +\mathfrak{p}_{+}\left(  m_{i}%
\mathfrak{p}_{-}n_{j}u\right)  =\mathfrak{p}_{+}\left(  m_{i}\mathfrak{p}%
_{+}n_{j}u\right)
\]
holds for $i$, $j=1$, $2$, which together with $\mathfrak{p}_{+}%
R=R\mathfrak{p}_{+}$ implies%
\begin{align*}
T\left(  \boldsymbol{m}\right)  T\left(  \boldsymbol{n}\right)  u  &
=\mathfrak{p}_{+}\left(  m_{1}\mathfrak{p}_{+}\left(  n_{1}u+n_{2}Ru\right)
+m_{2}R\mathfrak{p}_{+}\left(  n_{1}u+n_{2}Ru\right)  \right) \\
&  =\mathfrak{p}_{+}\left(  m_{1}\left(  n_{1}u+n_{2}Ru\right)  \right)
+\mathfrak{p}_{+}\left(  m_{2}R\left(  n_{1}u+n_{2}Ru\right)  \right) \\
&  =\mathfrak{p}_{+}\left(  \left(  m_{1}n_{1}+m_{2}\widetilde{n}_{2}\right)
u+\left(  m_{1}n_{2}+m_{2}\widetilde{n}_{1}\right)  Ru\right)  =T\left(
\boldsymbol{m}\cdot\boldsymbol{n}\right)  u\text{,}%
\end{align*}
which shows (\ref{18}).\bigskip
\end{proof}

$\left\{  \varphi_{\boldsymbol{m}}^{\left(  -1\right)  },\varphi
_{\boldsymbol{m}}^{\left(  0\right)  }\right\}  $ is computable for
$\boldsymbol{m}\in\boldsymbol{M}\left(  C\right)  $.

\begin{lemma}
\label{l6}For $\boldsymbol{m}=\left(  m_{1},m_{2}\right)  \in\boldsymbol{M}%
\left(  C\right)  $ it holds that%
\[
z^{-1}+\varphi_{\boldsymbol{m}}^{\left(  -1\right)  }=z^{-1}m_{1}+m_{2}\text{,
\ \ }1+\varphi_{\boldsymbol{m}}^{\left(  0\right)  }=m_{1}+m_{2}z^{-1}\text{.}%
\]

\end{lemma}

\begin{proof}
Generally, for $\boldsymbol{m}=\left(  m_{1},m_{2}\right)  \in\boldsymbol{M}%
\left(  C\right)  $ one has%
\[
\left\{
\begin{array}
[c]{l}%
T\left(  \boldsymbol{m}\right)  1=\mathfrak{p}_{+}\left(  m_{1}+m_{2}%
z^{-1}\right)  =1+m_{2}\left(  0\right)  z^{-1}=1\\
T\left(  \boldsymbol{m}\right)  z^{-1}=\mathfrak{p}_{+}\left(  m_{1}%
z^{-1}+m_{2}\right)  =m_{1}\left(  0\right)  z^{-1}=z^{-1}%
\end{array}
\right.  \text{,}%
\]
which implies%
\[
T\left(  \boldsymbol{m}\right)  ^{-1}1=1\text{, \ }T\left(  \boldsymbol{m}%
\right)  ^{-1}z^{-1}=z^{-1}\text{.}%
\]
Multiplying $\boldsymbol{m}$ to the both sides, one has the conclusion.
\end{proof}

\section{Tau function}

Tau function was a key notion in Sato theory. In this article this quantity is
used to examine $g\boldsymbol{a}\in\boldsymbol{A}^{inv}\left(  C\right)  $ for
$g\in\Gamma$ and to express the Toda flow.

Recall that the domain $D_{+}$ is chosen suitably so that $D_{+}$ contains
neither poles nor zeros of $g\in\Gamma$. We have a new symbol $g\boldsymbol{a}%
$ for $g\in\Gamma$, $\boldsymbol{a}\in\boldsymbol{A}^{inv}\left(  C\right)  $.
However, we do not know if $g\boldsymbol{a}\in\boldsymbol{A}^{inv}\left(
C\right)  $ is valid or not, namely the invertibility of $T\left(
g\boldsymbol{a}\right)  $ does not always hold, which is certified by a
Fredholm determinant $\det\left(  g^{-1}T\left(  g\boldsymbol{a}\right)
T\left(  \boldsymbol{a}\right)  ^{-1}\right)  $ if it is well-defined. For
this purpose we define%
\[
\left\{
\begin{array}
[c]{l}%
S_{\boldsymbol{a}}u=\mathfrak{p}_{-}\left(  \boldsymbol{a}u\right)  \text{
\ \ for \ }u\in H_{+}\\
H_{g}v=\mathfrak{p}_{+}\left(  gv\right)  \text{ \ \ for \ }v\in H_{-}%
\end{array}
\right.  \text{.}%
\]
Note for $v\in H_{-}$ and $z\in D_{+}$%
\begin{equation}
H_{g}v\left(  z\right)  =\dfrac{1}{2\pi i}\int_{C}\dfrac{g\left(
\lambda\right)  v\left(  \lambda\right)  }{\lambda-z}d\lambda=\dfrac{1}{2\pi
i}\int_{C}\dfrac{\left(  g\left(  \lambda\right)  -g\left(  z\right)  \right)
v\left(  \lambda\right)  }{\lambda-z}d\lambda\text{,} \label{19}%
\end{equation}
which means $H_{g}$ has a smooth kernel.

\begin{lemma}
\label{l7}For any bounded symbol and $g\in\Gamma$ one has%
\[
T\left(  g\boldsymbol{a}\right)  =gT\left(  \boldsymbol{a}\right)
+H_{g}S_{\boldsymbol{a}}\text{,}%
\]
and $H_{g}S_{\boldsymbol{a}}$ is of trace class.
\end{lemma}

\begin{proof}
Observe $gu\in H_{+}$ for $u\in H_{+}$. Then, for $u\in H_{+}$ one has
identities%
\begin{align*}
T\left(  g\boldsymbol{a}\right)  u  &  =\mathfrak{p}_{+}\left(
g\boldsymbol{a}u\right)  =\mathfrak{p}_{+}\left(  g\mathfrak{p}_{+}%
\boldsymbol{a}u\right)  +\mathfrak{p}_{+}\left(  g\mathfrak{p}_{-}%
\boldsymbol{a}u\right) \\
&  =g\mathfrak{p}_{+}\boldsymbol{a}u+\mathfrak{p}_{+}\left(  g\mathfrak{p}%
_{-}\boldsymbol{a}u\right)  =gT\left(  \boldsymbol{a}\right)  u+H_{g}%
S_{\boldsymbol{a}}u\text{,}%
\end{align*}
which shows the identity. Let $\Delta$ be the Laplacian $C$. In the identity
\[
H_{g}=\left(  I-\Delta\right)  ^{-1}\left(  I-\Delta\right)  H_{g}%
\]
$\left(  I-\Delta\right)  ^{-1}$ is a trace class operator on $L^{2}\left(
C\right)  $ and $\left(  I-\Delta\right)  H_{g}$ is a bounded operator from
$H_{-}$ to $H_{+}$ due to (\ref{19}), the operator $H_{g}$ turns to be of
trace class. Since $S_{\boldsymbol{a}}$ is a bounded operator from $H_{+}$ to
$H_{-}$, we have the trace class property of $H_{g}S_{\boldsymbol{a}}%
$.\bigskip
\end{proof}

Lemma \ref{l7} implies for $\boldsymbol{a}\in\boldsymbol{A}^{inv}\left(
C\right)  $%
\[
g^{-1}T\left(  g\boldsymbol{a}\right)  T\left(  \boldsymbol{a}\right)
^{-1}=I+g^{-1}H_{g}S_{\boldsymbol{a}}T\left(  \boldsymbol{a}\right)  ^{-1}%
\]
and the second term is of trace class. Then tau-function can be defined by%
\[
\tau_{\boldsymbol{a}}\left(  g\right)  =\det\left(  g^{-1}T\left(
g\boldsymbol{a}\right)  T\left(  \boldsymbol{a}\right)  ^{-1}\right)  \text{.}%
\]
Tau-functions satisfy several properties.

\begin{lemma}
\label{l8}For $\boldsymbol{a}\in\boldsymbol{A}^{inv}\left(  C\right)  $ and
$g$, $g_{1}$, $g_{2}\in\Gamma$ it holds that\newline(i) $\ g\boldsymbol{a}%
\in\boldsymbol{A}^{inv}\left(  C\right)  $ holds if and only if $\tau
_{\boldsymbol{a}}\left(  g\right)  \neq0$.\newline(ii) $\tau_{\boldsymbol{a}%
}\left(  g\right)  =\tau_{\widetilde{\boldsymbol{a}}}\left(  \widetilde
{g}\right)  $ (note $\widetilde{\boldsymbol{a}}\in\boldsymbol{A}^{inv}\left(
C\right)  $ due to Lemma \ref{l1})\newline(iii) If $\tau_{\boldsymbol{a}%
}\left(  g_{1}\right)  \neq0$, it holds that%
\[
\tau_{\boldsymbol{a}}\left(  g_{1}g_{2}\right)  =\tau_{\boldsymbol{a}}\left(
g_{1}\right)  \tau_{g_{1}\boldsymbol{a}}\left(  g_{2}\right)  \text{ (cocycle
property).}%
\]
\ \newline(iv) If $g_{1}$ satisfies $g_{1}\left(  z\right)  =g_{1}\left(
z^{-1}\right)  $, then one has $\tau_{\boldsymbol{a}}\left(  g_{1}\right)
\neq0$ and%
\[
\tau_{\boldsymbol{a}}\left(  g_{1}g_{2}\right)  =\tau_{\boldsymbol{a}}\left(
g_{1}\right)  \tau_{\boldsymbol{a}}\left(  g_{2}\right)  \text{.}%
\]

\end{lemma}

\begin{proof}
(i) follows from the property of Fredholm determinant (see \cite{si}). Lemma
\ref{l1} immediately implies (ii). To show (iii) assume $\tau_{\boldsymbol{a}%
}\left(  g_{1}\right)  \neq0$. Then, (i) implies that $T\left(  g_{1}%
\boldsymbol{a}\right)  $ is invertible. Then the properties of determinant
show%
\begin{align*}
\tau_{\boldsymbol{a}}\left(  g_{1}g_{2}\right)   &  =\det\left(  g_{2}%
^{-1}g_{1}^{-1}T\left(  g_{1}g_{2}\boldsymbol{a}\right)  T\left(
\boldsymbol{a}\right)  ^{-1}\right) \\
&  =\det\left(  g_{1}^{-1}g_{2}^{-1}T\left(  g_{2}g_{1}\boldsymbol{a}\right)
T\left(  g_{1}\boldsymbol{a}\right)  ^{-1}g_{1}\left(  g_{1}^{-1}T\left(
g_{1}\boldsymbol{a}\right)  T\left(  \boldsymbol{a}\right)  ^{-1}\right)
\right) \\
&  =\det\left(  g_{2}^{-1}T\left(  g_{2}g_{1}\boldsymbol{a}\right)  T\left(
g_{1}\boldsymbol{a}\right)  ^{-1}\right)  \tau_{\boldsymbol{a}}\left(
g_{1}\right)  =\tau_{\boldsymbol{a}}\left(  g_{1}\right)  \tau_{g_{1}%
\boldsymbol{a}}\left(  g_{2}\right)  \text{,}%
\end{align*}
which is (iii). Suppose $g_{1}\left(  z\right)  =g_{1}\left(  z^{-1}\right)  $
holds. Then, the identity $g_{1}\boldsymbol{a}u=\boldsymbol{a}g_{1}u$ implies%
\[
T\left(  g_{1}\boldsymbol{a}\right)  =T\left(  \boldsymbol{a}\right)
g_{1}\text{, \ }T\left(  g_{2}g_{1}\boldsymbol{a}\right)  =T\left(
g_{2}\boldsymbol{a}\right)  g_{1}\text{,}%
\]
which shows $g_{1}\boldsymbol{a}\in\boldsymbol{A}^{inv}\left(  C\right)  $ and%
\[
\tau_{g_{1}\boldsymbol{a}}\left(  g_{2}\right)  =\det\left(  g_{2}%
^{-1}T\left(  g_{2}g_{1}\boldsymbol{a}\right)  T\left(  g_{1}\boldsymbol{a}%
\right)  ^{-1}\right)  =\det\left(  g_{2}^{-1}T\left(  g_{2}\boldsymbol{a}%
\right)  T\left(  \boldsymbol{a}\right)  ^{-1}\right)  =\tau_{\boldsymbol{a}%
}\left(  g_{2}\right)  \text{.}%
\]
This together with (iii) implies (iv).\bigskip
\end{proof}

$\tau_{\boldsymbol{a}}\left(  r\right)  $ for rational functions $r\in\Gamma$
can be expressed by $\left\{  \varphi_{\boldsymbol{a}}^{\left(  0\right)
}\text{, }\varphi_{\boldsymbol{a}}^{\left(  -1\right)  }\right\}  $. Here, we
compute $\tau_{\boldsymbol{a}}\left(  r\right)  $ in 2 simple cases. For
$\zeta\in D_{-}$ set%
\[
q_{\zeta}\left(  z\right)  =\left(  1-\zeta^{-1}z\right)  ^{-1}\in
\Gamma\text{.}%
\]
Then, for $\boldsymbol{a}\in\boldsymbol{A}^{inv}\left(  C\right)  $ one has

\begin{lemma}
\label{l9}(i) \ $\tau_{\boldsymbol{a}}\left(  q_{\zeta}\right)  =1+\varphi
_{\boldsymbol{a}}^{\left(  0\right)  }\left(  \zeta\right)  $\newline(ii)
$\tau_{\boldsymbol{a}}\left(  q_{\zeta_{1}}q_{\zeta_{2}}\right)
=\dfrac{\left(  \zeta_{1}+\varphi_{\boldsymbol{a}}^{\left(  1\right)  }\left(
\zeta_{1}\right)  \right)  \left(  1+\varphi_{\boldsymbol{a}}^{\left(
0\right)  }\left(  \zeta_{2}\right)  \right)  -\left(  1+\varphi
_{\boldsymbol{a}}^{\left(  0\right)  }\left(  \zeta_{1}\right)  \right)
\left(  \zeta_{2}+\varphi_{\boldsymbol{a}}^{\left(  1\right)  }\left(
\zeta_{2}\right)  \right)  }{\zeta_{1}-\zeta_{2}}$.
\end{lemma}

\begin{proof}
Recall the identity
\[
g^{-1}T\left(  g\boldsymbol{a}\right)  T(\boldsymbol{a})^{-1}-I=g^{-1}%
H_{g}\mathfrak{p}_{-}\boldsymbol{a}T(\boldsymbol{a})^{-1}%
\]
of Lemma \ref{l7}. For $g=q_{\zeta}$ and $v\in H_{-}$ one has%
\begin{align*}
q_{\zeta}^{-1}\left(  H_{q_{\zeta}}v\right)  \left(  z\right)   &  =q_{\zeta
}(z)^{-1}\dfrac{1}{2\pi i}\int_{C}\dfrac{\zeta}{\left(  \lambda-z\right)
\left(  \zeta-\lambda\right)  }v\left(  \lambda\right)  d\lambda\\
&  =q_{\zeta}(z)^{-1}\dfrac{\zeta}{\zeta-z}v\left(  \zeta\right)  =v\left(
\zeta\right)  \text{,}%
\end{align*}
hence%
\[
q_{\zeta}^{-1}T\left(  q_{\zeta}\boldsymbol{a}\right)  T(\boldsymbol{a}%
)^{-1}u-u=q_{\zeta}^{-1}\left(  H_{q_{\zeta}}\mathfrak{p}_{-}\boldsymbol{a}%
T(\boldsymbol{a})^{-1}u\right)  =\left(  \mathfrak{p}_{-}\boldsymbol{a}%
T(\boldsymbol{a})^{-1}u\right)  \left(  \zeta\right)
\]
holds for any $u\in H_{+}$, which shows $q_{\zeta}^{-1}T\left(  q_{\zeta
}\boldsymbol{a}\right)  T(\boldsymbol{a})^{-1}-I$ is a rank $1$ operator with
image constant times $1$. Since $\mathfrak{p}_{-}\boldsymbol{a}%
T(\boldsymbol{a})^{-1}1=\varphi_{\boldsymbol{a}}^{\left(  0\right)  }$, one
has (i).

For $g=q_{\zeta_{1}}q_{\zeta_{2}}$ note%
\[
\left(  q_{\zeta_{1}}q_{\zeta_{2}}\right)  \left(  z\right)  =q_{\zeta_{2}%
}\left(  \zeta_{1}\right)  q_{\zeta_{1}}\left(  z\right)  +q_{\zeta_{1}%
}\left(  _{\zeta_{2}}\right)  q_{\zeta_{2}}\left(  z\right)  \text{.}%
\]
Then, one has%
\[
H_{q_{\zeta_{1}}q_{\zeta_{2}}}v=q_{\zeta_{2}}\left(  \zeta_{1}\right)
H_{q_{\zeta_{1}}}v+q_{\zeta_{1}}\left(  _{\zeta_{2}}\right)  H_{q_{\zeta_{2}}%
}v=q_{\zeta_{2}}\left(  \zeta_{1}\right)  v\left(  \zeta_{1}\right)
q_{\zeta_{1}}+q_{\zeta_{1}}\left(  \zeta_{2}\right)  v\left(  \zeta
_{2}\right)  q_{\zeta_{2}}\text{,}%
\]
and $\left(  q_{\zeta_{1}}q_{\zeta_{2}}\right)  ^{-1}T\left(  q_{\zeta_{1}%
}q_{\zeta_{2}}\boldsymbol{a}\right)  T(\boldsymbol{a})^{-1}-I$ is a rank 2
operator with image spanned by $\left\{  q_{\zeta_{1}}^{-1},q_{\zeta_{2}}%
^{-1}\right\}  $. Since%
\[
\left\{
\begin{array}
[c]{c}%
v_{1}\equiv\mathfrak{p}_{-}\boldsymbol{a}T(\boldsymbol{a})^{-1}q_{\zeta_{1}%
}^{-1}=\varphi_{\boldsymbol{a}}^{\left(  0\right)  }-\zeta_{1}^{-1}%
\varphi_{\boldsymbol{a}}^{\left(  1\right)  }\\
v_{2}\equiv\mathfrak{p}_{-}\boldsymbol{a}T(\boldsymbol{a})^{-1}q_{\zeta_{2}%
}^{-1}=\varphi_{\boldsymbol{a}}^{\left(  0\right)  }-\zeta_{2}^{-1}%
\varphi_{\boldsymbol{a}}^{\left(  1\right)  }%
\end{array}
\right.  \text{,}%
\]
setting $a_{j}=1+\varphi_{\boldsymbol{a}}^{\left(  0\right)  }\left(
\zeta_{j}\right)  $, $b_{j}=\zeta_{j}+\varphi_{\boldsymbol{a}}^{\left(
1\right)  }\left(  \zeta_{j}\right)  $ for $j=1$, $2$, one has%
\begin{align*}
\tau_{\boldsymbol{a}}\left(  q_{\zeta_{1}}q_{\zeta_{2}}\right)   &
=\det\left(
\begin{array}
[c]{cc}%
1+q_{\zeta_{1}}\left(  \zeta_{2}\right)  v_{1}\left(  \zeta_{2}\right)  &
q_{\zeta_{1}}\left(  \zeta_{2}\right)  v_{2}\left(  \zeta_{2}\right) \\
q_{\zeta_{2}}\left(  \zeta_{1}\right)  v_{1}\left(  \zeta_{1}\right)  &
1+q_{\zeta_{2}}\left(  \zeta_{1}\right)  v_{2}\left(  \zeta_{1}\right)
\end{array}
\right) \\
&  =\det\left(
\begin{array}
[c]{cc}%
\dfrac{\zeta_{1}a_{2}-b_{2}}{\zeta_{1}-\zeta_{2}} & \dfrac{\zeta_{1}\left(
a_{2}-\zeta_{2}^{-1}b_{2}\right)  }{\zeta_{1}-\zeta_{2}}\\
\dfrac{\zeta_{2}\left(  a_{1}-\zeta_{1}^{-1}b_{1}\right)  }{\zeta_{2}%
-\zeta_{1}} & \dfrac{\zeta_{2}a_{1}-b_{1}}{\zeta_{2}-\zeta_{1}}%
\end{array}
\right)  =\dfrac{a_{2}b_{1}-a_{1}b_{2}}{\zeta_{1}-\zeta_{2}}\text{,}%
\end{align*}
which is (ii).\bigskip
\end{proof}

Later we need the continuity of $\tau_{\boldsymbol{a}}\left(  g\right)  $ with
respect to $g\in\Gamma$ since we approximate $g$ by rational functions.

\begin{lemma}
\label{l10}Assume $g_{n}$, $g\in\Gamma$ are analytic and have no zeros on a
fixed neighborhood $U$ of $\overline{D}_{+}$. If $g_{n}$ converges to $g$
uniformly on $U$, then $\tau_{\boldsymbol{a}}\left(  g_{n}\right)  $ converges
to $\tau_{\boldsymbol{a}}\left(  g\right)  $ for $\boldsymbol{a}%
\in\boldsymbol{A}^{inv}\left(  C\right)  $.
\end{lemma}

\begin{proof}
Suppose $g_{1}$, $g_{2}\in\Gamma$ are analytic and have no zeros on $U$. Then
trace class operators $A_{j}=g_{j}^{-1}H_{g_{j}}S_{\boldsymbol{a}}T\left(
\boldsymbol{a}\right)  ^{-1}$ ($j=1$, $2$) satisfy (see \cite{si})
\[
\left\vert \det\left(  I+A_{1}\right)  -\det\left(  I+A_{2}\right)
\right\vert \leq\left\Vert A_{1}-A_{2}\right\Vert _{1}\exp\left(  1+\left\Vert
A_{1}\right\Vert _{1}+\left\Vert A_{2}\right\Vert _{1}\right)
\]
with trace norm $\left\Vert \cdot\right\Vert _{1}$. Observe%
\begin{align*}
\left\Vert A_{1}-A_{2}\right\Vert _{1}  &  \leq\left\Vert g_{1}^{-1}H_{g_{1}%
}-g_{2}^{-1}H_{g_{2}}\right\Vert _{1}\left\Vert S_{\boldsymbol{a}}T\left(
\boldsymbol{a}\right)  ^{-1}\right\Vert \\
&  \leq c_{1}\left(  \left\Vert g_{1}^{-1}-g_{2}^{-1}\right\Vert \left\Vert
H_{g_{1}}\right\Vert _{1}+\left\Vert g_{2}^{-1}\right\Vert \left\Vert
H_{g_{1}}-H_{g_{2}}\right\Vert _{1}\right)
\end{align*}
hold with operator norm $\left\Vert \cdot\right\Vert $ and $c_{1}=\left\Vert
S_{\boldsymbol{a}}T\left(  \boldsymbol{a}\right)  ^{-1}\right\Vert $. Since we
have estimates%
\[
\left\{
\begin{array}
[c]{l}%
\left\Vert g_{1}^{-1}-g_{2}^{-1}\right\Vert \leq\left\Vert g_{1}^{-1}%
g_{2}^{-1}\right\Vert _{U}\left\Vert g_{1}-g_{2}\right\Vert _{U}\\
\left\Vert H_{g_{1}}-H_{g_{2}}\right\Vert _{1}\leq\left\Vert \left(
I-\Delta\right)  ^{-1}\right\Vert _{1}\left\Vert \left(  I-\Delta\right)
\left(  H_{g_{1}}-H_{g_{2}}\right)  \right\Vert \leq c_{2}\left\Vert
g_{1}-g_{2}\right\Vert _{U}\\
\left\Vert H_{g_{1}}\right\Vert _{1}\leq\left\Vert \left(  I-\Delta\right)
^{-1}\right\Vert _{1}\left\Vert \left(  I-\Delta\right)  H_{g_{1}}\right\Vert
\leq c_{2}\left\Vert g_{1}\right\Vert _{U}%
\end{array}
\right.
\]
with $\left\Vert g\right\Vert _{U}=\sup_{z\in U}\left\vert g\left(  z\right)
\right\vert $ and some constant $c_{2}$ depending only on $U$, the convergence
of $\tau_{\boldsymbol{a}}\left(  g_{n}\right)  \rightarrow\tau_{\boldsymbol{a}%
}\left(  g\right)  $ is clear due to $\tau_{\boldsymbol{a}}\left(
g_{n}\right)  =\det\left(  I+g_{n}^{-1}H_{g_{n}}S_{\boldsymbol{a}}T\left(
\boldsymbol{a}\right)  ^{-1}\right)  $.
\end{proof}

\section{Derivation of Toda hierarchy}

In this section we show that the symbols $\boldsymbol{A}^{inv}\left(
C\right)  $ and the group $\Gamma$ generate Jacobi operators and Toda lattice
under some conditions on symbols.

In view of Lemma \ref{l4}, assuming $z^{n}\boldsymbol{a}\in\boldsymbol{A}%
^{inv}\left(  C\right)  $ for $n\in\mathbb{Z}$, we define%
\begin{equation}
a_{n}=\dfrac{\sqrt{1+\varphi_{z^{n-1}\boldsymbol{a}}^{\left(  0\right)  }(0)}%
}{\sqrt{1+\varphi_{z^{n}\boldsymbol{a}}^{\left(  0\right)  }\left(  0\right)
}}\text{, \ }b_{n}=\varphi_{\widetilde{z^{n}\boldsymbol{a}}}^{\left(
-1\right)  }\left(  0\right)  -\varphi_{\widetilde{z^{n-1}\boldsymbol{a}}%
}^{\left(  -1\right)  }\left(  0\right)  \text{, \ }g_{n}=\dfrac
{\boldsymbol{a}T\left(  z^{n}\boldsymbol{a}\right)  ^{-1}1}{\sqrt
{1+\varphi_{z^{n}\boldsymbol{a}}^{\left(  0\right)  }(0)}}\text{,} \label{20}%
\end{equation}
where $\sqrt{\cdot}$ is taken arbitrary.

\begin{lemma}
\label{l11}The coefficients in (\ref{20}) can be expressed by tau functions as%
\begin{equation}
a_{n}^{2}==\dfrac{\tau_{\boldsymbol{a}}\left(  z^{n-2}\right)  \tau
_{\boldsymbol{a}}\left(  z^{n}\right)  }{\tau_{\boldsymbol{a}}\left(
z^{n-1}\right)  ^{2}}\text{, \ }b_{n}=\left.  \partial_{\varepsilon}\log
\dfrac{\tau_{\boldsymbol{a}}\left(  z^{n}q_{\varepsilon^{-1}}\right)  }%
{\tau_{\boldsymbol{a}}\left(  z^{n-1}q_{\varepsilon^{-1}}\right)  }\right\vert
_{\varepsilon=0}\text{,} \label{21}%
\end{equation}
and $g_{n}$ satisfies%
\begin{equation}
a_{n+1}g_{n+1}+a_{n}g_{n-1}+b_{n}g_{n}=\left(  z+z^{-1}\right)  g_{n}\text{.}
\label{22}%
\end{equation}

\end{lemma}

\begin{proof}
Equation (\ref{22}) follows from Lemma \ref{l4} without difficulty. (i) of
Lemma \ref{l9} implies (setting $\zeta=0$)%
\[
a_{n}^{2}=\dfrac{1+\varphi_{z^{n-1}\boldsymbol{a}}^{\left(  0\right)  }%
(0)}{1+\varphi_{z^{n}\boldsymbol{a}}^{\left(  0\right)  }\left(  0\right)
}=\dfrac{\tau_{z^{n-1}\boldsymbol{a}}\left(  z^{-1}\right)  }{\tau
_{z^{n}\boldsymbol{a}}\left(  z^{-1}\right)  }=\dfrac{\tau_{\boldsymbol{a}%
}\left(  z^{n-2}\right)  \tau_{\boldsymbol{a}}\left(  z^{n}\right)  }%
{\tau_{\boldsymbol{a}}\left(  z^{n-1}\right)  ^{2}}\text{.}%
\]
In the last equality (iii) of Lemma \ref{l8} was used.

(i) of Lemma \ref{l2} shows%
\begin{align*}
\varphi_{\widetilde{z^{n}\boldsymbol{a}}}^{\left(  -1\right)  }\left(
0\right)   &  =\lim_{\varepsilon\rightarrow0}\varphi_{\widetilde
{z^{n}\boldsymbol{a}}}^{\left(  -1\right)  }\left(  \varepsilon\right)
=\lim_{\varepsilon\rightarrow0}\varepsilon^{-1}\varphi_{z^{n}\boldsymbol{a}%
}^{\left(  0\right)  }\left(  \varepsilon^{-1}\right)  =\lim_{\varepsilon
\rightarrow0}\dfrac{\tau_{z^{n}\boldsymbol{a}}\left(  q_{\varepsilon^{-1}%
}\right)  -1}{\varepsilon}\\
&  =\lim_{\varepsilon\rightarrow0}\dfrac{\tau_{\boldsymbol{a}}\left(
z^{n}q_{\varepsilon^{-1}}\right)  -\tau_{\boldsymbol{a}}\left(  z^{n}\right)
}{\varepsilon\tau_{\boldsymbol{a}}\left(  z^{n}\right)  }=\left.
\partial_{\varepsilon}\log\tau_{\boldsymbol{a}}\left(  z^{n}q_{\varepsilon
^{-1}}\right)  \right\vert _{\varepsilon=0}\text{,}%
\end{align*}
which completes the proof.\bigskip
\end{proof}

The last derivative can be replaced by $\left.  \partial_{\varepsilon}\log
\tau_{\boldsymbol{a}}\left(  z^{n}e^{\varepsilon z}\right)  \right\vert
_{\varepsilon=0}$, since $q_{\varepsilon^{-1}}\left(  z\right)
-e^{\varepsilon z}=O\left(  \varepsilon^{2}\right)  $. In this way one can
derive a Jacobi equation (\ref{22}) from $z^{n}\in\Gamma$ and $\boldsymbol{a}%
\in\boldsymbol{A}^{inv}\left(  C\right)  $ under the condition $z^{n}%
\boldsymbol{a}\in\boldsymbol{A}^{inv}\left(  C\right)  $ for any
$n\in\mathbb{Z}$. Usually Jacobi equations are considered for real
coefficients $a_{n}$, $b_{n}$, and later we restrict ourselves only in the
real case.

For a polynomial $p$ let $\widehat{p}$ be the polynomial part of $p\left(
z+z^{-1}\right)  $. Assume $z^{n}e^{t\widehat{p}}\boldsymbol{a}\in
\boldsymbol{A}^{inv}\left(  C\right)  $ for any $n\in\mathbb{Z}$ and
$t\in\mathbb{R}$. Define a second order difference operator $L_{\boldsymbol{a}%
}$ by%
\[
\left(  L_{\boldsymbol{a}}v\right)  _{n}=a_{n+1}^{2}v_{n+1}+b_{n}v_{n}+v_{n-1}%
\]
with $a_{n}$, $b_{n}$ in (\ref{21}), and set%
\[
f_{n}\left(  t,z\right)  =\boldsymbol{a}T\left(  z^{n}e^{t\widehat{p}%
}\boldsymbol{a}\right)  ^{-1}1\text{.}%
\]
Then, Lemma \ref{l4} implies%
\begin{equation}
\left(  L_{e^{t\widehat{p}}\boldsymbol{a}}f\right)  =\left(  z+z^{-1}\right)
f\text{, \ hence \ }p\left(  L_{e^{t\widehat{p}}\boldsymbol{a}}\right)
f=p\left(  z+z^{-1}\right)  f\text{.} \label{23}%
\end{equation}
We decompose $p\left(  L_{\boldsymbol{a}}\right)  $ into 2 parts $p\left(
L_{\boldsymbol{a}}\right)  _{\pm}$: Define coefficients $\alpha_{in}$ by%
\[
\left(  p\left(  L_{\boldsymbol{a}}\right)  v\right)  _{n}=\sum_{-k\leq i\leq
k}\alpha_{in}v_{i+n}\text{ \ \ with }k=\text{degree of }p\text{,}%
\]
and set%
\[
\left(  p\left(  L_{\boldsymbol{a}}\right)  _{+}v\right)  _{n}=\sum_{1\leq
i\leq k}\alpha_{in}v_{i+n}\text{, \ \ }\left(  p\left(  L_{\boldsymbol{a}%
}\right)  _{-}v\right)  _{n}=\sum_{-k\leq i\leq0}\alpha_{in}v_{i+n}\text{.}%
\]

\begin{lemma}
\label{l12}$L_{e^{t\widehat{p}}\boldsymbol{a}}$ satisfies an operator equation%
\begin{equation}
\partial_{t}L_{e^{t\widehat{p}}\boldsymbol{a}}=\left[  L_{e^{t\widehat{p}%
}\boldsymbol{a}},p\left(  L_{e^{t\widehat{p}}\boldsymbol{a}}\right)
_{-}\right]  \text{.} \label{24}%
\end{equation}

\end{lemma}

\begin{proof}
Since $H_{+}$ is generated by $\left\{  z^{n}\right\}  _{n\in\mathbb{Z}}$, one
has $H_{+}=H_{+}^{1}\oplus H_{+}^{2}$ (direct sum) with%
\[
H_{+}^{1}=\text{linear span of }\left\{  z^{n}\right\}  _{n\geq0}\text{,
\ }H_{+}^{2}=\text{linear span of }\left\{  z^{n}\right\}  _{n<0}\text{.}%
\]
Observe that for any integer $k\geq0$, $f\in H_{-}$ and $z\in D_{+}$%
\[
\left(  \mathfrak{p}_{+}\left(  z^{k}f\right)  \right)  \left(  z\right)
=\dfrac{1}{2\pi i}\int_{C}\dfrac{\lambda^{k}f\left(  \lambda\right)  }%
{\lambda-z}d\lambda=\dfrac{1}{2\pi i}\int_{C_{1}}\dfrac{\lambda^{k}f\left(
\lambda\right)  }{\lambda-z}d\lambda\in H_{+}^{1}\text{,}%
\]
holds, since $\lambda^{k}f\left(  \lambda\right)  \left(  \lambda-z\right)
^{-1}$ is analytic in $D_{-}^{2}$ (inside of $C_{2}$). Then, we see
$\mathfrak{p}_{+}\left(  z^{k}f\right)  \in H_{+}^{1}$. Since%
\[
z^{n}e^{t\widehat{p}}\boldsymbol{a}f_{n}\left(  t,z\right)  =1+\varphi
_{z^{n}e^{t\widehat{p}}\boldsymbol{a}}^{\left(  0\right)  }\left(  z\right)
\text{ \ with }\varphi_{z^{n}e^{t\widehat{p}}\boldsymbol{a}}^{\left(
0\right)  }\in H_{-}\text{,}%
\]
one has%
\[
\mathfrak{p}_{+}\left(  \widehat{p}z^{n}e^{t\widehat{p}}f_{n}(t,z)\right)  \in
H_{+}^{1}\text{.}%
\]
Similarly $\mathfrak{p}_{+}\left(  \left(  p\left(  z+z^{-1}\right)
-\widehat{p}\right)  z^{n}e^{t\widehat{p}}f_{n}(t,z)\right)  \in H_{+}^{1}$
holds. On the other hand, if $k\leq n$, then%
\[
z^{n}e^{t\widehat{p}}f_{k}\left(  t,z\right)  =z^{n-k}z^{k}e^{t\widehat{p}%
}f_{k}\left(  t,z\right)  =z^{n-k}\left(  1+\varphi_{z^{k}e^{t\widehat{p}%
}\boldsymbol{a}}^{\left(  0\right)  }\left(  z\right)  \right)  \text{,}%
\]
hence%
\[
\mathfrak{p}_{+}z^{n}e^{t\widehat{p}}p\left(  L_{e^{t\widehat{p}%
}\boldsymbol{a}}\right)  _{-}f_{n}\left(  t,z\right)  \in H_{+}^{1}\text{,}%
\]
and similarly\ $\mathfrak{p}_{+}z^{n}e^{t\widehat{p}}p\left(  L_{e^{t\widehat
{p}}\boldsymbol{a}}\right)  _{+}f_{n}\left(  t,z\right)  \in H_{+}^{2}$ hold.
Since%
\[
\left(  p\left(  L_{e^{t\widehat{p}}\boldsymbol{a}}\right)  f\right)
_{n}=\widehat{p}f_{n}+\left(  p\left(  z+z^{-1}\right)  -\widehat{p}\right)
f=p\left(  L_{e^{t\widehat{p}}\boldsymbol{a}}\right)  _{-}f_{n}+p\left(
L_{e^{t\widehat{p}}\boldsymbol{a}}\right)  _{+}f_{n}\text{,}%
\]
one sees%
\[
\mathfrak{p}_{+}z^{n}\widehat{p}e^{t\widehat{p}}f_{n}=\mathfrak{p}_{+}%
z^{n}e^{t\widehat{p}}p\left(  L_{e^{ts_{k}}\boldsymbol{a}}\right)  _{-}%
f_{n}\text{.}%
\]
Noting $\mathfrak{p}_{+}z^{n}e^{t\widehat{p}}f_{n}\left(  t,z\right)  =1$
implies%
\[
0=\partial_{t}\mathfrak{p}_{+}z^{n}e^{t\widehat{p}}f_{n}\left(  t,z\right)
=\mathfrak{p}_{+}z^{n}\widehat{p}e^{t\widehat{p}}f_{n}\left(  t,z\right)
+\mathfrak{p}_{+}z^{n}e^{t\widehat{p}}\partial_{t}f_{n}\left(  t,z\right)
\text{,}%
\]
one has%
\[
\mathfrak{p}_{+}z^{n}e^{t\widehat{p}}\left(  \partial_{t}f_{n}\left(
t,z\right)  +p\left(  L_{e^{ts_{k}}\boldsymbol{a}}\right)  _{-}f_{n}\left(
t,z\right)  \right)  =0\text{.}%
\]
Since $z^{n}e^{t\widehat{p}}\left(  \partial_{t}f_{n}\left(  t,z\right)
+p\left(  L_{e^{ts_{k}}\boldsymbol{a}}\right)  _{-}f_{n}\left(  t,z\right)
\right)  \in W_{z^{n}e^{t\widehat{p}}\boldsymbol{a}}$ and $\mathfrak{p}%
_{+}:W_{z^{n}e^{t\widehat{p}}\boldsymbol{a}}\rightarrow H_{+}$ is one to one,
this implies%
\[
\partial_{t}f_{n}\left(  t,z\right)  =-p\left(  L_{e^{ts_{k}}\boldsymbol{a}%
}\right)  _{-}f_{n}\left(  t,z\right)  \text{.}%
\]
Since%
\[
\partial_{t}\left(  L_{e^{t\widehat{p}}\boldsymbol{a}}f\right)  =\left(
\partial_{t}L_{e^{t\widehat{p}}\boldsymbol{a}}\right)  f+L_{e^{t\widehat{p}%
}\boldsymbol{a}}\partial_{t}f
\]
and $\left(  L_{e^{t\widehat{p}}\boldsymbol{a}}f\right)  =\left(
z+z^{-1}\right)  f$ hold, one has%
\begin{align*}
\left(  \partial_{t}L_{e^{t\widehat{p}}\boldsymbol{a}}\right)  f  &  =\left(
z+z^{-1}\right)  \partial_{t}f-L_{e^{t\widehat{p}}\boldsymbol{a}}\partial
_{t}f\\
&  =-p\left(  L_{e^{t\widehat{p}}\boldsymbol{a}}\right)  _{-}\left(
z+z^{-1}\right)  f+L_{e^{t\widehat{p}}\boldsymbol{a}}p\left(  L_{e^{t\widehat
{p}}\boldsymbol{a}}\right)  _{-}f\\
&  =-p\left(  L_{e^{t\widehat{p}}\boldsymbol{a}}\right)  _{-}L_{e^{t\widehat
{p}}\boldsymbol{a}}f+L_{e^{t\widehat{p}}\boldsymbol{a}}p\left(
L_{e^{t\widehat{p}}\boldsymbol{a}}\right)  _{-}f=\left[  L_{e^{t\widehat{p}%
}\boldsymbol{a}},p\left(  L_{e^{t\widehat{p}}\boldsymbol{a}}\right)
_{-}\right]  f\text{.}%
\end{align*}
If we denote the coefficient of the deference operator $\partial
_{t}L_{e^{t\widehat{p}}\boldsymbol{a}}-\left[  L_{e^{t\widehat{p}%
}\boldsymbol{a}},p\left(  L_{e^{t\widehat{p}}\boldsymbol{a}}\right)
_{-}\right]  $ by $\alpha_{jk}$, then we have%
\[
\sum\nolimits_{k}\alpha_{jk}f_{k}\left(  t,z\right)  =0\text{ \ \ in
\ }W_{\boldsymbol{a}}\text{ \ for each fixed }t\text{,}%
\]
where the summation is finite for each $j$. Since Lemma \ref{l13} below
implies that $\left\{  f_{k}\left(  t,\cdot\right)  \right\}  _{k}$ is
linearly independent in $W_{\boldsymbol{a}}$, one has $\alpha_{jk}=0$, which
yields (\ref{24}).\medskip
\end{proof}

\begin{lemma}
\label{l13}For a symbol $\boldsymbol{a}$\ assume $z^{n}\boldsymbol{a}%
\in\boldsymbol{A}^{inv}\left(  C\right)  $ for any $n\in\mathbb{Z}$ and set%
\[
u_{n}=T\left(  z^{n}\boldsymbol{a}\right)  ^{-1}1\text{, \ }u^{\left(
n\right)  }=T\left(  \boldsymbol{a}\right)  ^{-1}z^{n}\in H_{+}\text{ \ for
}n\in\mathbb{Z}\text{.}%
\]
Then, there exist constants $\left\{  c_{kj}^{+}\right\}  _{1\leq j\leq k}$,
$\left\{  c_{kj}^{-}\right\}  _{-k\leq j\leq0}$ such that%
\begin{equation}
u_{k}=\left\{
\begin{array}
[c]{ll}%
\sum_{1\leq j\leq k}c_{kj}^{+}u^{\left(  -j\right)  } & \text{if }k\geq1\\
\sum_{k\leq j\leq0}c_{kj}^{-}u^{\left(  j\right)  } & \text{if }k\leq0
\end{array}
\right.  \text{with }c_{kk}^{+}\neq0\text{, }c_{kk}^{-}=1\text{,} \label{26}%
\end{equation}
which implies $\left\{  u_{n}\right\}  _{n\in\mathbb{Z}}$ and $\left\{
f_{n}=\boldsymbol{a}u_{n}\right\}  _{n\in\mathbb{Z}}$ are linearly independent.
\end{lemma}

\begin{proof}
For $k\geq1$ observe%
\begin{align*}
1  &  =T\left(  z^{k}\boldsymbol{a}\right)  u_{k}=\dfrac{1}{2\pi i}\int
_{C}\dfrac{\lambda^{k}-z^{k}}{\lambda-z}\boldsymbol{a}u_{k}\left(
\lambda\right)  d\lambda+z^{k}T\left(  \boldsymbol{a}\right)  u_{k}\\
&  =\sum_{0\leq j\leq k-1}z^{k-1-j}\dfrac{1}{2\pi i}\int_{C}\lambda
^{j}\boldsymbol{a}u_{k}\left(  \lambda\right)  d\lambda+z^{k}T\left(
\boldsymbol{a}\right)  u_{k}\text{.}%
\end{align*}
Multiplying $T\left(  \boldsymbol{a}\right)  ^{-1}z^{-k}$, one has%
\[
T\left(  \boldsymbol{a}\right)  ^{-1}z^{-k}=\sum_{0\leq j\leq k-1}\left(
\dfrac{1}{2\pi i}\int_{C}\lambda^{j}\boldsymbol{a}u_{k}\left(  \lambda\right)
d\lambda\right)  T\left(  \boldsymbol{a}\right)  ^{-1}z^{-1-j}+u_{k}\text{,}%
\]
hence%
\begin{align*}
u_{k}  &  =\left(  1-\dfrac{1}{2\pi i}\int_{C}\lambda^{k-1}\boldsymbol{a}%
u_{k}\left(  \lambda\right)  d\lambda\right)  -\sum_{1\leq j\leq k-1}\left(
\dfrac{1}{2\pi i}\int_{C}\lambda^{j}\boldsymbol{a}u_{k}\left(  \lambda\right)
d\lambda\right)  u^{\left(  -j\right)  }\\
&  =\sum_{1\leq j\leq k}c_{kj}^{+}u^{\left(  -j\right)  }\text{ \ with
\ }c_{kj}^{+}=\left\{
\begin{array}
[c]{ll}%
1-\dfrac{1}{2\pi i}%
{\displaystyle\int_{C}}
\lambda^{k-1}\boldsymbol{a}u_{k}\left(  \lambda\right)  d\lambda & \text{if
\ }j=k\\
-\dfrac{1}{2\pi i}\int_{C}\lambda^{j-1}\boldsymbol{a}u_{k}\left(
\lambda\right)  d\lambda & \text{if \ }1\leq j<k
\end{array}
\right.  \text{.}%
\end{align*}
Here%
\[
\dfrac{1}{2\pi i}%
{\displaystyle\int_{C}}
\lambda^{k-1}\boldsymbol{a}u_{k}\left(  \lambda\right)  d\lambda=\dfrac
{1}{2\pi i}\int_{C}\lambda^{-1}\left(  1+\varphi_{z^{k}\boldsymbol{a}%
}^{\left(  0\right)  }\left(  \lambda\right)  \right)  d\lambda=-\varphi
_{z^{k}\boldsymbol{a}}^{\left(  0\right)  }\left(  0\right)
\]
hold, hence%
\[
c_{kk}^{+}=1+\varphi_{z^{k}\boldsymbol{a}}^{\left(  0\right)  }\left(
0\right)  =\tau_{z^{k}\boldsymbol{a}}\left(  z^{-1}\right)  =\dfrac
{\tau_{\boldsymbol{a}}\left(  z^{k-1}\right)  }{\tau_{\boldsymbol{a}}\left(
z^{k}\right)  }\neq0\text{.}%
\]
Similarly, for $k\leq0$%
\begin{align*}
1  &  =T\left(  z^{-k}\boldsymbol{a}\right)  u_{-k}=\dfrac{1}{2\pi i}\int
_{C}\dfrac{\lambda^{-k}-z^{-k}}{\lambda-z}\boldsymbol{a}u_{-k}\left(
\lambda\right)  d\lambda+z^{-k}T\left(  \boldsymbol{a}\right)  u_{-k}\\
&  =-\sum_{0\leq j\leq k-1}\left(  \dfrac{1}{2\pi i}\int_{C}\lambda
^{j-k}\boldsymbol{a}u_{-k}\left(  \lambda\right)  d\lambda\right)
z^{-1-j}+z^{-k}T\left(  \boldsymbol{a}\right)  u_{-k}%
\end{align*}
hold, hence one has%
\[
u_{-k}=u^{\left(  k\right)  }+\sum_{0\leq j\leq k-1}\left(  \dfrac{1}{2\pi
i}\int_{C}\lambda^{j-k}\boldsymbol{a}u_{-k}\left(  \lambda\right)
d\lambda\right)  u^{\left(  k-j-1\right)  }\text{,}%
\]
which shows the second identity of (\ref{26}) by setting%
\[
c_{kk}^{-}=1\text{ and }c_{kj}^{-}=\dfrac{1}{2\pi i}\int_{C}\lambda
^{-j-1}\boldsymbol{a}u_{-k}\left(  \lambda\right)  d\lambda\text{ for }k+1\leq
j\leq0\text{.}%
\]
The linear independence of $\left\{  u^{\left(  k\right)  }\right\}
_{k\in\mathbb{Z}}$ follows from that of $\left\{  z^{k}\right\}
_{k\in\mathbb{Z}}$, hence $\left\{  u_{k}\right\}  _{k\in\mathbb{Z}}$ is also
linearly independent, since the triangular matrices $\left(  c_{kj}%
^{+}\right)  _{1\leq j\leq k}$ and $\left(  c_{kj}^{-}\right)  _{k\leq j\leq
0}$ are invertible.\bigskip
\end{proof}

Our Jacobi operator $H_{q}$ is related with $L_{\boldsymbol{a}}$ by%
\[
H_{q}=\Lambda_{\boldsymbol{a}}^{-1}L_{\boldsymbol{a}}\Lambda_{\boldsymbol{a}%
}\text{ \ with a diagonal matrix }\Lambda_{\boldsymbol{a}}=\left(  d_{n}%
^{1/2}\right)  _{n}\text{,}%
\]
where $d_{n}=1+\varphi_{z^{n}\boldsymbol{a}}^{\left(  0\right)  }\left(
0\right)  $. Hence, for $q=q(t)$ with $a_{n}$, $b_{n}$ replaced by those of
$e^{t\widehat{p}}\boldsymbol{a}$ one has from (\ref{24})%
\begin{align*}
\partial_{t}H_{q\left(  t\right)  }  &  =-\left(  \partial_{t}\Lambda
_{e^{t\widehat{p}}\boldsymbol{a}}\right)  \Lambda_{e^{t\widehat{p}%
}\boldsymbol{a}}^{-2}L_{e^{t\widehat{p}}\boldsymbol{a}}\Lambda_{e^{t\widehat
{p}}\boldsymbol{a}}+\Lambda_{e^{t\widehat{p}}\boldsymbol{a}}^{-1}\left(
\partial_{t}L_{e^{t\widehat{p}}\boldsymbol{a}}\right)  \Lambda_{e^{t\widehat
{p}}\boldsymbol{a}}+\Lambda_{e^{t\widehat{p}}\boldsymbol{a}}^{-1}%
L_{e^{t\widehat{p}}\boldsymbol{a}}\partial_{t}\Lambda_{e^{t\widehat{p}%
}\boldsymbol{a}}\\
&  =\Lambda_{e^{t\widehat{p}}\boldsymbol{a}}^{-1}\left[  L_{e^{t\widehat{p}%
}\boldsymbol{a}},\Lambda_{e^{t\widehat{p}}\boldsymbol{a}}^{-1}\partial
_{t}\Lambda_{e^{t\widehat{p}}\boldsymbol{a}}+p\left(  L_{e^{t\widehat{p}%
}\boldsymbol{a}}\right)  _{-}\right]  \Lambda_{e^{t\widehat{p}}\boldsymbol{a}%
}\\
&  =\left[  H_{q\left(  t\right)  },\Lambda_{e^{t\widehat{p}}\boldsymbol{a}%
}^{-1}\partial_{t}\Lambda_{e^{t\widehat{p}}\boldsymbol{a}}+p\left(
H_{q\left(  t\right)  }\right)  _{-}\right]  \text{.}%
\end{align*}
Let $\lambda_{n}$ be the diagonal element of $p\left(  L_{e^{t\widehat{p}%
}\boldsymbol{a}}\right)  $. Then, the identity (\ref{24}) implies that
$\left[  L_{e^{t\widehat{p}}\boldsymbol{a}},p\left(  L_{e^{t\widehat{p}%
}\boldsymbol{a}}\right)  _{-}\right]  $ takes a form%
\[
\left(  \left[  L_{e^{t\widehat{p}}\boldsymbol{a}},p\left(  L_{e^{t\widehat
{p}}\boldsymbol{a}}\right)  _{-}\right]  f\right)  _{n}=p_{n}f_{n+1}%
+q_{n}f_{n}+r_{n}f_{n-1}\text{,}%
\]
hence%
\begin{align*}
\left(  \left[  L_{e^{t\widehat{p}}\boldsymbol{a}},p\left(  L_{e^{t\widehat
{p}}\boldsymbol{a}}\right)  _{-}\right]  f\right)  _{n}  &  =\left(
L_{e^{t\widehat{p}}\boldsymbol{a}}p\left(  L_{e^{t\widehat{p}}\boldsymbol{a}%
}\right)  _{-}f\right)  _{n}-\left(  p\left(  L_{e^{t\widehat{p}%
}\boldsymbol{a}}\right)  _{-}L_{e^{t\widehat{p}}\boldsymbol{a}}f\right)
_{n}\\
&  =a_{n+1}^{2}\left(  \lambda_{n+1}-\lambda_{n}\right)  f_{n+1}+q_{n}%
f_{n}+r_{n}f_{n-1}\text{,}%
\end{align*}
which yields%
\[
\partial_{t}a_{n+1}^{2}=a_{n+1}^{2}\left(  \lambda_{n+1}-\lambda_{n}\right)
\text{.}%
\]
Since $\left(  \Lambda_{e^{t\widehat{p}}\boldsymbol{a}}^{-1}\partial
_{t}\Lambda_{e^{t\widehat{p}}\boldsymbol{a}}\right)  _{n}=\left(  \partial
_{t}\log d_{n}\right)  /2$, it holds that for $n\geq1$%
\begin{align*}
\partial_{t}\log d_{n}  &  =\partial_{t}\log d_{0}-\sum_{k=1}^{n}\partial
_{t}\log a_{k}^{2}\text{ \ \ (due to }a_{k}^{2}=d_{k-1}/d_{k}\text{)}\\
&  =\partial_{t}\log d_{0}-\sum_{k=1}^{n}\left(  \lambda_{k}-\lambda
_{k-1}\right)  =\partial_{t}\log d_{0}+\lambda_{0}-\lambda_{n}\text{.}%
\end{align*}
This identity holds also for $n\leq0$, hence we have%
\begin{align*}
\partial_{t}H_{q\left(  t\right)  }  &  =\left[  H_{q\left(  t\right)
},p\left(  H_{q\left(  t\right)  }\right)  _{-}-\left(  \lambda_{n}/2\right)
\right] \\
&  =\left[  H_{q\left(  t\right)  },p\left(  H_{q\left(  t\right)  }\right)
_{-}-p\left(  H_{q\left(  t\right)  }\right)  /2-\left(  \lambda_{n}/2\right)
\right]  =\left[  H_{q\left(  t\right)  },-p\left(  H_{q\left(  t\right)
}\right)  _{a}/2\right]  \text{,}%
\end{align*}
where $X_{a}$ is defined in (\ref{4}). Therefore, replacing $p$ by $-2p$, we have

\begin{proposition}
\label{p1}Let $p$ be a polynomial and assume $z^{n}e^{t\widehat{p}%
}\boldsymbol{a}\in\boldsymbol{A}^{inv}\left(  C\right)  $ for any
$n\in\mathbb{Z}$ and $t\in\mathbb{R}$. Let $q(t)$ be the coefficients
consisting of $\left\{  a_{n},b_{n}\right\}  $ of (\ref{21}) for the symbol
$e^{-2t\widehat{p}}\boldsymbol{a}$. Then, $H_{q\left(  t\right)  }$ satisfies
an operator equation%
\[
\partial_{t}H_{q\left(  t\right)  }=\left[  H_{q\left(  t\right)  },p\left(
H_{q\left(  t\right)  }\right)  _{a}\right]  \text{.}%
\]
Especially, if $p(z)=z$, this provides a Toda lattice:%
\begin{equation}
\left\{
\begin{array}
[c]{l}%
\partial_{t}a_{n}=a_{n}\left(  b_{n}-b_{n-1}\right) \\
\partial_{t}b_{n}=2\left(  a_{n+1}^{2}-a_{n}^{2}\right)
\end{array}
\right.  \text{.} \label{25}%
\end{equation}

\end{proposition}

\begin{proof}
We have only to verify (\ref{25}). If $p(z)=z$, then $\left(  p\left(
H_{q\left(  t\right)  }\right)  _{a}f\right)  _{n}=a_{n}f_{n-1}-a_{n+1}%
f_{n+1}$, hence%
\begin{align*}
&  \left(  \left[  H_{q\left(  t\right)  },p\left(  H_{q\left(  t\right)
}\right)  _{a}\right]  f\right)  _{n}\\
&  =a_{n+1}\left(  a_{n+1}u_{n}-a_{n+2}u_{n+2}\right)  +b_{n}\left(
a_{n}u_{n-1}-a_{n+1}u_{n+1}\right)  +a_{n}\left(  a_{n-1}u_{n-2}-a_{n}%
u_{n}\right) \\
&  +a_{n+1}\left(  a_{n+2}u_{n+2}+b_{n+1}u_{n+1}+a_{n+1}u_{n}\right)
-a_{n}\left(  a_{n}u_{n}+b_{n-1}u_{n-1}+a_{n-1}u_{n-2}\right) \\
&  =a_{n+1}\left(  b_{n+1}-b_{n}\right)  u_{n+1}+2\left(  a_{n+1}^{2}%
-a_{n}^{2}\right)  u_{n}+a_{n}\left(  b_{n}-b_{n-1}\right)  u_{n-1}\text{,}%
\end{align*}
which leads us to (\ref{25}).
\end{proof}

\section{$m$-function and non-vanishing of tau functions}

In the previous sections we have shown that symbols $\boldsymbol{a}$ generate
the Toda hierarchy under a certain non-degenerate condition on
$g\boldsymbol{a}$. However, as we will see later, the correspondence between
$\boldsymbol{a}$ and $q=\left\{  a_{n},b_{n}\right\}  _{n\in\mathbb{Z}}$ is
not one-to-one. Different symbols $\boldsymbol{a}$ may create the same $q$. An
essential quantity giving one-to one correspondence is called an $m$-function,
which is defined by%
\begin{equation}
m_{\boldsymbol{a}}\left(  z\right)  =\dfrac{z+\varphi_{\boldsymbol{a}%
}^{\left(  1\right)  }\left(  z\right)  }{1+\varphi_{\boldsymbol{a}}^{\left(
0\right)  }\left(  z\right)  }+\lim_{\zeta\rightarrow\infty}\zeta
\varphi_{\boldsymbol{a}}^{\left(  0\right)  }\left(  \zeta\right)
=\dfrac{z+\varphi_{\boldsymbol{a}}^{\left(  1\right)  }\left(  z\right)
}{1+\varphi_{\boldsymbol{a}}^{\left(  0\right)  }\left(  z\right)  }%
+\varphi_{\widetilde{\boldsymbol{a}}}^{\left(  -1\right)  }\left(  0\right)
\label{27}%
\end{equation}
for $\boldsymbol{a}\in\boldsymbol{A}^{inv}\left(  C\right)  $. The constant
term is added so that $m_{\boldsymbol{a}}$ satisfies%
\[
m_{\boldsymbol{a}}\left(  z\right)  =z+O\left(  z^{-1}\right)  \text{ \ as
}z\rightarrow\infty\text{.}%
\]
$m$-functions also play a key role to show the non-vanishing of tau functions
if we assume some non-negativity condition on tau functions.

Since $\varphi_{\boldsymbol{a}}^{\left(  0\right)  }\left(  z\right)
=O\left(  z^{-1}\right)  $ as $z\rightarrow\infty$, we see that $1+\varphi
_{\boldsymbol{a}}^{\left(  0\right)  }$ is not identically $0$ on $D_{-}%
^{1}\left(  =D_{-}\cap\left\{  \left\vert z\right\vert >1\right\}  \right)  $.
However, there is a possibility that $1+\varphi_{\boldsymbol{a}}^{\left(
0\right)  }$ is identically $0$ on $D_{-}^{2}\left(  =D_{-}\cap\left\{
\left\vert z\right\vert <1\right\}  \right)  $, which makes it impossible to
define $m_{\boldsymbol{a}}$ on $D_{-}^{2}$. In this section we investigate
conditions under which the properties $1+\varphi_{\boldsymbol{a}}^{\left(
0\right)  }\left(  z\right)  \neq0$ on $D_{-}^{2}$ and $\tau_{\boldsymbol{a}%
}\left(  g\right)  \neq0$ hold.

For simplicity of notations we introduce an auxiliary quantity%
\begin{equation}
n_{\boldsymbol{a}}\left(  z\right)  =\dfrac{z^{-1}+\varphi_{\boldsymbol{a}%
}^{\left(  -1\right)  }\left(  z\right)  }{1+\varphi_{\boldsymbol{a}}^{\left(
0\right)  }\left(  z\right)  }=\dfrac{z+z^{-1}-m_{\boldsymbol{a}}\left(
z\right)  }{1+\varphi_{\boldsymbol{a}}^{\left(  0\right)  }\left(  0\right)
}\text{ (Lemma \ref{l2}),} \label{28}%
\end{equation}
which has a tau function's expression:%
\begin{equation}
n_{\boldsymbol{a}}\left(  \zeta\right)  =\dfrac{1+\varphi_{\widetilde
{\boldsymbol{a}}}^{\left(  0\right)  }\left(  \zeta^{-1}\right)  }%
{\zeta\left(  1+\varphi_{\boldsymbol{a}}^{\left(  0\right)  }\left(
\zeta\right)  \right)  }=\dfrac{\tau_{\widetilde{\boldsymbol{a}}}\left(
q_{\zeta^{-1}}\right)  }{\zeta\tau_{\boldsymbol{a}}\left(  q_{\zeta}\right)
}=\dfrac{\tau_{\boldsymbol{a}}\left(  zq_{\zeta}\right)  }{\zeta
\tau_{\boldsymbol{a}}\left(  q_{\zeta}\right)  }\text{ \ (Lemmas \ref{l2},
\ref{l9}).} \label{28-1}%
\end{equation}

\begin{lemma}
\label{l13-1}For $\boldsymbol{a}\in\boldsymbol{A}^{inv}\left(  C\right)  $,
$\zeta\in D_{-}$ assume $\tau_{\boldsymbol{a}}\left(  q_{\zeta}\right)  $,
$\tau_{\boldsymbol{a}}\left(  z^{-1}\right)  \neq0$. Then, we have%
\begin{equation}
m_{q_{\zeta}\boldsymbol{a}}\left(  z\right)  =\left(  m_{\boldsymbol{a}%
}\left(  0\right)  -m_{\boldsymbol{a}}\left(  \zeta\right)  \right)  \left(
1-\dfrac{\phi\left(  z\right)  -\phi\left(  \zeta\right)  }{m_{\boldsymbol{a}%
}\left(  z\right)  -m_{\boldsymbol{a}}\left(  \zeta\right)  }\right)
+\phi\left(  z\right)  \text{ on }D_{-} \label{28-4}%
\end{equation}
with $\phi\left(  z\right)  =z+z^{-1}$. Especially, setting $\zeta=0$, we have%
\begin{equation}
m_{z^{-1}\boldsymbol{a}}\left(  z\right)  =\dfrac{m_{\boldsymbol{a}}^{\prime
}\left(  0\right)  }{m_{\boldsymbol{a}}\left(  0\right)  -m_{\boldsymbol{a}%
}\left(  z\right)  }+\phi\left(  z\right)  \text{ on }D_{-}\text{.}
\label{28-5}%
\end{equation}

\end{lemma}

\begin{proof}
First we compute $n_{q_{\zeta}\boldsymbol{a}}\left(  \eta\right)  $. From
Lemma \ref{l9} we have%
\begin{equation}
1+\varphi_{q_{\zeta}\boldsymbol{a}}^{\left(  0\right)  }\left(  \eta\right)
=\tau_{q_{\zeta}\boldsymbol{a}}\left(  q_{\eta}\right)  =\dfrac{\tau
_{\boldsymbol{a}}\left(  q_{\zeta}q_{\eta}\right)  }{\tau_{\boldsymbol{a}%
}\left(  q_{\zeta}\right)  }=\left(  1+\varphi_{\boldsymbol{a}}^{\left(
0\right)  }\left(  \eta\right)  \right)  \dfrac{m_{\boldsymbol{a}}\left(
\zeta\right)  -m_{\boldsymbol{a}}\left(  \eta\right)  }{\zeta-\eta}\text{.}
\label{28-2}%
\end{equation}
We compute $\eta^{-1}+\varphi_{q_{\zeta}\boldsymbol{a}}^{\left(  -1\right)
}\left(  \eta\right)  $ without using tau functions. Set $c_{1}=\lim
_{z\rightarrow\infty}z\varphi_{q_{\zeta}\boldsymbol{a}}^{\left(  -1\right)
}\left(  z\right)  $. Since%
\[
q_{\zeta}^{-1}\left(  \eta^{-1}+\varphi_{q_{\zeta}\boldsymbol{a}}^{\left(
-1\right)  }\right)  =q_{\zeta}^{-1}q_{\zeta}\boldsymbol{a}T\left(  q_{\zeta
}\boldsymbol{a}\right)  ^{-1}z^{-1}=\boldsymbol{a}T\left(  q_{\zeta
}\boldsymbol{a}\right)  ^{-1}z^{-1}\in W_{\boldsymbol{a}}%
\]
and%
\[
\mathfrak{p}_{+}\left(  q_{\zeta}^{-1}\left(  \eta^{-1}+\varphi_{q_{\zeta
}\boldsymbol{a}}^{\left(  -1\right)  }\right)  \right)  =\left(  1-\zeta
^{-1}\eta\right)  \eta^{-1}-c_{1}\zeta^{-1}%
\]
hold, we have%
\begin{equation}
q_{\zeta}^{-1}\left(  \eta^{-1}+\varphi_{q_{\zeta}\boldsymbol{a}}^{\left(
-1\right)  }\right)  =\eta^{-1}+\varphi_{\boldsymbol{a}}^{\left(  -1\right)
}-\left(  1+c_{1}\right)  \zeta^{-1}\left(  1+\varphi_{\boldsymbol{a}%
}^{\left(  0\right)  }\right)  \text{.} \label{28-3}%
\end{equation}
Substituting $\eta=\zeta$ yields%
\[
\left(  1+c_{1}\right)  \zeta^{-1}=\dfrac{\zeta^{-1}+\varphi_{\boldsymbol{a}%
}^{\left(  -1\right)  }\left(  \zeta\right)  }{1+\varphi_{\boldsymbol{a}%
}^{\left(  0\right)  }\left(  \zeta\right)  }=n_{\boldsymbol{a}}\left(
\zeta\right)  \text{,}%
\]
hence (\ref{28-3}) shows%
\[
\dfrac{\eta^{-1}+\varphi_{q_{\zeta}\boldsymbol{a}}^{\left(  -1\right)  }%
}{1+\varphi_{\boldsymbol{a}}^{\left(  0\right)  }}=\left(  n_{\boldsymbol{a}%
}-n_{\boldsymbol{a}}\left(  \zeta\right)  \right)  q_{\zeta}\text{,}%
\]
which together with (\ref{28-2}) leads us to%
\[
\dfrac{m_{\boldsymbol{a}}\left(  \zeta\right)  -m_{\boldsymbol{a}}\left(
\eta\right)  }{\zeta-\eta}n_{q_{\zeta}\boldsymbol{a}}\left(  \eta\right)
=\zeta\dfrac{n_{\boldsymbol{a}}\left(  \eta\right)  -n_{\boldsymbol{a}}\left(
\zeta\right)  }{\zeta-\eta}\Longrightarrow n_{q_{\zeta}\boldsymbol{a}}\left(
\eta\right)  =\zeta\dfrac{n_{\boldsymbol{a}}\left(  \eta\right)
-n_{\boldsymbol{a}}\left(  \zeta\right)  }{m_{\boldsymbol{a}}\left(
\zeta\right)  -m_{\boldsymbol{a}}\left(  \eta\right)  }\text{.}%
\]
Then, (\ref{28}) yields the proof of (\ref{28-4}).\bigskip
\end{proof}

Set%
\begin{equation}
\left\{
\begin{array}
[c]{l}%
\Gamma_{\operatorname{real}}\left(  C\right)  =\left\{  g\in\Gamma\text{;
}g=\overline{g}\text{,\ \ }g\text{ has no zeros nor poles on }\overline{D}%
_{+}\right\} \\
\boldsymbol{A}_{+}^{inv}\left(  C\right)  =\left\{  \boldsymbol{a}%
\in\boldsymbol{A}^{inv}\left(  C\right)  \text{; }\boldsymbol{a}%
=\overline{\boldsymbol{a}}\text{, }\tau_{\boldsymbol{a}}\left(  r\right)
\geq0\text{ for any rational }r\in\Gamma_{\operatorname{real}}\left(
C\right)  \right\}
\end{array}
\text{,}\right.  \label{29}%
\end{equation}
\ where we use the notation%
\[
\overline{g}\left(  z\right)  =\overline{g\left(  \overline{z}\right)
}\text{, \ \ }\overline{\boldsymbol{a}}\left(  \lambda\right)  =\overline
{\boldsymbol{a}\left(  \overline{\lambda}\right)  }\text{.}%
\]
The continuity of $\tau_{\boldsymbol{a}}$ implies%
\begin{equation}
\boldsymbol{A}_{+}^{inv}\left(  C\right)  =\left\{  \boldsymbol{a}%
\in\boldsymbol{A}^{inv}\left(  C\right)  \text{; }\boldsymbol{a}%
=\overline{\boldsymbol{a}}\text{, }\tau_{\boldsymbol{a}}\left(  g\right)
\geq0\text{ for any }g\in\Gamma_{\operatorname{real}}\left(  C\right)
\right\}  \text{.} \label{31}%
\end{equation}
The condition $\boldsymbol{a}=\overline{\boldsymbol{a}}$ is necessary to have
real $a_{n}$, $b_{n}$. (ii) of Lemma \ref{l8} implies%
\[
\widetilde{\boldsymbol{a}}\in\boldsymbol{A}_{+}^{inv}\left(  C\right)  \text{
holds for }\boldsymbol{a}\in\boldsymbol{A}_{+}^{inv}\left(  C\right)  \text{.}%
\]

\begin{lemma}
\label{l14}For $\boldsymbol{a}\in\boldsymbol{A}_{+}^{inv}\left(  C\right)  $
we have the followings:\newline(i) \ $1+\varphi_{\boldsymbol{a}}^{\left(
0\right)  }\left(  \zeta\right)  \neq0$, $\operatorname{Im}m_{\boldsymbol{a}%
}\left(  \zeta\right)  >0$ on $D_{-}^{1}$ hold.\newline Assume further
$1+\varphi_{\boldsymbol{a}}^{\left(  0\right)  }\left(  0\right)  $
$(=\tau_{\boldsymbol{a}}\left(  z^{-1}\right)  )>0$ and $\tau_{\boldsymbol{a}%
}\left(  z^{-2}\right)  >0$. Then we have\newline(ii) $1+\varphi
_{\boldsymbol{a}}^{\left(  0\right)  }\left(  \zeta\right)  \neq0$ on $D_{-}$,
and $m_{\boldsymbol{a}}^{\prime}\left(  0\right)  >0$, $\operatorname{Im}%
m_{\boldsymbol{a}}\left(  \zeta\right)  >0$, $\operatorname{Im}m_{z^{-1}%
\boldsymbol{a}}\left(  \zeta\right)  >\operatorname{Im}\left(  \zeta
+\zeta^{-1}\right)  $ on $D_{-}^{1}\cap\mathbb{C}_{+}$.
\end{lemma}

\begin{proof}
Set%
\[
\mathcal{Z}=\left\{  \zeta\in D_{-}\text{; \ }1+\varphi_{\boldsymbol{a}%
}^{\left(  0\right)  }\left(  \zeta\right)  =0\right\}  \text{.}%
\]
$1+\varphi_{\boldsymbol{a}}^{\left(  0\right)  }\left(  \zeta\right)  $ is not
identically $0$ on $D_{-}^{1}$, since $1+\varphi_{\boldsymbol{a}}^{\left(
0\right)  }\left(  \zeta\right)  \rightarrow1$ as $\zeta\rightarrow\infty$.
Hence $m_{\boldsymbol{a}}\left(  \zeta\right)  $ is meromorphic on $D_{-}^{1}%
$. Lemma \ref{l9} yields an identity
\begin{equation}
0\leq\tau_{\boldsymbol{a}}\left(  q_{z}q_{\overline{z}}\right)  =\left\vert
1+\varphi_{\boldsymbol{a}}^{\left(  0\right)  }\left(  z\right)  \right\vert
^{2}\dfrac{\operatorname{Im}m_{\boldsymbol{a}}\left(  z\right)  }%
{\operatorname{Im}z}\text{ for }z\in D_{-}^{1}\backslash\mathcal{Z}\text{,}
\label{30}%
\end{equation}
hence $\operatorname{Im}m_{\boldsymbol{a}}\left(  \zeta\right)  \geq0$ on
$D_{-}^{1}\cap\mathbb{C}_{+}$. Let $\zeta_{0}\in\mathcal{Z}\cap D_{-}^{1}%
\cap\mathbb{C}_{+}$. Then, $\operatorname{Im}\left(  -m_{\boldsymbol{a}%
}\left(  \zeta\right)  ^{-1}\right)  \geq0$ on a neighborhood of $\zeta_{0}$
and $\operatorname{Im}\left(  -m_{\boldsymbol{a}}\left(  \zeta_{0}\right)
^{-1}\right)  =0$ hold, hence the minimum principle for harmonic functions
implies $\operatorname{Im}\left(  -m_{\boldsymbol{a}}\left(  \zeta\right)
^{-1}\right)  =0$ identically on $D_{-}^{1}\cap\mathbb{C}_{+}$, which shows
$m_{\boldsymbol{a}}$ is a constant on $D_{-}^{1}\cap\mathbb{C}_{+}$. This
contradicts $m_{\boldsymbol{a}}\left(  \zeta\right)  -\zeta\rightarrow0$ as
$\zeta\rightarrow\infty$, hence we have $\mathcal{Z}\cap D_{-}^{1}%
\cap\mathbb{C}_{+}=\varnothing$. Since $\varphi_{\boldsymbol{a}}^{\left(
0\right)  }=\overline{\varphi}_{\boldsymbol{a}}^{\left(  0\right)  }$,
$m_{\boldsymbol{a}}=\overline{m}_{\boldsymbol{a}}$ hold, $\mathcal{Z}%
\cap\left(  D_{-}^{1}\backslash\mathbb{R}\right)  =\varnothing$ is valid.

For $x\in D_{-}^{1}\cap\mathbb{R}$ one has from (i) of (\ref{28})%
\[
0\leq\tau_{\boldsymbol{a}}\left(  q_{x}\right)  =1+\varphi_{\boldsymbol{a}%
}^{\left(  0\right)  }\left(  x\right)  \text{ }\ \text{(}q_{\zeta}\left(
z\right)  =\left(  1-\zeta^{-1}z\right)  ^{-1}\text{).}%
\]
Suppose $1+\varphi_{\boldsymbol{a}}^{\left(  0\right)  }\left(  x_{0}\right)
=0$ holds for some $x_{0}\in D_{-}^{1}\cap\mathbb{R}$. Choose an interval
$(a,b)\subset D_{-}^{1}\cap\mathbb{R}$ such that $x_{0}\in(a,b)$ and
$1+\varphi_{\boldsymbol{a}}^{\left(  0\right)  }\left(  x\right)  >0$ for any
$x\in(a,b)\backslash\left\{  x_{0}\right\}  $. (\ref{30}) implies
$m_{\boldsymbol{a}}^{\prime}\left(  x\right)  \geq0$ for $x\in(a,b)\backslash
\left\{  x_{0}\right\}  $. Hence we have%
\[
m_{\boldsymbol{a}}\left(  x_{0}-0\right)  =+\infty\Longrightarrow
x_{0}+\varphi_{\boldsymbol{a}}^{\left(  1\right)  }\left(  x_{0}\right)
>0\Longrightarrow m_{\boldsymbol{a}}\left(  x_{0}+0\right)  =+\infty\text{,}%
\]
which contradicts $m_{\boldsymbol{a}}^{\prime}\left(  x\right)  \geq0$.
Therefore, $\mathcal{Z}\cap D_{-}^{1}=\varnothing$ holds. The property
$\operatorname{Im}m_{\boldsymbol{a}}\left(  \zeta\right)  >0$ on $D_{-}%
^{1}\cap\mathbb{C}_{+}$ follows from $\operatorname{Im}m_{\boldsymbol{a}%
}\left(  \zeta\right)  \geq0$ and $m_{\boldsymbol{a}}\left(  \zeta\right)
-\zeta\rightarrow0$ as $\zeta\rightarrow\infty$, which shows (i).

To show (ii) observe an identity%
\[
\dfrac{\tau_{\boldsymbol{a}}\left(  z^{-1}q_{\zeta}\right)  }{\tau
_{\boldsymbol{a}}\left(  q_{\zeta}\right)  \tau_{\boldsymbol{a}}\left(
z^{-1}\right)  }=\dfrac{m_{\boldsymbol{a}}\left(  \zeta\right)
-m_{\boldsymbol{a}}\left(  0\right)  }{\zeta}%
\]
deduced from Lemma \ref{l9}, which shows $m_{\boldsymbol{a}}^{\prime}\left(
0\right)  =\tau_{\boldsymbol{a}}\left(  z^{-2}\right)  /\tau_{\boldsymbol{a}%
}\left(  z^{-1}\right)  ^{2}$. If $\tau_{\boldsymbol{a}}\left(  z^{-2}\right)
>0$, then $m_{\boldsymbol{a}}$ is not a constant on $D_{-}^{1}$, and similar
arguments as (i) are possible, and we have $1+\varphi_{\boldsymbol{a}%
}^{\left(  0\right)  }\left(  \zeta\right)  \neq0$ on $D_{-}$,
$\operatorname{Im}m_{\boldsymbol{a}}\left(  \zeta\right)  >0$ on $D_{-}%
^{2}\cap\mathbb{C}_{+}$. The inequality $\operatorname{Im}m_{z^{-1}%
\boldsymbol{a}}\left(  \zeta\right)  >\operatorname{Im}\left(  \zeta
+\zeta^{-1}\right)  $ on $D_{-}^{1}\cap\mathbb{C}_{+}$ follows from
(\ref{28-5}).\bigskip
\end{proof}

For $\boldsymbol{a}\in\boldsymbol{A}_{+}^{inv}\left(  C\right)  $ one can show
partially the property $\tau_{\boldsymbol{a}}\left(  g\right)  >0$. For
$\zeta\in D_{-}$ set%
\[
r_{\zeta}\left(  z\right)  =q_{\zeta}\left(  z\right)  q_{\overline{\zeta}%
}\left(  z\right)  \in\Gamma_{\operatorname{real}}\left(  C\right)  \text{.}%
\]

\begin{lemma}
\label{l15}If $\boldsymbol{a}\in\boldsymbol{A}_{+}^{inv}\left(  C\right)  $,
then $\tau_{\boldsymbol{a}}\left(  re^{h}\right)  >0$ for any rational
function $r\in\Gamma_{\operatorname{real}}\left(  C\right)  $ with no poles
nor zeros on $D_{-}^{2}$ and analytic function $h$ on $\overline{D}_{+}$
satisfying $h=\overline{h}$.
\end{lemma}

\begin{proof}
Note that any rational function $r\in\Gamma_{\operatorname{real}}\left(
C\right)  $ with no poles nor zeros on $D_{-}^{2}$ can be expressed by
$r=r_{1}r_{2}^{-1}$%
\[
\left\{
\begin{array}
[c]{l}%
r_{1}=q_{x_{1}}q_{x_{2}}\cdots q_{x_{m}}r_{\zeta_{1}}r_{\zeta_{2}}\cdots
r_{\zeta_{n}}\text{ \ with }x_{j}\in D_{-}^{1}\cap\mathbb{R}\text{, }\zeta
_{j}\in D_{-}^{1}\backslash\mathbb{R}\\
r_{2}=q_{y_{1}}q_{y_{2}}\cdots q_{y_{m^{\prime}}}r_{\eta_{1}}r_{\eta_{2}%
}\cdots r_{\eta_{n^{\prime}}}\text{ \ with }y_{j}\in D_{-}^{1}\cap
\mathbb{R}\text{, }\eta_{j}\in D_{-}^{1}\backslash\mathbb{R}%
\end{array}
\right.  \text{.}%
\]
First we show $\tau_{\boldsymbol{a}}\left(  r_{1}\right)  >0$. For any $x\in
D_{-}^{1}\cap\mathbb{R}$, $\zeta\in D_{-}^{1}\backslash\mathbb{R}$ we have%
\[
\left\{
\begin{array}
[c]{l}%
\tau_{\boldsymbol{a}}\left(  q_{x}\right)  =1+\varphi_{\boldsymbol{a}%
}^{\left(  0\right)  }\left(  x\right)  >0\\
\tau_{\boldsymbol{a}}\left(  r_{\zeta}\right)  =\dfrac{\left\vert
1+\varphi_{\boldsymbol{a}}^{\left(  0\right)  }\left(  \zeta\right)
\right\vert ^{2}\operatorname{Im}m_{\boldsymbol{a}}\left(  \zeta\right)
}{\operatorname{Im}\zeta}>0
\end{array}
\right.
\]
\ from Lemma \ref{l14}. Then, $q_{x}\boldsymbol{a}$, $r_{\zeta}\boldsymbol{a}%
\in\boldsymbol{A}_{+}^{inv}\left(  C\right)  $ hold, and the property
$\tau_{\boldsymbol{a}}\left(  r_{1}\right)  >0$ can be shown inductively. For
general $r=r_{1}r_{2}^{-1}$ Lemma \ref{l8} shows%
\begin{align*}
\tau_{\boldsymbol{a}}\left(  r\right)   &  =\tau_{\boldsymbol{a}}\left(
r_{1}r_{2}^{-1}\right)  =\tau_{\boldsymbol{a}}\left(  r_{1}r_{2}%
^{-1}\widetilde{r}_{2}^{-1}\widetilde{r}_{2}\right)  =\tau_{\boldsymbol{a}%
}\left(  \left(  r_{2}\widetilde{r}_{2}\right)  ^{-1}\right)  \tau
_{\boldsymbol{a}}\left(  r_{1}\widetilde{r}_{2}\right) \\
&  =\tau_{\boldsymbol{a}}\left(  \left(  r_{2}\widetilde{r}_{2}\right)
^{-1}\right)  \tau_{\boldsymbol{a}}\left(  r_{1}\right)  \tau_{r_{1}%
\boldsymbol{a}}\left(  \widetilde{r}_{2}\right)  =\tau_{\boldsymbol{a}}\left(
\left(  r_{2}\widetilde{r}_{2}\right)  ^{-1}\right)  \tau_{\boldsymbol{a}%
}\left(  r_{1}\right)  \tau_{\widetilde{r_{1}\boldsymbol{a}}}\left(
r_{2}\right)  \text{,}%
\end{align*}
where $\tau_{\boldsymbol{a}}\left(  \left(  r_{2}\widetilde{r}_{2}\right)
^{-1}\right)  >0$ due to $\left(  r_{2}\widetilde{r}_{2}\right)  \left(
z\right)  =\left(  r_{2}\widetilde{r}_{2}\right)  \left(  z^{-1}\right)  $.
Then, $\tau_{\boldsymbol{a}}\left(  r\right)  >0$ is valid, since
$\widetilde{r_{1}\boldsymbol{a}}\in\boldsymbol{A}_{+}^{inv}\left(  C\right)  $.

On the other hand, note that for $z\in D_{+}$%
\begin{align*}
h(z)  &  =\dfrac{1}{2\pi i}\int_{C}\dfrac{h(\lambda)}{\lambda-z}%
d\lambda=\dfrac{1}{2\pi i}\int_{C_{1}}\dfrac{h(\lambda)}{\lambda-z}%
d\lambda+\dfrac{1}{2\pi i}\int_{C_{2}}\dfrac{h(\lambda)}{\lambda-z}d\lambda\\
&  \equiv h_{1}(z)+h_{2}(z)\text{,}%
\end{align*}
where $h_{1}$ is analytic on $D_{+}\cup D_{-}^{2}$ and $h_{2}$ is analytic on
$D_{+}\cup D_{-}^{1}$ (if it is necessary, we have only to move $C_{1}$ and
$C_{2}$ in the inside of a neighborhood of $\overline{D}_{+}$). We first
consider $h_{1}$. Since $\overline{D_{+}\cup D_{-}^{2}}$ is a simply connected
domain containing $0$, one sees that $h_{1}$ can be approximated by a real
polynomial $p$. Observe%
\[
e^{p}=\lim_{n\rightarrow\infty}\left(  1-\dfrac{p}{n}\right)  ^{-n}\text{
uniformly on }\overline{D_{+}\cup D_{-}^{2}}\text{,}%
\]
and%
\[
r_{1}\equiv\left(  1-\dfrac{p}{n}\right)  ^{-n}=q_{x_{1}}\cdots q_{x_{m_{1}}%
}r_{\zeta_{1}}\cdots r_{\zeta_{m_{2}}}\text{ for }\exists x_{j}\in D_{-}%
^{1}\cap\mathbb{R}\text{, }\zeta_{j}\in D_{-}^{1}\backslash\mathbb{R}\text{,}%
\]
since zeros of $1-p/n$ are located on $D_{-}^{1}$ for sufficiently large $n$.
Then, the continuity of the tau function implies $\tau_{\boldsymbol{a}}\left(
r_{1}^{-1}e^{h_{1}}\right)  >0$, hence%
\[
\tau_{\boldsymbol{a}}\left(  e^{h_{1}}\right)  =\tau_{\boldsymbol{a}}\left(
r_{1}^{-1}e^{h_{1}}r_{1}\right)  =\tau_{\boldsymbol{a}}\left(  r_{1}%
^{-1}e^{h_{1}}\right)  \tau_{r_{1}^{-1}e^{h_{1}}\boldsymbol{a}}\left(
r_{1}\right)
\]
is valid. Since $r_{1}^{-1}e^{h_{1}}\boldsymbol{a}\in\boldsymbol{A}_{+}%
^{inv}\left(  C\right)  $ due to (\ref{31}), applying the above argument to
$r_{1}^{-1}e^{h_{1}}\boldsymbol{a}$, one has $\tau_{r_{1}^{-1}e^{h_{1}%
}\boldsymbol{a}}\left(  r_{1}\right)  >0$, which shows $\tau_{\boldsymbol{a}%
}\left(  e^{h_{1}}\right)  >0$. Observe%
\[
\tau_{\boldsymbol{a}}\left(  e^{h}\right)  =\tau_{\boldsymbol{a}}\left(
e^{h_{1}}\right)  \tau_{e^{h_{1}}\boldsymbol{a}}\left(  e^{h_{2}}\right)
=\tau_{\boldsymbol{a}}\left(  e^{h_{1}}\right)  \tau_{\widetilde{e^{h_{1}%
}\boldsymbol{a}}}\left(  e^{\widetilde{h}_{2}}\right)  \text{.}%
\]
\ The property $e^{h_{1}}\boldsymbol{a}\in\boldsymbol{A}_{+}^{inv}\left(
C\right)  $ implies $\widetilde{e^{h_{1}}\boldsymbol{a}}\in\boldsymbol{A}%
_{+}^{inv}\left(  C\right)  $, hence $\tau_{\widetilde{e^{h_{1}}%
\boldsymbol{a}}}\left(  e^{\widetilde{h}_{2}}\right)  >0$ holds, which shows
$\tau_{\boldsymbol{a}}\left(  e^{h}\right)  >0$. Now $\tau_{\boldsymbol{a}%
}\left(  re^{h}\right)  >0$ clearly follows from the fact $r\boldsymbol{a}%
\in\boldsymbol{A}_{+}^{inv}\left(  C\right)  $.\bigskip
\end{proof}

$\boldsymbol{A}_{+}^{inv}\left(  C\right)  $ is not sufficient to have the
property $\tau_{\boldsymbol{a}}\left(  g\right)  >0$ for any $g\in
\Gamma_{\operatorname{real}}\left(  C\right)  $, since there is possibility
that $1+\varphi_{\boldsymbol{a}}^{\left(  0\right)  }\left(  z\right)  =0$
identically on $D_{-}^{2}$. We define%
\begin{equation}
\boldsymbol{A}_{++}^{inv}\left(  C\right)  =\left\{  \boldsymbol{a}%
\in\boldsymbol{A}_{+}^{inv}\left(  C\right)  \text{; \ }\tau_{\boldsymbol{a}%
}\left(  z^{n}\right)  >0\text{ \ for any }n\in\mathbb{Z}\right\}  \text{.}
\label{32}%
\end{equation}

\begin{proposition}
\label{p2}Suppose $\boldsymbol{a}\in\boldsymbol{A}_{++}^{inv}\left(  C\right)
$. Then, for any $g\in\Gamma_{\operatorname{real}}\left(  C\right)  $ one has
$\tau_{\boldsymbol{a}}\left(  g\right)  >0$ and $g\boldsymbol{a}%
\in\boldsymbol{A}_{++}^{inv}\left(  C\right)  $.
\end{proposition}

\begin{proof}
If $\boldsymbol{a}\in\boldsymbol{A}_{++}^{inv}\left(  C\right)  $, then
$z^{k}\boldsymbol{a}\in\boldsymbol{A}_{++}^{inv}\left(  C\right)  $\ for
$k\in\mathbb{Z}$ is valid. For $x\in D_{-}\cap\mathbb{R}$ Lemma \ref{l14}
implies $\tau_{\boldsymbol{a}}\left(  q_{x}\right)  =1+\varphi_{\boldsymbol{a}%
}^{\left(  0\right)  }\left(  x\right)  >0$, and the cocycle property of tau
functions shows%
\[
\tau_{q_{x}\boldsymbol{a}}\left(  z^{n}\right)  =\dfrac{\tau_{\boldsymbol{a}%
}\left(  q_{x}z^{n}\right)  }{\tau_{\boldsymbol{a}}\left(  q_{x}\right)
}=\dfrac{\tau_{\boldsymbol{a}}\left(  z^{n}\right)  \tau_{z^{n}\boldsymbol{a}%
}\left(  q_{x}\right)  }{\tau_{\boldsymbol{a}}\left(  q_{x}\right)  }>0
\]
due to $z^{n}\boldsymbol{a}\in\boldsymbol{A}_{++}^{inv}\left(  C\right)  $.
Since $q_{x}\boldsymbol{a}\in\boldsymbol{A}_{+}^{inv}\left(  C\right)  $
clearly holds, one has $q_{x}\boldsymbol{a}\in\boldsymbol{A}_{++}^{inv}\left(
C\right)  $. Similarly one can show $r_{\zeta}\boldsymbol{a}\in\boldsymbol{A}%
_{++}^{inv}\left(  C\right)  $ for $\zeta\in D_{-}\backslash\mathbb{R}$. Then,
similarly to the proof of Lemma \ref{l15} we have $\tau_{\boldsymbol{a}%
}\left(  r\right)  >0$ for any $r\in\Gamma_{\operatorname{real}}\left(
C\right)  $, which implies $r\boldsymbol{a}\in\boldsymbol{A}_{++}^{inv}\left(
C\right)  $. Then, applying Lemma \ref{l15} to $r\boldsymbol{a}$, we have
$\tau_{r\boldsymbol{a}}\left(  e^{h}\right)  >0$, and hence $\tau
_{\boldsymbol{a}}\left(  re^{h}\right)  >0$, which shows $re^{h}%
\boldsymbol{a}\in\boldsymbol{A}_{++}^{inv}\left(  C\right)  $.\bigskip
\end{proof}

The next issue is to find a sufficient condition for $\boldsymbol{a}%
\in\boldsymbol{A}^{inv}\left(  C\right)  $ to be an element of $\boldsymbol{A}%
_{++}^{inv}\left(  C\right)  $, which will be clarified in the next section.

\section{$m$-function and Weyl function}

In this section we relate $m$-functions to Weyl functions (see Appendix for
the definition)

First we present a lemma which connects the existence of a positive solution
to a Jacobi equation with the estimate of the spectrum of the Jacobi operator.
For $a_{n}>0$, $b_{n}\in\mathbb{R}$ let $q=\left\{  a_{n},b_{n}\right\}
_{n\in\mathbb{Z}}$ and%
\[
\left(  H_{q}u\right)  _{n}=a_{n+1}u_{n+1}+a_{n}u_{n-1}+b_{n}u_{n}\text{.}%
\]
Set%
\[
\ell_{0}\left(  \mathbb{Z}\right)  =\left\{  f\in\ell^{2}\left(
\mathbb{Z}\right)  \text{; }f\text{ has a finite support}\right\}  .
\]

\begin{lemma}
\label{l16}Suppose there exists $\lambda_{0}<\lambda_{1}$ and two solutions
$u_{n}$, $v_{n}$ to%
\[
\left\{
\begin{array}
[c]{ll}%
H_{q}u=\lambda_{1}u\text{, \ }u_{n}>0 & \text{for any }n\in\mathbb{Z}\\
H_{q}v=\lambda_{0}v\text{, \ }\left(  -1\right)  ^{n}v_{n}>0 & \text{for any
}n\in\mathbb{Z}%
\end{array}
\right.  \text{.}%
\]
Then it holds that%
\begin{equation}
\lambda_{0}\left(  f,f\right)  \leq\left(  H_{q}f,f\right)  \leq\lambda
_{1}\left(  f,f\right)  \text{ \ for any }f\in\ell_{0}\left(  \mathbb{Z}%
\right)  \text{.} \label{33}%
\end{equation}
Consequently $H_{q}$ can be uniquely extended to $\ell^{2}\left(
\mathbb{Z}\right)  $ as a bounded self-adjoint operator, and the boundaries
$\pm\infty$ are of limit point type.
\end{lemma}

\begin{proof}
Define%
\[
\phi_{n-1}=\dfrac{u_{n}}{u_{n-1}}>0\text{ and }\rho_{o}\left(  n\right)
=-\sqrt{\dfrac{a_{n+1}}{\phi_{n}}}<0\text{, \ }\rho_{e}\left(  n\right)
=\sqrt{a_{n+1}\phi_{n}}>0\text{,}%
\]
and operators $A$, $A^{\ast}$ on $\ell_{0}^{2}\left(  \mathbb{Z}\right)  $ by%
\[
\left(  Af\right)  _{n}=\rho_{o}\left(  n\right)  f_{n+1}+\rho_{e}\left(
n\right)  f_{n}\text{, \ \ }\left(  A^{\ast}f\right)  _{n}=\rho_{o}\left(
n-1\right)  f_{n-1}+\rho_{e}\left(  n\right)  f_{n}\text{.}%
\]
Note that%
\[
\left(  Af,g\right)  =\left(  f,A^{\ast}g\right)  \text{ \ for any }f\text{,
}g\in\ell_{0}^{2}\left(  \mathbb{Z}\right)
\]
is valid. Then it holds that%
\[
\left\{
\begin{array}
[c]{l}%
\rho_{e}\left(  n\right)  \rho_{o}\left(  n\right)  =-a_{n+1}\\
\rho_{o}\left(  n-1\right)  ^{2}+\rho_{e}\left(  n\right)  ^{2}=\dfrac{a_{n}%
}{\phi_{n-1}}+a_{n+1}\phi_{n}=\lambda_{1}-b_{n}%
\end{array}
\right.  \text{,}%
\]
and%
\begin{align*}
&  \left(  A^{\ast}Af\right)  _{n}\\
&  =\rho_{o}\left(  n-1\right)  \left(  \rho_{o}\left(  n-1\right)  f_{n}%
+\rho_{e}\left(  n-1\right)  f_{n-1}\right)  +\rho_{e}\left(  n\right)
\left(  \rho_{o}\left(  n\right)  f_{n+1}+\rho_{e}\left(  n\right)
f_{n}\right) \\
&  =\rho_{e}\left(  n\right)  \rho_{o}\left(  n\right)  f_{n+1}+\left(
\rho_{o}\left(  n-1\right)  ^{2}+\rho_{e}\left(  n\right)  ^{2}\right)
f_{n}+\rho_{o}\left(  n-1\right)  \rho_{e}\left(  n-1\right)  f_{n-1}\\
&  =-a_{n}f_{n+1}-a_{n-1}f_{n-1}+\left(  \lambda_{1}-b_{n}\right)
f_{n}=\left(  \left(  \lambda_{1}-H_{q}\right)  f\right)  _{n}\text{,}%
\end{align*}
hence $\lambda_{1}-H_{q}=A^{\ast}A$ is valid on $\ell_{0}^{2}\left(
\mathbb{Z}\right)  $, which shows for any $f\in\ell_{0}^{2}\left(
\mathbb{Z}\right)  $%
\[
\left(  \left(  \lambda_{1}-H_{q}\right)  f,f\right)  =\left(  A^{\ast
}Af,f\right)  =\left(  Af,Af\right)  \geq0\text{.}%
\]

To show the converse inequality set%
\[
\phi_{n-1}=-\dfrac{v_{n}}{v_{n-1}}>0\text{, }\rho_{o}\left(  n\right)
=-\sqrt{\dfrac{a_{n+1}}{\phi_{n}}}<0\text{, \ }\rho_{e}\left(  n\right)
=\sqrt{a_{n+1}\phi_{n}}>0\text{,}%
\]
and operators $A$, $A^{\ast}$ on $\ell_{0}^{2}\left(  \mathbb{Z}\right)  $ by%
\[
\left(  Af\right)  _{n}=\rho_{o}\left(  n\right)  f_{n+1}-\rho_{e}\left(
n\right)  f_{n}\text{, \ \ }\left(  A^{\ast}f\right)  _{n}=\rho_{o}\left(
n-1\right)  f_{n-1}-\rho_{e}\left(  n\right)  f_{n}\text{.}%
\]
Then, $A^{\ast}A=H_{q}-\lambda_{0}$ holds, hence the converse inequality is
valid. The rest of the statement is clear.\bigskip
\end{proof}

Let $\lambda_{\pm}$ be%
\begin{equation}
\left\{
\begin{array}
[c]{l}%
\lambda_{+}=\inf\left\{  x+x^{-1}\text{; \ }x\in D_{-}\cap\mathbb{R}%
_{+}\right\}  >0\text{,}\\
\lambda_{-}=\sup\left\{  x+x^{-1}\text{; \ }x\in D_{-}\cap\mathbb{R}%
_{-}\right\}  <0\text{.}%
\end{array}
\right.  \label{33-1}%
\end{equation}
Now we have

\begin{proposition}
\label{p3}Suppose $\boldsymbol{a}\in\boldsymbol{A}_{++}^{inv}\left(  C\right)
$. Then, the $q=\left\{  a_{n},b_{n}\right\}  _{n\in\mathbb{Z}}$ associated
with $\boldsymbol{a}$ by (\ref{21}) provides a bounded Jacobi operator $H_{q}$
satisfying \textrm{sp}$H_{q}\subset\left[  \lambda_{-},\lambda_{+}\right]  $.
Moreover, its $m$-function $m_{\boldsymbol{a}}$ is given by the Weyl functions
$m_{\pm}$ of $H_{q}$ as%
\[
m_{\boldsymbol{a}}\left(  z\right)  =\left\{
\begin{array}
[c]{ll}%
z+z^{-1}+a_{1}^{2}m_{+}\left(  z+z^{-1}\right)  & \text{if }z\in D_{-}^{1}\\
-a_{0}^{2}m_{-}\left(  z+z^{-1}\right)  +b_{0} & \text{if }z\in D_{-}^{2}%
\end{array}
\right.  \text{.}%
\]

\end{proposition}

\begin{proof}
Lemma \ref{l11} says%
\[
g_{n}\left(  \zeta\right)  =\dfrac{\boldsymbol{a}T\left(  z^{n}\boldsymbol{a}%
\right)  ^{-1}1\left(  \zeta\right)  }{\sqrt{1+\varphi_{z^{n}\boldsymbol{a}%
}^{\left(  0\right)  }(0)}}=\zeta^{-n}\dfrac{1+\varphi_{z^{n}\boldsymbol{a}%
}^{\left(  0\right)  }\left(  \zeta\right)  }{\sqrt{1+\varphi_{z^{n}%
\boldsymbol{a}}^{\left(  0\right)  }(0)}}%
\]
satisfies $H_{q}g=\left(  \zeta+\zeta^{-1}\right)  g$ for $\zeta\in D_{-}$.
Since $g_{n}\left(  \zeta\right)  >0$ for any $n\in\mathbb{Z}$ if $\zeta\in
D_{-}\cap\mathbb{R}_{+}$ and $\left(  -1\right)  ^{n}g_{n}\left(
\zeta\right)  >0$ for any $n\in\mathbb{Z}$ if $\zeta\in D_{-}\cap
\mathbb{R}_{-}$ hold, from Lemma \ref{l16} one has the property \textrm{sp}%
$H_{q}\subset\left[  \lambda_{-},\lambda_{+}\right]  $. The boundedness of
$H_{q}$ implies that the boundaries $\pm\infty$ are of limit point type.

We show $g_{n}\left(  \zeta\right)  \in\ell^{2}\left(  \mathbb{Z}_{+}\right)
$ for $\zeta\in D_{-}^{1}\cap\mathbb{C}_{+}$. Observe%
\[
\dfrac{g_{k+1}\left(  \zeta\right)  }{g_{k}\left(  \zeta\right)  }%
=\dfrac{\sqrt{\tau_{z^{k}\boldsymbol{a}}(z^{-1})}}{\sqrt{\tau_{z^{k+1}%
\boldsymbol{a}}(z^{-1})}}\dfrac{\tau_{z^{k+1}\boldsymbol{a}}(q_{\zeta})}%
{\zeta\tau_{z^{k}\boldsymbol{a}}(q_{\zeta})}=\dfrac{\sqrt{\tau_{z^{k}%
\boldsymbol{a}}(z^{-1})}}{\sqrt{\tau_{z^{k+1}\boldsymbol{a}}(z^{-1})}}%
\dfrac{\tau_{z^{k+1}\boldsymbol{a}}(q_{\zeta})}{\zeta\tau_{z^{kn}%
\boldsymbol{a}}(q_{\zeta})}\text{.}%
\]
Since%
\[
\tau_{z^{k+1}\boldsymbol{a}}(q_{\zeta})=\tau_{z^{-1}\widetilde{z^{k}%
\boldsymbol{a}}}(\widetilde{q}_{\zeta})=\dfrac{\tau_{\widetilde{z^{k}%
\boldsymbol{a}}}(z^{-1}\widetilde{q}_{\zeta})}{\tau_{\widetilde{z^{k}%
\boldsymbol{a}}}(z^{-1})}=\dfrac{\tau_{\widetilde{z^{k}\boldsymbol{a}}%
}(q_{\zeta^{-1}})}{\tau_{z^{k}\boldsymbol{a}}(z)}\text{ \ (}z^{-1}%
\widetilde{q}_{\zeta}=-\zeta q_{\zeta^{-1}}\text{)}%
\]
holds, one has%
\begin{equation}
\dfrac{g_{k+1}\left(  \zeta\right)  }{g_{k}\left(  \zeta\right)  }%
=\dfrac{\sqrt{\tau_{z^{k}\boldsymbol{a}}(z^{-1})}}{\sqrt{\tau_{z^{k+1}%
\boldsymbol{a}}(z^{-1})}}\dfrac{1}{\tau_{z^{k}\boldsymbol{a}}(z)}\dfrac
{\tau_{\widetilde{z^{k}\boldsymbol{a}}}(q_{\zeta^{-1}})}{\zeta\tau
_{z^{k}\boldsymbol{a}}(q_{\zeta})}=c_{k}n_{z^{k}\boldsymbol{a}}\left(
\zeta\right)  \label{34}%
\end{equation}
with $c_{k}=\sqrt{\tau_{\boldsymbol{a}}(z^{k-1})}/\sqrt{\tau_{\boldsymbol{a}%
}(z^{k+1})}$. Since $g_{k}\left(  \zeta\right)  $ satisfies (\ref{22}),
(\ref{34}) implies%
\begin{align*}
c_{k}n_{z^{k}\boldsymbol{a}}\left(  \zeta\right)   &  =a_{k+1}^{-1}%
\dfrac{\left(  \zeta+\zeta^{-1}\right)  g_{k}\left(  \zeta\right)
-a_{k}g_{k-1}\left(  \zeta\right)  -b_{k}g_{k}\left(  \zeta\right)  }%
{g_{k}\left(  \zeta\right)  }\\
&  =a_{k+1}^{-1}\left(  \zeta+\zeta^{-1}-b_{k}\right)  -a_{k}a_{k+1}%
^{-1}\left(  c_{k-1}n_{z^{k-1}\boldsymbol{a}}\left(  \zeta\right)  \right)
^{-1}\text{.}%
\end{align*}
Taking the imaginary parts, one has%
\[
\operatorname{Im}c_{k}n_{z^{k}\boldsymbol{a}}\left(  \zeta\right)
=a_{k+1}^{-1}\operatorname{Im}\left(  \zeta+\zeta^{-1}\right)  +a_{k}%
a_{k+1}^{-1}\dfrac{\operatorname{Im}c_{k-1}n_{z^{k-1}\boldsymbol{a}}\left(
\zeta\right)  }{\left\vert c_{k-1}n_{z^{k-1}\boldsymbol{a}}\left(
\zeta\right)  \right\vert ^{2}}\text{,}%
\]
hence%
\[
1-\dfrac{a_{k+1}^{-1}\operatorname{Im}\left(  \zeta+\zeta^{-1}\right)
}{\operatorname{Im}c_{k}n_{z^{k}\boldsymbol{a}}\left(  \zeta\right)  }%
=a_{k}a_{k+1}^{-1}\dfrac{\operatorname{Im}c_{k-1}n_{z^{k-1}\boldsymbol{a}%
}\left(  \zeta\right)  }{\operatorname{Im}c_{k}n_{z^{k}\boldsymbol{a}}\left(
\zeta\right)  }\left\vert c_{k-1}n_{z^{k-1}\boldsymbol{a}}\left(
\zeta\right)  \right\vert ^{-2}%
\]
follows. Note here $\operatorname{Im}c_{k}n_{z^{k}\boldsymbol{a}}\left(
\zeta\right)  <0$ on $D_{-}^{1}\cap\mathbb{C}_{+}$ due to (ii) of Lemma
\ref{l14}. Since one has%
\[
\dfrac{\left\vert g_{n+1}\left(  \zeta\right)  \right\vert ^{2}}{\left\vert
g_{0}\left(  \zeta\right)  \right\vert ^{2}}=\dfrac{a_{1}\operatorname{Im}%
c_{0}n_{\boldsymbol{a}}}{a_{n+2}\operatorname{Im}c_{n+1}n_{z^{n+1}%
\boldsymbol{a}}}\prod\nolimits_{1\leq k\leq n+1}\left(  1-\dfrac
{\operatorname{Im}\left(  \zeta+\zeta^{-1}\right)  }{a_{k+1}\operatorname{Im}%
c_{k}n_{z^{k}\boldsymbol{a}}}\right)  ^{-1}\text{,}%
\]
setting%
\[
A_{k}=\operatorname{Im}\left(  \zeta+\zeta^{-1}\right)  \left(  -a_{k+1}%
\operatorname{Im}c_{k}n_{z^{k}\boldsymbol{a}}\right)  ^{-1}\text{,
\ }B=\left(  -a_{1}\operatorname{Im}c_{0}n_{\boldsymbol{a}}\right)
/\operatorname{Im}\left(  \zeta+\zeta^{-1}\right)  \text{,}%
\]
we have%
\begin{align*}
\sum\nolimits_{0\leq n\leq N}\dfrac{\left\vert g_{n+1}\left(  \zeta\right)
\right\vert ^{2}}{\left\vert g_{0}\left(  \zeta\right)  \right\vert ^{2}}  &
=B\sum\nolimits_{0\leq n\leq N}A_{n+1}\prod\nolimits_{1\leq k\leq n+1}\left(
1+A_{k}\right)  ^{-1}\\
&  =B\left(  1-\prod\nolimits_{1\leq k\leq N+1}\left(  1+A_{k}\right)
^{-1}\right)  \leq B\text{,}%
\end{align*}
which implies $g_{n}\left(  \zeta\right)  \in\ell^{2}\left(  \mathbb{Z}%
_{+}\right)  $ for $\zeta\in D_{-}^{1}\cap\mathbb{C}_{+}$. Similarly we have
$g_{n}\left(  \zeta\right)  \in\ell^{2}\left(  \mathbb{Z}_{-}\right)  $ for
$\zeta\in D_{-}^{2}\cap\mathbb{C}_{+}$. Then, the definition of Weyl functions
yields%
\[
m_{+}\left(  \zeta+\zeta^{-1}\right)  =-\dfrac{g_{1}\left(  \zeta\right)
}{a_{1}g_{0}\left(  \zeta\right)  }=-\dfrac{c_{0}}{a_{1}}n_{\boldsymbol{a}%
}\left(  \zeta\right)  =-a_{1}^{-2}\left(  \zeta+\zeta^{-1}-m_{\boldsymbol{a}%
}\left(  \zeta\right)  \right)  \text{.}%
\]
For $m_{-}$ one has%
\[
m_{-}\left(  \zeta+\zeta^{-1}\right)  =-\dfrac{g_{-1}\left(  \zeta\right)
}{a_{0}g_{0}\left(  \zeta\right)  }=\dfrac{-1}{a_{0}c_{-1}n_{z^{-1}%
\boldsymbol{a}}\left(  \zeta\right)  }=\dfrac{-1}{a_{0}\sqrt{\tau
_{\boldsymbol{a}}(z^{-2})}}\dfrac{\zeta\tau_{\boldsymbol{a}}\left(
z^{-1}q_{\zeta}\right)  }{\tau_{\boldsymbol{a}}\left(  q_{\zeta}\right)
}\text{.}%
\]
Since Lemma \ref{l9} implies%
\[
\zeta\tau_{\boldsymbol{a}}\left(  z^{-1}q_{\zeta}\right)  =\tau
_{\boldsymbol{a}}\left(  z^{-1}\right)  \tau_{\boldsymbol{a}}\left(  q_{\zeta
}\right)  \left(  m_{\boldsymbol{a}}\left(  \zeta\right)  -m_{\boldsymbol{a}%
}\left(  0\right)  \right)  \text{,}%
\]
one has%
\[
m_{\boldsymbol{a}}\left(  z\right)  =-a_{0}^{2}m_{-}\left(  z+z^{-1}\right)
+m_{\boldsymbol{a}}\left(  0\right)  \text{.}%
\]
To identify $m_{\boldsymbol{a}}\left(  0\right)  $ with $b_{0}$ we start from
the definition (\ref{20}) of $b_{0}$:%
\[
b_{0}=\varphi_{\widetilde{\boldsymbol{a}}}^{\left(  -1\right)  }\left(
0\right)  -\varphi_{\widetilde{z^{-1}\boldsymbol{a}}}^{\left(  -1\right)
}\left(  0\right)  \text{.}%
\]
To change the symbol $z^{-1}\boldsymbol{a}$ to $\boldsymbol{a}$ we use tau
functions. We have%
\begin{align*}
\varphi_{\widetilde{z^{-1}\boldsymbol{a}}}^{\left(  -1\right)  }\left(
\zeta\right)   &  =\zeta^{-1}\varphi_{z^{-1}\boldsymbol{a}}^{\left(  0\right)
}\left(  \zeta^{-1}\right)  =\zeta^{-1}\left(  \tau_{z^{-1}\boldsymbol{a}%
}\left(  q_{\zeta^{-1}}\right)  -1\right)  =\dfrac{\tau_{\boldsymbol{a}%
}\left(  z^{-1}q_{\zeta^{-1}}\right)  }{\zeta\tau_{\boldsymbol{a}}\left(
z^{-1}\right)  }-\zeta^{-1}\\
&  =\left(  1+\varphi_{\boldsymbol{a}}^{\left(  0\right)  }\left(  \zeta
^{-1}\right)  \right)  \left(  m_{\boldsymbol{a}}\left(  \zeta^{-1}\right)
-m_{\boldsymbol{a}}\left(  0\right)  \right)  -\zeta^{-1}\\
&  =-m_{\boldsymbol{a}}\left(  0\right)  +\varphi_{\widetilde{\boldsymbol{a}}%
}^{\left(  -1\right)  }\left(  0\right)  +O\left(  \zeta\right)  \text{ \ \ as
\ }\zeta\rightarrow0\text{,}%
\end{align*}
which completes the proof.\bigskip
\end{proof}

This proposition implies that for $\boldsymbol{a}\in\boldsymbol{A}_{++}%
^{inv}\left(  C\right)  $ the associated $m$-function $m_{\boldsymbol{a}}$ is
automatically extended analytically to $\mathbb{C}\backslash\Sigma
_{\lambda_{0}}$ with $\lambda_{0}=\max\left\{  \lambda_{+},-\lambda
_{-}\right\}  $.

\section{Proof of Theorem by identifying $\boldsymbol{A}_{++}^{inv}$ with
$Q_{\lambda_{0}}$}

Although we have constructed a Toda flow on $\boldsymbol{A}_{++}^{inv}$ by
Propositions \ref{p1}, \ref{p2}, we have not been able to specify
$\boldsymbol{A}_{++}^{inv}\left(  C\right)  $ in a concrete terminology. In
this section we identify $\boldsymbol{A}_{++}^{inv}\left(  C\right)  $ with
$Q_{\lambda_{0}}$ ($\lambda_{0}>2$), which provides automatically a proof of Theorem.

Proposition \ref{p3} shows $\boldsymbol{A}_{++}^{inv}\left(  C\right)  \subset
Q_{\lambda_{0}}$ in a certain sense. Our next task here is to verify the
converse statement. For $q=\left\{  a_{n},b_{n}\right\}  _{n\in\mathbb{Z}}\in
Q_{\lambda_{0}}$ let $m_{\pm}$ be the Weyl functions for the Jacobi operator
$H_{q}$ and set%
\[
m\left(  z\right)  =\left\{
\begin{array}
[c]{ll}%
z+z^{-1}+a_{1}^{2}m_{+}\left(  z+z^{-1}\right)  & \text{if }z\in
\Sigma_{\lambda_{0}}\text{ and }\left\vert z\right\vert >1\\
-a_{0}^{2}m_{-}\left(  z+z^{-1}\right)  +b_{0} & \text{if }z\in\Sigma
_{\lambda_{0}}\text{ and\ }\left\vert z\right\vert <1
\end{array}
\right.
\]
with%
\[
\Sigma_{\lambda_{0}}=\left\{  \left\vert z\right\vert =1\right\}  \cup\left[
-\ell,-\ell^{-1}\right]  \cup\left[  \ell^{-1},\ell\right]  \text{ }\left(
\ell=\left(  \lambda_{0}+\sqrt{\lambda_{0}^{2}-4}\right)  /2\right)  .
\]
Then, $m$ is an element of%
\begin{equation}
M_{\lambda_{0}}=\left\{
\begin{array}
[c]{l}%
m\text{; \ }m\text{ is analytic on }\mathbb{C}\backslash\Sigma_{\lambda_{0}%
}\text{ satisfying\ }m=\overline{m}\text{ and}\\
\text{(i) \ }m(z)=z+O\left(  z^{-1}\right)  \text{ \ as }z\rightarrow\infty\\
\text{(ii) \ }\operatorname{Im}m\left(  z\right)  >0\text{ \ on }%
\mathbb{C}_{+}\backslash\Sigma_{\lambda_{0}}\\
\text{(iii)\ }m\left(  z\right)  -m\left(  z^{-1}\right)  \neq0\text{ \ on
}\mathbb{C}_{+}\backslash\Sigma_{\lambda_{0}}\\
\text{(iv) }m\left(  z\right)  \text{ is not a rational function of }z+z^{-1}%
\end{array}
\right\}  \text{,} \label{35}%
\end{equation}
which obeys from Lemma \ref{la1} in Appendix.

$m\in M_{\lambda_{0}}$ has the following expression.

\begin{lemma}
\label{l17}For $m\in M_{\lambda_{0}}$ there exist $a_{0}^{2}>0$, $a_{1}^{2}%
>0$, $b_{0}\in\mathbb{R}$ and probability measures $\sigma_{\pm}$ on $\left[
-\lambda_{0},\lambda_{0}\right]  $ such that%
\[
m\left(  z\right)  =\left\{
\begin{array}
[c]{ll}%
z+z^{-1}+a_{1}^{2}%
{\displaystyle\int_{-\lambda_{0}}^{\lambda_{0}}}
\dfrac{\sigma_{+}\left(  d\lambda\right)  }{\lambda-\left(  z+z^{-1}\right)  }
& \text{if \ }\left\vert z\right\vert >1\\
-a_{0}^{2}%
{\displaystyle\int_{-\lambda_{0}}^{\lambda_{0}}}
\dfrac{\sigma_{-}\left(  d\lambda\right)  }{\lambda-\left(  z+z^{-1}\right)
}+b_{0} & \text{if \ }\left\vert z\right\vert <1
\end{array}
\right.  \text{.}%
\]

\end{lemma}

\begin{proof}
$\phi\left(  z\right)  =z+z^{-1}$ yields a conformal map:%
\[
\phi:\left\{
\begin{array}
[c]{c}%
\mathbb{C}_{+}\cap\left\{  \left\vert z\right\vert >1\right\}  \rightarrow
\mathbb{C}_{+}\\
\mathbb{C}_{+}\cap\left\{  \left\vert z\right\vert <1\right\}  \rightarrow
\mathbb{C}_{-}%
\end{array}
\right.  \text{.}%
\]
Analytic functions $n_{\pm}\left(  z\right)  =\pm m\left(  \phi^{-1}\left(
z\right)  \right)  $ on $\mathbb{C}_{\pm}$ satisfy $\pm\operatorname{Im}%
n_{\pm}\left(  z\right)  >0$, hence they have Herglotz representations, namely
there exist $\alpha_{\pm}\in\mathbb{R}$, $\beta_{\pm}\geq0$ and measures
$\sigma_{\pm}$ on $\mathbb{R}$ such that%
\[
n_{\pm}\left(  z\right)  =\alpha_{\pm}+\beta_{\pm}z+%
{\displaystyle\int_{-\infty}^{\infty}}
\left(  \dfrac{1}{\lambda-z}-\dfrac{\lambda}{\lambda^{2}+1}\right)
\sigma_{\pm}\left(  d\lambda\right)  \text{,}%
\]
hence%
\[
m\left(  z\right)  =\left\{
\begin{array}
[c]{c}%
\alpha_{+}+\beta_{+}\phi\left(  z\right)  +%
{\displaystyle\int_{-\infty}^{\infty}}
\left(  \dfrac{1}{\lambda-\phi\left(  z\right)  }-\dfrac{\lambda}{\lambda
^{2}+1}\right)  \sigma_{+}\left(  d\lambda\right)  \text{ \ if }\left\vert
z\right\vert >1\\
-\alpha_{-}-\beta_{-}\phi\left(  z\right)  -%
{\displaystyle\int_{-\infty}^{\infty}}
\left(  \dfrac{1}{\lambda-\phi\left(  z\right)  }-\dfrac{\lambda}{\lambda
^{2}+1}\right)  \sigma_{-}\left(  d\lambda\right)  \text{ \ if }\left\vert
z\right\vert <1
\end{array}
\right.  \text{.}%
\]
The property $m=\overline{m}$ of (\ref{35}) imply the measures $\sigma_{\pm}$
should be supported on $\left[  -\lambda_{0},\lambda_{0}\right]  $. Moreover,
the property (i) shows $\beta_{+}=1$ and%
\[
m\left(  z\right)  =\phi\left(  z\right)  +%
{\displaystyle\int_{-\lambda_{0}}^{\lambda_{0}}}
\dfrac{1}{\lambda-\phi\left(  z\right)  }\sigma_{+}\left(  d\lambda\right)
\text{ \ if }\left\vert z\right\vert >1\text{.}%
\]
The property (iv) implies $m\left(  z\right)  =\phi\left(  z\right)  $ not
identically, hence setting $a_{1}^{2}\equiv\sigma_{+}\left(  \left[
-\lambda_{0},\lambda_{0}\right]  \right)  >0$ and normalizing $\sigma_{+}$,
one has the first expression of the lemma. Since we know $m\left(  0\right)  $
is finite, $m(z)$ on $\left\{  \left\vert z\right\vert <1\right\}  $ takes a
form:%
\[
m\left(  z\right)  =-%
{\displaystyle\int_{-\lambda_{0}}^{\lambda_{0}}}
\dfrac{1}{\lambda-\phi\left(  z\right)  }\sigma_{-}\left(  d\lambda\right)
+b_{0}%
\]
with $b_{0}\in\mathbb{R}$. The strict positivity of $\operatorname{Im}m\left(
z\right)  >0$ on $\mathbb{C}_{+}\cap\left\{  \left\vert z\right\vert
<1\right\}  $ implies non-vanishing of $\sigma_{-}$. Then, setting $a_{0}%
^{2}=\sigma_{-}\left(  \left[  -\lambda_{0},\lambda_{0}\right]  \right)  >0$,
one has the second expression of the lemma.\bigskip
\end{proof}

For $m\in M_{\lambda_{0}}$ define a symbol%
\begin{equation}
\boldsymbol{m}\left(  z\right)  =\left(  \dfrac{zm(z)-1}{z^{2}-1},z^{2}%
\dfrac{z-m\left(  z\right)  }{z^{2}-1}\right)  \text{.} \label{37}%
\end{equation}
We fix any bounded domain $D_{+}$ containing $\Sigma_{\lambda_{0}}$ and
satisfying $D_{+}\ni z\rightarrow z^{-1}$, $\overline{z}\in D_{+}$ and having
smooth boundaries. Then, this $\boldsymbol{m}$ satisfies the conditions (i)
and (ii) of $\boldsymbol{M}\left(  C\right)  $ (see (\ref{18-1})). The
condition (iii) of $\boldsymbol{M}\left(  C\right)  $ is certified by
computing%
\[
m_{1}\left(  z\right)  \widetilde{m}_{1}\left(  z\right)  -m_{2}\left(
z\right)  \widetilde{m}_{2}\left(  z\right)  =\dfrac{m\left(  z\right)
-m\left(  z^{-1}\right)  }{z-z^{-1}}\neq0\text{,}%
\]
hence $\boldsymbol{m}\in\boldsymbol{M}\left(  C\right)  $ and $\boldsymbol{m}%
\in\boldsymbol{A}^{inv}\left(  C\right)  $. Moreover, Lemma \ref{l6} yields%
\[
\left\{
\begin{array}
[c]{l}%
1+\varphi_{\boldsymbol{m}}^{\left(  0\right)  }\left(  z\right)  =1\\
z^{-1}+\varphi_{\boldsymbol{m}}^{\left(  -1\right)  }\left(  z\right)
=\phi\left(  z\right)  -m(z)
\end{array}
\right.  \Longrightarrow m_{\boldsymbol{m}}\left(  z\right)  =m\left(
z\right)  \text{.}%
\]
Recall $D_{-}^{1}=D_{-}\cap\left\{  \left\vert z\right\vert >1\right\}  $,
$D_{-}^{2}=D_{-}\cap\left\{  \left\vert z\right\vert <1\right\}  $.

\begin{lemma}
\label{l18}If $\boldsymbol{a}\in\boldsymbol{A}^{inv}\left(  C\right)  $
satisfies $\boldsymbol{a}=\overline{\boldsymbol{a}}$ and%
\[
1+\varphi_{\boldsymbol{a}}^{\left(  0\right)  }\left(  0\right)
>0,\ \ \ m_{\boldsymbol{a}}\in M_{\lambda_{0}}\text{,}%
\]
then it holds that\newline(i) \ $\tau_{\boldsymbol{a}}\left(  z\right)
\tau_{\boldsymbol{a}}\left(  z^{-1}\right)  =\lim_{\zeta\rightarrow\infty
}\zeta\left(  \phi\left(  \zeta\right)  -m_{\boldsymbol{a}}\left(
\zeta\right)  \right)  >0$,\newline(ii) $\ \tau_{\boldsymbol{a}}\left(
q_{x}\right)  >0$ on $D_{-}\cap\mathbb{R}$ and $\tau_{\boldsymbol{a}}\left(
r_{\zeta}\right)  >0$ on $D_{-}$,\newline(iii) $m_{\widetilde{\boldsymbol{a}}%
}$, $m_{q_{x}\boldsymbol{a}}$, $m_{r_{\zeta}\boldsymbol{a}}\in M_{\lambda_{0}%
}$ for any $x\in D_{-}\cap\mathbb{R}$ and $\zeta\in D_{-}\backslash\mathbb{R}$.
\end{lemma}

\begin{proof}
Note%
\[
\tau_{\boldsymbol{a}}\left(  z\right)  =\tau_{\widetilde{\boldsymbol{a}}%
}\left(  z^{-1}\right)  =1+\varphi_{\widetilde{\boldsymbol{a}}}^{\left(
0\right)  }\left(  0\right)  =\lim_{\varepsilon\rightarrow0}\left(
1+\varepsilon^{-1}\varphi_{\boldsymbol{a}}^{\left(  -1\right)  }\left(
\varepsilon^{-1}\right)  \right)  \text{.}%
\]
(\ref{28}), (\ref{28-1}) imply%
\[
\tau_{\boldsymbol{a}}\left(  z^{-1}\right)  \left(  1+\zeta\varphi
_{\boldsymbol{a}}^{\left(  -1\right)  }\left(  \zeta\right)  \right)
=\zeta\left(  \phi\left(  \zeta\right)  -m_{\boldsymbol{a}}\left(
\zeta\right)  \right)  \left(  1+\varphi_{\boldsymbol{a}}^{\left(  0\right)
}\left(  \zeta\right)  \right)  \text{,}%
\]
hence%
\begin{equation}
\tau_{\boldsymbol{a}}\left(  z^{-1}\right)  \lim_{\varepsilon\rightarrow
0}\left(  1+\varepsilon^{-1}\varphi_{\boldsymbol{a}}^{\left(  -1\right)
}\left(  \varepsilon^{-1}\right)  \right)  =\lim_{\varepsilon\rightarrow
0}\varepsilon^{-1}\left(  \phi\left(  \varepsilon^{-1}\right)
-m_{\boldsymbol{a}}\left(  \varepsilon^{-1}\right)  \right)  \text{.}
\label{36}%
\end{equation}
And the expression of Lemma \ref{l17} shows (i).

If $1+\varphi_{\boldsymbol{a}}^{\left(  0\right)  }\left(  0\right)  \neq0$,
$1+\varphi_{\boldsymbol{a}}^{\left(  0\right)  }$ and $z+\varphi
_{\boldsymbol{a}}^{\left(  1\right)  }$ cannot vanish simultaneously on
$D_{-}$, which together with the analyticity of $m_{\boldsymbol{a}}$ on
$\mathbb{C}\backslash\Sigma_{\lambda_{0}}$ implies $1+\varphi_{\boldsymbol{a}%
}^{\left(  0\right)  }\left(  \zeta\right)  \neq0$ on $D_{-}$. On $\mathbb{R}$
we have%
\[
1+\varphi_{\boldsymbol{a}}^{\left(  0\right)  }\left(  x\right)
\rightarrow1\text{ \ as }x\rightarrow\infty\text{ and }1+\varphi
_{\boldsymbol{a}}^{\left(  0\right)  }\left(  0\right)  >0\text{,}%
\]
hence $\tau_{\boldsymbol{a}}\left(  q_{x}\right)  =1+\varphi_{\boldsymbol{a}%
}^{\left(  0\right)  }\left(  x\right)  >0$ on $D_{-}\cap\mathbb{R}$. The
property $\tau_{\boldsymbol{a}}\left(  r_{\zeta}\right)  >0$ on $D_{-}%
\backslash\mathbb{R}$ follows from (ii) of Lemma \ref{l9} and $1+\varphi
_{\boldsymbol{a}}^{\left(  0\right)  }\left(  \zeta\right)  \neq0$ on $D_{-}$.

(\ref{28}), (\ref{28-1}) imply%
\[
m_{\widetilde{\boldsymbol{a}}}\left(  \zeta\right)  =\phi\left(  \zeta\right)
-\dfrac{\tau_{\boldsymbol{a}}\left(  z\right)  }{n_{\boldsymbol{a}}\left(
\zeta^{-1}\right)  }=\phi\left(  \zeta\right)  -\dfrac{\tau_{\boldsymbol{a}%
}\left(  z\right)  \tau_{\boldsymbol{a}}\left(  z^{-1}\right)  }{\phi\left(
\zeta\right)  -m_{\boldsymbol{a}}\left(  \zeta^{-1}\right)  }\text{.}%
\]
$\operatorname{Im}m_{\widetilde{\boldsymbol{a}}}\left(  \zeta\right)
>\operatorname{Im}\phi\left(  \zeta\right)  >0$ on $D_{-}^{1}$ obeys from
$\tau_{\boldsymbol{a}}\left(  z\right)  \tau_{\boldsymbol{a}}\left(
z^{-1}\right)  >0$ and (\ref{35}). To see $\operatorname{Im}m_{\widetilde
{\boldsymbol{a}}}\left(  \zeta\right)  >0$ on $D_{-}^{2}$ set $f\left(
\zeta\right)  =\tau_{\boldsymbol{a}}\left(  z\right)  \tau_{\boldsymbol{a}%
}\left(  z^{-1}\right)  \left(  \phi\left(  \zeta\right)  -m_{\boldsymbol{a}%
}\left(  \zeta^{-1}\right)  \right)  ^{-1}$. Since $\phi$ maps $\mathbb{C}%
_{+}\cap\left\{  \left\vert z\right\vert <1\right\}  $ onto $\mathbb{C}_{-}$
conformaly, $f\left(  \phi^{-1}\left(  z\right)  \right)  $ satisfies
$\operatorname{Im}f\left(  \phi^{-1}\left(  z\right)  \right)  <0$ on
$\mathbb{C}_{-}$ and%
\[
\lim_{z\rightarrow\infty}z^{-1}f\left(  \phi^{-1}\left(  z\right)  \right)
=-1\text{ \ due to (\ref{36}).}%
\]
Then, $f\left(  \phi^{-1}\left(  z\right)  \right)  $ has an representation%
\[
f\left(  \phi^{-1}\left(  z\right)  \right)  =\alpha+z+\int_{-\infty}^{\infty
}\left(  \dfrac{1}{\lambda-z}-\dfrac{\lambda}{\lambda^{2}+1}\right)
\sigma\left(  d\lambda\right)  \text{ \ on }\mathbb{C}_{-}%
\]
with $\alpha\in\mathbb{R}$ and a non-negative measure $\sigma$, hence we have
$\operatorname{Im}f\left(  \zeta\right)  \leq\operatorname{Im}\phi\left(
\zeta\right)  $ if $\zeta\in D_{-}^{2}$. If we have an equality here for some
$\zeta_{0}$, then $f\left(  \zeta\right)  =\alpha+\phi\left(  \zeta\right)  $
identically should hold, which means $m_{\boldsymbol{a}}$ is a rational
function of $\phi$ on $D_{-}^{1}$. This contradicts (iv) of (\ref{35}), hence
one has $m_{\widetilde{\boldsymbol{a}}}\in M_{\lambda_{0}}$.

To verify $m_{q_{x}\boldsymbol{a}}\in M_{\lambda_{0}}$ we use (\ref{28-4}):%
\[
\left\{
\begin{array}
[c]{l}%
m_{q_{x}\boldsymbol{a}}\left(  z\right)  =\left(  m_{\boldsymbol{a}}\left(
0\right)  -m_{\boldsymbol{a}}\left(  x\right)  \right)  \left(  1-\dfrac
{\phi\left(  z\right)  -\phi\left(  x\right)  }{m_{\boldsymbol{a}}\left(
z\right)  -m_{\boldsymbol{a}}\left(  x\right)  }\right)  +\phi\left(  z\right)
\\
m_{q_{0}\boldsymbol{a}}\left(  z\right)  =m_{z^{-1}\boldsymbol{a}}\left(
z\right)  =\dfrac{m_{\boldsymbol{a}}^{\prime}\left(  0\right)  }%
{m_{\boldsymbol{a}}\left(  x\right)  -m_{\boldsymbol{a}}\left(  z\right)
}+\phi\left(  z\right)
\end{array}
\right.  \text{.}%
\]
The property $m_{q_{x}\boldsymbol{a}}=\overline{m}_{q_{x}\boldsymbol{a}}$ is
easy to see, and $\operatorname{Im}m_{\boldsymbol{a}}\left(  z\right)  >0$ on
$\mathbb{C}_{+}\backslash\Sigma_{\lambda_{0}}$ implies $m_{\boldsymbol{a}%
}\left(  z\right)  -m_{\boldsymbol{a}}\left(  x\right)  \neq0$, hence
$m_{q_{x}\boldsymbol{a}}$ is analytic on $\mathbb{C}_{\pm}\backslash
\Sigma_{\lambda_{0}}$. Note $m_{\boldsymbol{a}}^{\prime}\left(  x\right)  >0$
on $\mathbb{R}\backslash\Sigma_{\lambda_{0}}$ due to Lemma \ref{l17} (or Hopf
lemma), which yields $m_{\boldsymbol{a}}\left(  z\right)  -m_{\boldsymbol{a}%
}\left(  x\right)  \neq0$ for any $z\in\mathbb{R}\backslash\Sigma_{\lambda
_{0}}$. Then, we have the analyticity of $m_{\boldsymbol{a}}$ on
$\mathbb{C}\backslash\Sigma_{\lambda_{0}}$. The properties (i), (iii) and (iv)
of (\ref{35}) can be easily verified. The property (ii) follows from Lemma
\ref{la2} in Appendix, hence we have $m_{q_{x}\boldsymbol{a}}\in
M_{\lambda_{0}}$.

In the proof of $m_{r_{\zeta}\boldsymbol{a}}\in M_{\lambda_{0}}$ the
properties (i), (iii) and (iv) are easily verified and the property (ii)
follows from Lemma \ref{la2}. Only the analyticity remains to be proved. The
analyticity of $m_{r_{\zeta}\boldsymbol{a}}$ on $\mathbb{C}_{+}\backslash
\Sigma_{\lambda_{0}}$ follows from $\operatorname{Im}m_{r_{\zeta
}\boldsymbol{a}}\left(  z\right)  >0$ by (ii). On $\mathbb{R}$ we use%
\[
1+\varphi_{r_{\zeta}\boldsymbol{a}}^{\left(  0\right)  }\left(  x\right)
=\dfrac{\tau_{\boldsymbol{a}}\left(  r_{\zeta}q_{x}\right)  }{\tau
_{\boldsymbol{a}}\left(  r_{\zeta}\right)  }=\dfrac{\tau_{\boldsymbol{a}%
}\left(  q_{x}\right)  \tau_{q_{x}\boldsymbol{a}}\left(  r_{\zeta}\right)
}{\tau_{\boldsymbol{a}}\left(  r_{\zeta}\right)  }=\dfrac{\tau_{\boldsymbol{a}%
}\left(  q_{x}\right)  \left\vert \tau_{q_{x}\boldsymbol{a}}\left(  q_{\zeta
}\right)  \right\vert ^{2}\operatorname{Im}m_{q_{x}\boldsymbol{a}}\left(
\zeta\right)  }{\operatorname{Im}m_{\boldsymbol{a}}\left(  \zeta\right)
}\text{,}%
\]
which is positive, since $\tau_{\boldsymbol{a}}\left(  q_{x}\right)
\left\vert \tau_{q_{x}\boldsymbol{a}}\left(  q_{\zeta}\right)  \right\vert
^{2}>0$ and $m_{\boldsymbol{a}}$, $m_{q_{x}\boldsymbol{a}}\in M_{\lambda_{0}}%
$. Therefore, we see the analyticity of $m_{r_{\zeta}\boldsymbol{a}}$ on
$\mathbb{C}\backslash\Sigma_{\lambda_{0}}$, which completes the proof.\bigskip
\end{proof}

As in Introduction for $\boldsymbol{a}\in\boldsymbol{A}_{++}^{inv}\left(
C\right)  $ we define%
\[
a_{n}\left(  \boldsymbol{a}\right)  =\sqrt{\dfrac{\tau_{\boldsymbol{a}}\left(
z^{n}\right)  \tau_{\boldsymbol{a}}\left(  z^{n-2}\right)  }{\tau
_{\boldsymbol{a}}\left(  z^{n-1}\right)  ^{2}}}\text{, \ }b_{n}\left(
\boldsymbol{a}\right)  =\left.  \partial_{\varepsilon}\log\dfrac
{\tau_{\boldsymbol{a}}\left(  z^{n}e^{\varepsilon z}\right)  }{\tau
_{\boldsymbol{a}}\left(  z^{n-1}e^{\varepsilon z}\right)  }\right\vert
_{\varepsilon=0}\text{.}%
\]

\begin{lemma}
\label{l19}For $\boldsymbol{a}\in\boldsymbol{A}_{++}^{inv}\left(  C\right)  $,
$\boldsymbol{a}^{\prime}\in\boldsymbol{A}_{++}^{inv}\left(  C^{\prime}\right)
$ assume $m_{\boldsymbol{a}}=m_{\boldsymbol{a}^{\prime}}$ on $D_{-}\cap
D_{-}^{\prime}$. Then, one has $m_{g\boldsymbol{a}}=m_{g\boldsymbol{a}%
^{\prime}}$ for any $g\in\Gamma_{\operatorname{real}}\left(  C\right)
\cap\Gamma_{\operatorname{real}}\left(  C^{\prime}\right)  $ and $a_{n}\left(
\boldsymbol{a}\right)  =a_{n}\left(  \boldsymbol{a}^{\prime}\right)  $,
$b_{n}\left(  \boldsymbol{a}\right)  =b_{n}\left(  \boldsymbol{a}^{\prime
}\right)  $ for any $n\in\mathbb{Z}$.
\end{lemma}

\begin{proof}
The identity $m_{g\boldsymbol{a}}=m_{g\boldsymbol{a}^{\prime}}$ follows from
\begin{align*}
m_{q_{\zeta}\boldsymbol{a}}  &  =\left(  m_{\boldsymbol{a}}\left(  0\right)
-m_{\boldsymbol{a}}\left(  \zeta\right)  \right)  \left(  1-\dfrac{\phi
-\phi\left(  \zeta\right)  }{m_{\boldsymbol{a}}-m_{\boldsymbol{a}}\left(
\zeta\right)  }\right)  +\phi\\
&  =\left(  m_{\boldsymbol{a}^{\prime}}\left(  0\right)  -m_{\boldsymbol{a}%
^{\prime}}\left(  \zeta\right)  \right)  \left(  1-\dfrac{\phi-\phi\left(
\zeta\right)  }{m_{\boldsymbol{a}^{\prime}}-m_{\boldsymbol{a}^{\prime}}\left(
\zeta\right)  }\right)  +\phi=m_{q_{\zeta}\boldsymbol{a}^{\prime}}%
\end{align*}
on $D_{-}\cap D_{-}^{\prime}$. The identities for $a_{n}$, $b_{n}$ are
obtained by the uniqueness of the correspondence%
\[
q=\left\{  a_{n},b_{n}\right\}  _{n\in\mathbb{Z}}\Longleftrightarrow\left\{
m_{\pm},a_{0},a_{1},b_{0}\right\}  \text{.}%
\]
\medskip
\end{proof}

\noindent\textbf{Proof of Theorem}\newline Lemma \ref{l18} provides a proof of
Theorem as follows. We show $\boldsymbol{m}\in\boldsymbol{A}_{++}^{inv}\left(
C\right)  $. Let $r$ be a rational function such that%
\[
r_{1}=q_{x_{1}}\cdots q_{x_{m_{1}}}r_{\zeta_{1}}\cdots r_{\zeta_{n_{1}}}\text{
for }\exists x_{j}\in D_{-}\cap\mathbb{R}\text{, }\zeta_{j}\in D_{-}%
\backslash\mathbb{R}\text{.}%
\]
Applying Lemma \ref{l18} to $r_{1}$ inductively, we see $\tau_{\boldsymbol{m}%
}\left(  r_{1}\right)  >0$. General rational function $r\in\Gamma
_{\operatorname{real}}$ can be described by $r=r_{1}r_{2}^{-1}$ with%
\[
r_{2}=q_{y_{1}}\cdots q_{y_{m_{2}}}r_{\eta_{1}}\cdots r_{\eta_{n_{2}}}\text{
for }\exists y_{j}\in D_{-}\cap\mathbb{R}\text{, }\eta_{j}\in D_{-}%
\backslash\mathbb{R}\text{.}%
\]
Then, one has%
\[
\tau_{\boldsymbol{m}}\left(  r_{1}r_{2}^{-1}\right)  =\tau_{\boldsymbol{m}%
}\left(  r_{1}r_{2}^{-1}\widetilde{r}_{2}^{-1}\widetilde{r}_{2}\right)
=\tau_{\boldsymbol{m}}\left(  r_{2}^{-1}\widetilde{r}_{2}^{-1}\right)
\tau_{\boldsymbol{m}}\left(  r_{1}\widetilde{r}_{2}\right)  =\tau
_{\boldsymbol{m}}\left(  r_{2}^{-1}\widetilde{r}_{2}^{-1}\right)
\tau_{\widetilde{r_{1}\boldsymbol{m}}}\left(  r_{2}\right)  \text{.}%
\]
Since we know $m_{\widetilde{r_{1}\boldsymbol{m}}}\in M_{\lambda_{0}}$, we see
$\tau_{\widetilde{r_{1}\boldsymbol{m}}}\left(  r_{2}\right)  >0$, which shows
$\tau_{\boldsymbol{m}}\left(  r\right)  >0$. The positivity $\tau
_{\boldsymbol{m}}\left(  z^{n}\right)  $ is included in the statement
$\tau_{\boldsymbol{m}}\left(  r\right)  >0,$ since $\tau_{\boldsymbol{a}%
}\left(  z^{-1}\right)  =\tau_{\boldsymbol{a}}\left(  q_{0}\right)  $,
$\tau_{\boldsymbol{a}}\left(  z\right)  =\tau_{\widetilde{\boldsymbol{a}}%
}\left(  q_{0}\right)  $ for any $\boldsymbol{a}\in\boldsymbol{A}^{inv}\left(
C\right)  $, which proves $\boldsymbol{m}\in\boldsymbol{A}_{++}^{inv}\left(
C\right)  $.

For $q=\left\{  a_{n},b_{n}\right\}  _{n\in\mathbb{Z}}\in Q_{\lambda_{0}}$
define $\boldsymbol{m}\in M_{\lambda_{0}}$ by (\ref{37}) and for $g\in
\Gamma_{\operatorname{real}}$ set%
\[
\mathrm{Toda}\left(  g\right)  q=\left\{  a_{n}\left(  g\boldsymbol{m}\right)
,b_{n}\left(  g\boldsymbol{m}\right)  \right\}  _{n\in\mathbb{Z}}\text{.}%
\]
Then, Proposition \ref{p3} implies $\mathrm{Toda}\left(  g\right)  q\in
Q_{\lambda_{0}^{\prime}}$\ with $\lambda_{0}^{\prime}=\max\left\{  \pm
\lambda_{\pm}\right\}  $. Since $g\boldsymbol{m}\in\boldsymbol{A}_{++}%
^{inv}\left(  C\right)  $ for arbitrary $C$ containing $\Sigma_{\lambda_{0}}$
and $a_{n}\left(  g\boldsymbol{m}\right)  ,b_{n}\left(  g\boldsymbol{m}%
\right)  $ do not depend on $C$ due to Lemma \ref{l19}, one has $\mathrm{Toda}%
\left(  g\right)  q\in Q_{\lambda_{0}}$. The flow property $\mathrm{Toda}%
\left(  g_{1}g_{2}\right)  q=\mathrm{Toda}\left(  g_{1}\right)  \mathrm{Toda}%
\left(  g_{2}\right)  q$\ is deduced by the following fact:
\begin{gather*}
\text{For }\boldsymbol{a}\in\boldsymbol{A}_{++}^{inv}\left(  C\right)  \text{,
}g\in\Gamma_{\operatorname{real}}\left(  C\right)  \Longrightarrow
m_{g\boldsymbol{a}}=m_{g\boldsymbol{m}_{\boldsymbol{a}}}\text{ with}\\
\boldsymbol{m}_{\boldsymbol{a}}\left(  z\right)  =\left(  \dfrac
{zm_{\boldsymbol{a}}(z)-1}{z^{2}-1},z^{2}\dfrac{z-m_{\boldsymbol{a}}\left(
z\right)  }{z^{2}-1}\right)  \in\boldsymbol{A}_{++}^{inv}\left(  C\right)
\text{.}%
\end{gather*}
This is a conclusion of Lemma \ref{l19}, since 2 symbols $\boldsymbol{a}$,
$\boldsymbol{m}_{\boldsymbol{a}}\in\boldsymbol{A}_{++}^{inv}\left(  C\right)
$ have the same $m$-function $m_{\boldsymbol{a}}$. Then, setting
$\boldsymbol{a}=\boldsymbol{m}_{g_{2}\boldsymbol{m}}$, we have $m_{g_{1}%
\boldsymbol{a}}=m_{g_{1}\boldsymbol{m}_{\boldsymbol{a}}}$, which leads us to
$\mathrm{Toda}\left(  g_{1}g_{2}\right)  q=\mathrm{Toda}\left(  g_{1}\right)
\mathrm{Toda}\left(  g_{2}\right)  q$. This completes the proof of Theorem.

\section{Appendix}

\subsection{Weyl functions of Jacobi operators}

In this section we provide necessary information on spectral theory for Jacobi
operators. For $q\equiv\left\{  a_{n},b_{n}\right\}  _{n\in\mathbb{Z}}$ with
$a_{n}>0$, $b_{n}\in\mathbb{R}$ the associated Jacobi operator is defined by%
\[
\left(  H_{q}u\right)  _{n}=a_{n+1}u_{n+1}+a_{n}u_{n-1}+b_{n}u_{n}%
\]
on $\ell_{0}\left(  \mathbb{Z}\right)  =\left\{  f\in\ell^{2}\left(
\mathbb{Z}\right)  \text{; }f\text{ has a finite support}\right\}  $ as a
symmetric operator in $\ell^{2}\left(  \mathbb{Z}\right)  $. It is possible to
define a similar operator on $\ell^{2}\left(  \mathbb{Z}_{\geq1}\right)  $ by%
\[
\left(  H_{q}^{+}u\right)  _{n}=\left\{
\begin{array}
[c]{ll}%
a_{n+1}u_{n+1}+a_{n}u_{n-1}+b_{n}u_{n} & \text{for }n\geq2\\
a_{2}u_{2}+b_{1}u_{1} & \text{for }n\geq1
\end{array}
\right.  \text{.}%
\]
There are 2 cases concerning the self-adjoint extension of $H_{q}^{+}$, one is
the unique case and the other is the non-unique case. The boundary $+\infty$
is called limit point type if the extension is unique and limit circle type if
the extension is not unique. It is known (see \cite{t}) that the boundary
$+\infty$ is limit point type if and only if for any $z\in\mathbb{C}%
\backslash\mathbb{R}$ (or equivalently for some $z\in\mathbb{C}\backslash
\mathbb{R}$) the space $\left\{  u\in\ell^{2}\left(  \mathbb{Z}_{+}\right)
\text{; }H_{q}^{+}u=zu\right\}  $ is one dimensional, and in this case the
Weyl function $m_{+}$ for $H_{q}^{+}$ is defined by%
\[
m_{+}\left(  z\right)  =-\dfrac{g_{1}^{+}\left(  z\right)  }{a_{1}g_{0}%
^{+}\left(  z\right)  }\text{ for }g_{n}^{+}\left(  z\right)  \in\left\{
u\in\ell^{2}\left(  \mathbb{Z}_{+}\right)  \text{; }H_{q}^{+}u=zu\right\}
\backslash\left\{  0\right\}  \text{.}%
\]
The Weyl function $m_{-}$ for $H_{q}^{-}$ can be defined similarly by%
\[
m_{-}\left(  z\right)  =-\dfrac{g_{-1}^{-}\left(  z\right)  }{a_{0}g_{0}%
^{-}\left(  z\right)  }\text{ for }g_{n}^{-}\left(  z\right)  \in\left\{
u\in\ell^{2}\left(  \mathbb{Z}_{-}\right)  \text{; }H_{q}^{-}u=zu\right\}
\backslash\left\{  0\right\}  \text{,}%
\]
if the boundary $-\infty$ is limit point type. The Weyl functions $m_{\pm}$
are known to be analytic on $\mathbb{C}\backslash\mathbb{R}$ and satisfy
$\operatorname{Im}m_{\pm}\left(  z\right)  /\operatorname{Im}z>0$ on
$\mathbb{C}\backslash\mathbb{R}$. It is known that if coefficients
$q\equiv\left\{  a_{n},b_{n}\right\}  _{n\in\mathbb{Z}}$ are bounded, which is
equivalent to the boundedness of $H_{q}$, then the boundaries $\pm\infty$ are
limit point type. In this paper we treat only bounded operators $H_{q}$. In
this case the Green functions for $H_{q}^{\pm}$, $H_{q}$ are given by%
\begin{equation}
\left(  H_{q}^{+}-z\right)  ^{-1}\left(  j,k\right)  =\dfrac{s_{j}^{+}\left(
z\right)  g_{k}^{+}\left(  z\right)  }{-a_{1}g_{0}^{+}\left(  z\right)
}\text{ \ if \ }1\leq j\leq k \label{38-1}%
\end{equation}
with $s_{j}^{+}$ the solutions to $H_{q}s^{+}=zs^{+}$ satisfying $s_{0}^{+}%
=0$, $s_{1}^{+}=1$. Similarly one can describe $\left(  H_{q}^{-}-z\right)
^{-1}\left(  j,k\right)  $ in terms of $s_{j}^{-}$, $g_{j}^{-}$, which yields%
\begin{equation}
m_{+}\left(  z\right)  =\left(  H_{q}^{+}-z\right)  ^{-1}\left(  1,1\right)
\text{, \ \ }m_{-}\left(  z\right)  =\left(  H_{q}^{-}-z\right)  ^{-1}\left(
-1,-1\right)  \text{.} \label{38}%
\end{equation}
The Green function for $H_{q}$ is given by%
\[
\left(  H_{q}-z\right)  ^{-1}\left(  j,k\right)  =\dfrac{g_{j}^{-}\left(
z\right)  g_{k}^{+}\left(  z\right)  }{a_{1}\left(  g_{0}^{-}\left(  z\right)
g_{1}^{+}\left(  z\right)  -g_{1}^{-}\left(  z\right)  g_{0}^{+}\left(
z\right)  \right)  }\text{,}%
\]
and one has%
\[
\left\{
\begin{array}
[c]{l}%
\left(  H_{q}-z\right)  ^{-1}\left(  0,0\right)  =\dfrac{-1}{a_{1}^{2}%
m_{+}\left(  z\right)  +a_{0}^{2}m_{-}\left(  z\right)  +z-b_{0}}\\
\left(  H_{q}-z\right)  ^{-1}\left(  1,1\right)  =\dfrac{m_{+}\left(
z\right)  \left(  z-b_{0}+a_{0}^{2}m_{-}\left(  z\right)  \right)  }{a_{1}%
^{2}m_{+}\left(  z\right)  +a_{0}^{2}m_{-}\left(  z\right)  +z-b_{0}}%
\end{array}
\right.  \text{.}%
\]
Generally the spectrum \textrm{sp}$A$ of a self-adjoint operator coincides
with the singular set of the Green operator of $A$. For $H_{q}$ one has%
\[
\mathrm{sp}H_{q}=\text{the union of the singular sets of }\left(
H_{q}-z\right)  ^{-1}\left(  0,0\right)  \text{, }\left(  H_{q}-z\right)
^{-1}\left(  1,1\right)  \text{.}%
\]
If \textrm{sp}$H_{q}\subset\left[  -\lambda_{0},\lambda_{0}\right]  $ holds,
then \textrm{sp}$H_{q}^{\pm}\subset\left[  -\lambda_{0},\lambda_{0}\right]  $
are also valid, since the boundary $0$ is absorbing. Hence, in this case
$m_{\pm}\left(  z\right)  $ are analytic on $\mathbb{C}\backslash\left[
-\lambda_{0},\lambda_{0}\right]  $. Moreover, it should hold%
\[
a_{1}^{2}m_{+}\left(  z\right)  +a_{0}^{2}m_{-}\left(  z\right)  +z-b_{0}%
\neq0\text{ \ on }\mathbb{C}\backslash\left[  -\lambda_{0},\lambda_{0}\right]
\text{.}%
\]

On the other hand, in order that $a_{n}>0$ for any $n\in\mathbb{Z}$ holds
\textrm{sp}$H_{q}^{\pm}$ should be infinite set. This can be shown as follows:
Let $s_{n}=s_{n}(z)$ be a unique solution to
\[
a_{n+1}s_{n+1}+a_{n}s_{n-1}+b_{n}s_{n}=zs_{n}\text{ \ with\ }s_{0}=0\text{,
}s_{1}=1\text{.}%
\]
This is possible to define, since, due to $a_{n}\neq0$ one has%
\[
s_{1}(z)=1\text{, }s_{2}(z)=a_{2}^{-1}\left(  z-b_{1}\right)  \text{, }%
s_{3}(z)=a_{3}^{-1}\left(  a_{2}^{-1}\left(  z-b_{2}\right)  \left(
z-b_{1}\right)  -a_{2}\right)  \text{,}\cdots\text{.}%
\]
The Herglotz function $m_{+}(z)$ has a representation%
\[
m_{+}(z)=\int_{-\lambda_{0}}^{\lambda_{0}}\dfrac{\sigma_{+}\left(
d\lambda\right)  }{\lambda-z}\text{,}%
\]
and (\ref{38-1}) yields%
\[
\int_{-\lambda_{0}}^{\lambda_{0}}s_{m}\left(  \lambda\right)  s_{n}\left(
\lambda\right)  \sigma_{+}\left(  d\lambda\right)  =\left(  \delta_{m}%
,\delta_{n}\right)  \text{,}%
\]
namely $\left\{  s_{n}\left(  z\right)  \right\}  _{n\geq1}$ forms orthogonal
polynomials with respect to $\sigma_{+}$. Suppose \textrm{supp} $\sigma_{+}$
is a finite set $\left\{  \lambda_{1},\lambda_{2},\cdots,\lambda_{n}\right\}
\subset\left[  -\lambda_{0},\lambda_{0}\right]  $. Then, there exists a
polynomial $p(z)$ of degree $n+1$ such that%
\[
\left(  p,p\right)  =\int_{-\lambda_{0}}^{\lambda_{0}}\left\vert p\left(
\lambda\right)  \right\vert ^{2}\sigma_{+}\left(  d\lambda\right)  =0\text{.}%
\]
However, the polynomials $\left\{  s_{n}(z)\right\}  _{n\geq1}$ are linearly
independent (due to \textrm{deg }$s_{n}=n-1$), and hence%
\[
p(z)=c_{1}s_{1}(z)+c_{2}s_{2}(z)+\cdots+c_{n}s_{n}(z)+c_{n+1}s_{n+1}(z)
\]
holds with $c_{j}\in\mathbb{C}$. Then, the orthogonality of $\left\{
s_{n}(z)\right\}  _{n\geq1}$ implies%
\[
c_{j}=\left(  p,s_{j}\right)  =0\text{ \ for }\forall j=1\text{, }%
2\text{,}\cdots\text{, }n+1\text{,}%
\]
which is a contradiction. Consequently \textrm{supp} $\sigma_{+}$ should be infinite.

Summing up the above argument, we have

\begin{lemma}
\label{la1}Suppose a Jacobi operator $H_{q}$ satisfies \textrm{sp}%
$H_{q}\subset\left[  -\lambda_{0},\lambda_{0}\right]  $. Then, the Weyl
functions $m_{\pm}$ must satisfy\newline(i) $\ m_{\pm}\left(  z\right)  $ are
analytic on $\mathbb{C}\backslash\Sigma_{\lambda_{0}}$\ and%
\[
a_{1}^{2}m_{+}\left(  z\right)  +a_{0}^{2}m_{-}\left(  z\right)  +z-b_{0}%
\neq0\text{ \ on }\mathbb{C}\backslash\left[  -\lambda_{0},\lambda_{0}\right]
\text{.}%
\]
(ii) $m_{\pm}\left(  z\right)  $ are not rational functions.
\end{lemma}

\subsection{Transformation of Herglotz functions}

Set $\phi\left(  z\right)  =z+z^{-1}$ and in view of Lemma \ref{l13-1} define%
\[
d_{\zeta}m\left(  z\right)  =\phi\left(  z\right)  -\left(  m\left(
\zeta\right)  -m\left(  0\right)  \right)  \left(  1-\dfrac{\phi\left(
z\right)  -\phi\left(  \zeta\right)  }{m\left(  z\right)  -m\left(
\zeta\right)  }\right)  \text{.}%
\]

\begin{lemma}
\label{la2}Let $m\in M_{\lambda_{0}}$. Then, one has for $x\in\mathbb{R}%
\backslash\Sigma_{\lambda_{0}}$ and $\zeta\in\mathbb{C}_{+}\backslash
\Sigma_{\lambda_{0}}$
\[
\operatorname{Im}d_{x}m\left(  z\right)  >0\text{, \ \ }\operatorname{Im}%
d_{\overline{\zeta}}d_{\zeta}m\left(  z\right)  >0\text{ \ on \ }%
\mathbb{C}_{+}\backslash\Sigma_{\lambda_{0}}\text{.}%
\]

\end{lemma}

\begin{proof}
Note%
\begin{equation}
d_{\overline{\zeta}}d_{\zeta}m\left(  z\right)  =\left(  d_{\zeta}m\left(
0\right)  -d_{\zeta}m\left(  \overline{\zeta}\right)  \right)  \left(
1-\dfrac{\phi\left(  z\right)  -\phi\left(  \overline{\zeta}\right)
}{d_{\zeta}m\left(  z\right)  -d_{\zeta}m\left(  \overline{\zeta}\right)
}\right)  +\phi\left(  z\right)  \text{.} \label{39}%
\end{equation}
Observe%
\begin{equation}
d_{\zeta}m\left(  z\right)  =\left(  m\left(  0\right)  -m\left(
\zeta\right)  \right)  \left(  1+\dfrac{\phi\left(  \zeta\right)  }{m\left(
z\right)  -m\left(  \zeta\right)  }\right)  +z\phi\left(  z\right)
\dfrac{z^{-1}\left(  m\left(  z\right)  -m\left(  0\right)  \right)
}{m\left(  z\right)  -m\left(  \zeta\right)  }\text{,} \label{42}%
\end{equation}
and%
\[
\left(  d_{\zeta}m\right)  \left(  0\right)  =m\left(  0\right)  -m\left(
\zeta\right)  +\phi\left(  \zeta\right)  +\dfrac{m^{\prime}\left(  0\right)
}{m\left(  0\right)  -m\left(  \zeta\right)  }\text{.}%
\]
Hence%
\begin{equation}
d_{\zeta}m\left(  0\right)  -d_{\zeta}m\left(  \overline{\zeta}\right)
=\dfrac{\phi\left(  \overline{\zeta}\right)  -\phi\left(  \zeta\right)
}{m\left(  \overline{\zeta}\right)  -m\left(  \zeta\right)  }\left(  m\left(
0\right)  -m\left(  \overline{\zeta}\right)  \right)  +\dfrac{m^{\prime
}\left(  0\right)  }{m\left(  0\right)  -m\left(  \zeta\right)  } \label{40}%
\end{equation}
holds. For simplicity of notations set%
\[
X=\dfrac{m\left(  z\right)  -m\left(  \overline{\zeta}\right)  }{\phi\left(
z\right)  -\phi\left(  \overline{\zeta}\right)  }\text{, \ }Y=\dfrac{m\left(
\overline{\zeta}\right)  -m\left(  \zeta\right)  }{\phi\left(  \overline
{\zeta}\right)  -\phi\left(  \zeta\right)  }\text{.}%
\]
Then, one has%
\[
\dfrac{\phi\left(  z\right)  -\phi\left(  \zeta\right)  }{m\left(  z\right)
-m\left(  \zeta\right)  }=\dfrac{\phi\left(  z\right)  -\phi\left(
\zeta\right)  }{\left(  \phi\left(  z\right)  -\phi\left(  \overline{\zeta
}\right)  \right)  X+m\left(  \overline{\zeta}\right)  -m\left(  \zeta\right)
}\text{,}%
\]
and%
\begin{align*}
&  d_{\zeta}m\left(  z\right)  -d_{\zeta}m\left(  \overline{\zeta}\right) \\
&  =\left(  m\left(  0\right)  -m\left(  \zeta\right)  \right)  \left(
1-\dfrac{\phi\left(  z\right)  -\phi\left(  \zeta\right)  }{m\left(  z\right)
-m\left(  \zeta\right)  }\right)  +\phi\left(  z\right) \\
&  \text{ \ \ \ \ \ \ \ \ \ \ }-\left(  m\left(  0\right)  -m\left(
\zeta\right)  \right)  \left(  1-\dfrac{\phi\left(  \overline{\zeta}\right)
-\phi\left(  \zeta\right)  }{m\left(  \overline{\zeta}\right)  -m\left(
\zeta\right)  }\right)  -\phi\left(  \overline{\zeta}\right) \\
&  =\left(  \dfrac{\left(  m\left(  0\right)  -m\left(  \zeta\right)  \right)
\left(  X-Y\right)  }{Y\left(  \left(  \phi\left(  z\right)  -\phi\left(
\overline{\zeta}\right)  \right)  X+\left(  \phi\left(  \overline{\zeta
}\right)  -\phi\left(  \zeta\right)  \right)  Y\right)  }+1\right)  \left(
\phi\left(  z\right)  -\phi\left(  \overline{\zeta}\right)  \right)  \text{.}%
\end{align*}
From (\ref{39}), (\ref{40})%
\begin{align*}
&  d_{\overline{\zeta}}d_{\zeta}m\left(  z\right)  -\phi\left(  z\right) \\
&  =\left(  Y^{-1}\left(  m\left(  0\right)  -m\left(  \overline{\zeta
}\right)  \right)  +\dfrac{m^{\prime}\left(  0\right)  }{m\left(  0\right)
-m\left(  \zeta\right)  }\right)  \left(  1-\dfrac{\phi\left(  z\right)
-\phi\left(  \overline{\zeta}\right)  }{d_{\zeta}m\left(  z\right)  -d_{\zeta
}m\left(  \overline{\zeta}\right)  }\right) \\
&  =\dfrac{Y^{-1}\left\vert m\left(  0\right)  -m\left(  \zeta\right)
\right\vert ^{2}+m^{\prime}\left(  0\right)  }{m\left(  0\right)  -m\left(
\zeta\right)  +Y\dfrac{\left(  \phi\left(  z\right)  -\phi\left(
\overline{\zeta}\right)  \right)  X+\left(  \phi\left(  \overline{\zeta
}\right)  -\phi\left(  \zeta\right)  \right)  Y}{X-Y}}%
\end{align*}
follows. Now we compute the imaginary part of%
\begin{align*}
A  &  \equiv-m\left(  \zeta\right)  +Y\dfrac{\left(  \phi\left(  z\right)
-\phi\left(  \overline{\zeta}\right)  \right)  X+\left(  \phi\left(
\overline{\zeta}\right)  -\phi\left(  \zeta\right)  \right)  Y}{X-Y}\\
&  =-m\left(  \zeta\right)  -Y\overline{\phi\left(  \zeta\right)  }%
+Y\dfrac{\phi\left(  z\right)  X-\phi\left(  \zeta\right)  Y}{X-Y}\text{.}%
\end{align*}
Note that $Y\in\mathbb{R}$ and $\operatorname{Im}m\left(  \zeta\right)
=Y\operatorname{Im}\phi\left(  \zeta\right)  $. Then,%
\begin{align*}
\operatorname{Im}A  &  =Y\operatorname{Im}\dfrac{\phi\left(  z\right)
X-\phi\left(  \zeta\right)  Y}{X-Y}=Y\dfrac{\operatorname{Im}\left(
\phi\left(  z\right)  X-\phi\left(  \zeta\right)  Y\right)  \left(
\overline{X}-Y\right)  }{\left\vert X-Y\right\vert ^{2}}\\
&  =Y\dfrac{\left\vert X\right\vert ^{2}\operatorname{Im}\phi\left(  z\right)
-Y\operatorname{Im}m\left(  z\right)  }{\left\vert X-Y\right\vert ^{2}}%
\end{align*}
holds. Therefore, we have%
\begin{align}
\operatorname{Im}\left(  d_{\overline{\zeta}}d_{\zeta}m\left(  z\right)
-\phi\left(  z\right)  \right)   &  =\operatorname{Im}\dfrac{Y^{-1}\left\vert
m\left(  0\right)  -m\left(  \zeta\right)  \right\vert ^{2}+m^{\prime}\left(
0\right)  }{m\left(  0\right)  +A}\nonumber\\
&  =\dfrac{\left\vert m\left(  0\right)  -m\left(  \zeta\right)  \right\vert
^{2}+m^{\prime}\left(  0\right)  Y}{\left\vert m\left(  0\right)
+A\right\vert ^{2}}\dfrac{Y\operatorname{Im}m\left(  z\right)  -\left\vert
X\right\vert ^{2}\operatorname{Im}\phi\left(  z\right)  }{\left\vert
X-Y\right\vert ^{2}}\text{.} \label{41}%
\end{align}
Without loss of generality one can assume $\operatorname{Im}\zeta>0$. Hence
$\zeta$, $z$ $\in\mathbb{C}_{+}\backslash\left\{  \left\vert z\right\vert
=1\right\}  $, and this is divided by 2 connected domains:%
\[
\mathbb{C}_{+}\backslash\left\{  \left\vert z\right\vert =1\right\}  =\left(
\mathbb{C}_{+}\cap\left\{  \left\vert z\right\vert >1\right\}  \right)
\cup\left(  \mathbb{C}_{+}\cap\left\{  \left\vert z\right\vert <1\right\}
\right)  \text{.}%
\]
The situation is different depending on the domain where $\zeta$, $z$ belong.
If they are elements of the same domain, they can be expressed by a single
$m_{+}$ or $m_{-}$, but if they belong to distinct domains, they have
expressions by distinct $m_{\pm}$.\medskip\newline1) $\zeta$, $z\in
\mathbb{C}_{+}\cap\left\{  \left\vert z\right\vert >1\right\}  $ or $\zeta$,
$z\in\mathbb{C}_{+}\cap\left\{  \left\vert z\right\vert <1\right\}  $.

To fix our idea we consider the case $\zeta$, $z\in\mathbb{C}_{+}\cap\left\{
\left\vert z\right\vert >1\right\}  $. The other case can be treated
similarly. Setting $\phi\left(  z\right)  =w$, $\phi\left(  \zeta\right)  =b$,
one has from Lemma \ref{l17}%
\[
X=\dfrac{m\left(  \phi^{-1}\left(  w\right)  \right)  -\overline{m\left(
\phi^{-1}\left(  b\right)  \right)  }}{w-\overline{b}}=1+a_{1}^{2}%
{\displaystyle\int_{-\lambda_{0}}^{\lambda_{0}}}
\dfrac{\sigma_{+}\left(  d\lambda\right)  }{\left(  \lambda-w\right)  \left(
\lambda-\overline{b}\right)  }\text{.}%
\]
Then, Cauchy-Schwarz inequality shows%
\begin{align*}
\left\vert X\right\vert ^{2}  &  =\left\vert 1+a_{1}^{2}%
{\displaystyle\int_{-\lambda_{0}}^{\lambda_{0}}}
\dfrac{\sigma_{+}\left(  d\lambda\right)  }{\left(  \lambda-w\right)  \left(
\lambda-\overline{b}\right)  }\right\vert ^{2}\\
&  \leq\left(  1+%
{\displaystyle\int_{-\lambda_{0}}^{\lambda_{0}}}
\dfrac{a_{1}^{2}\sigma_{+}\left(  d\lambda\right)  }{\left\vert \lambda
-w\right\vert ^{2}}\right)  \left(  1+%
{\displaystyle\int_{-\lambda_{0}}^{\lambda_{0}}}
\dfrac{a_{1}^{2}\sigma_{+}\left(  d\lambda\right)  }{\left\vert \lambda
-b\right\vert ^{2}}\right)  =Y\dfrac{\operatorname{Im}m\left(  z\right)
}{\operatorname{Im}\phi\left(  z\right)  }\text{,}%
\end{align*}
hence, noting $m^{\prime}(0)=a_{0}^{2}%
{\displaystyle\int_{-\lambda_{0}}^{\lambda_{0}}}
\sigma_{-}\left(  d\lambda\right)  >0$, one has for $z\in\mathbb{C}_{+}%
\cap\left\{  \left\vert z\right\vert >1\right\}  $%
\[
\operatorname{Im}d_{\overline{\zeta}}d_{\zeta}m\left(  z\right)
\geq\operatorname{Im}\phi\left(  z\right)  >0\text{.}%
\]
\newline2) $\zeta\in\mathbb{C}_{+}\cap\left\{  \left\vert z\right\vert
<1\right\}  $, $\ z\in\mathbb{C}_{+}\cap\left\{  \left\vert z\right\vert
>1\right\}  $.

In this case $\operatorname{Im}\phi\left(  \zeta\right)  <0,$
$\operatorname{Im}m\left(  \zeta\right)  $, $\operatorname{Im}\phi\left(
z\right)  ,$ $\operatorname{Im}m\left(  z\right)  >0$ hold. Therefore, one see%
\[
\left\vert m\left(  \zeta\right)  -m(0)\right\vert ^{2}=\left\vert
{\displaystyle\int_{-\lambda_{0}}^{\lambda_{0}}}
\dfrac{a_{0}^{2}\sigma_{-}\left(  d\lambda\right)  }{\lambda-\phi\left(
\zeta\right)  }\right\vert ^{2}\leq m^{\prime}(0)%
{\displaystyle\int_{-\lambda_{0}}^{\lambda_{0}}}
\dfrac{a_{0}^{2}\sigma_{-}\left(  d\lambda\right)  }{\left\vert \lambda
-\phi\left(  \zeta\right)  \right\vert ^{2}}=m^{\prime}(0)\dfrac
{\operatorname{Im}m\left(  \zeta\right)  }{-\operatorname{Im}\phi\left(
\zeta\right)  }\text{,}%
\]
which implies%
\[
\left\vert m\left(  0\right)  -m\left(  \zeta\right)  \right\vert
^{2}+m^{\prime}\left(  0\right)  Y\leq0\text{,}%
\]
hence (\ref{41}) shows%
\[
\operatorname{Im}\left(  d_{\overline{\zeta}}d_{\zeta}m\left(  z\right)
-\phi\left(  z\right)  \right)  \geq0\text{.}%
\]
Now%
\[
\operatorname{Im}\left(  d_{\overline{\zeta}}d_{\zeta}m\left(  z\right)
\right)  =\operatorname{Im}\left(  d_{\overline{\zeta}}d_{\zeta}m\left(
z\right)  -\phi\left(  z\right)  \right)  +\operatorname{Im}\phi\left(
z\right)  >0
\]
holds, if $z\in\mathbb{C}_{+}\cap\left\{  \left\vert z\right\vert >1\right\}
$.\medskip\newline3) $\zeta\in\mathbb{C}_{+}\cap\left\{  \left\vert
z\right\vert >1\right\}  $, $\ z\in\mathbb{C}_{+}\cap\left\{  \left\vert
z\right\vert <1\right\}  $.

In this case $\operatorname{Im}\phi\left(  z\right)  <0$, $\operatorname{Im}%
\phi\left(  \zeta\right)  ,$ $\operatorname{Im}m\left(  \zeta\right)  $,
$\operatorname{Im}m\left(  z\right)  >0$ hold. Therefore, $Y>0$ and (\ref{41})
show%
\[
\operatorname{Im}\left(  d_{\overline{\zeta}}d_{\zeta}m\left(  z\right)
-\phi\left(  z\right)  \right)  \geq0\text{,}%
\]
which implies%
\[
-\left(  \left(  d_{\overline{\zeta}}d_{\zeta}m\right)  \left(  \phi
^{-1}\left(  w\right)  \right)  -w\right)  =\alpha+\beta w+%
{\displaystyle\int_{-\infty}^{\infty}}
\left(  \dfrac{1}{\lambda-w}-\dfrac{\lambda}{\lambda^{2}+1}\right)
\sigma\left(  d\lambda\right)
\]
with $w=\phi\left(  z\right)  \in\mathbb{C}_{-}$. Since (\ref{42}) implies
$d_{\zeta}m$ is analytic in a neighborhood of $0$, so does $d_{\overline
{\zeta}}d_{\zeta}m$, which means $d_{\overline{\zeta}}d_{\zeta}m\left(
0\right)  $ exists finitely. Therefore, one has%
\[
\left(  d_{\overline{\zeta}}d_{\zeta}m\right)  \left(  \phi^{-1}\left(
w\right)  \right)  =O\left(  1\right)  \text{ \ as \ }w\rightarrow\infty\text{
in }\mathbb{C}_{-}\text{,}%
\]
and hence $\beta=1$, which shows%
\[
-\left(  d_{\overline{\zeta}}d_{\zeta}m\right)  \left(  z\right)  =\alpha+%
{\displaystyle\int_{-\infty}^{\infty}}
\left(  \dfrac{1}{\lambda-\phi\left(  z\right)  }-\dfrac{\lambda}{\lambda
^{2}+1}\right)  \sigma\left(  d\lambda\right)  \text{.}%
\]
This proves $\operatorname{Im}d_{\overline{\zeta}}d_{\zeta}m\left(  z\right)
>0$ in any case.

The proof of $\operatorname{Im}d_{x}m\left(  z\right)  >0$ proceeds similarly
to the above. Here we present the proof only in 2 cases. Suppose
$x\in\mathbb{R}\cap\left\{  \left\vert z\right\vert >1\right\}  $,
$\ z\in\mathbb{C}_{+}\cap\left\{  \left\vert z\right\vert >1\right\}  $. The
definition of $d_{x}m$ implies%
\begin{equation}
\operatorname{Im}\left(  d_{x}m\left(  z\right)  -\phi\left(  z\right)
\right)  =\left(  m\left(  x\right)  -m\left(  0\right)  \right)
\operatorname{Im}\left(  \dfrac{m\left(  z\right)  -m\left(  x\right)  }%
{\phi\left(  z\right)  -\phi\left(  x\right)  }\right)  ^{-1}\text{,}
\label{43}%
\end{equation}
hence the imaginary part is%
\begin{align*}
\operatorname{Im}\dfrac{m\left(  z\right)  -m\left(  x\right)  }{\phi\left(
z\right)  -\phi\left(  x\right)  }  &  =\operatorname{Im}\left(  1+a_{1}%
^{2}\int_{-\lambda_{0}}^{\lambda_{0}}\dfrac{\sigma_{+}\left(  d\lambda\right)
}{\left(  \lambda-\phi\left(  z\right)  \right)  \left(  \lambda-\phi\left(
x\right)  \right)  }\right) \\
&  =a_{1}^{2}\operatorname{Im}\phi\left(  z\right)  \int_{-\lambda_{0}%
}^{\lambda_{0}}\dfrac{\sigma_{+}\left(  d\lambda\right)  }{\left\vert
\lambda-\phi\left(  z\right)  \right\vert ^{2}\left(  \lambda-\phi\left(
x\right)  \right)  }\text{.}%
\end{align*}
If $x>0$, then $\lambda-\phi\left(  x\right)  <0$ and
\[
m\left(  x\right)  -m\left(  0\right)  =m(x)-m(x^{-1})+m(x^{-1})-m\left(
0\right)  >0
\]
due to $m(x^{-1})-m\left(  0\right)  >0$ ($m^{\prime}(x)>0$), $m(x)-m(x^{-1}%
)>0$. Therefore, we have $\operatorname{Im}\left(  d_{x}m\left(  z\right)
-\phi\left(  z\right)  \right)  >0$, which yields $\operatorname{Im}%
d_{x}m\left(  z\right)  >0$, since $\operatorname{Im}\phi\left(  z\right)
>0$. The case $x<0$ can be treated similarly.

Suppose $x\in\mathbb{R}\cap\left\{  \left\vert z\right\vert <1\right\}  $,
$\ z\in\mathbb{C}_{+}\cap\left\{  \left\vert z\right\vert >1\right\}  $. Then,
in the expression (\ref{43}) we decompose%
\begin{align*}
\dfrac{m\left(  z\right)  -m\left(  x\right)  }{\phi\left(  z\right)
-\phi\left(  x\right)  }  &  =\dfrac{m\left(  z\right)  -m\left(
x^{-1}\right)  }{\phi\left(  z\right)  -\phi\left(  x^{-1}\right)  }%
+\dfrac{m\left(  x^{-1}\right)  -m\left(  x\right)  }{\phi\left(  z\right)
-\phi\left(  x\right)  }\\
&  =1+a_{1}^{2}\int_{-\lambda_{0}}^{\lambda_{0}}\dfrac{\sigma_{+}\left(
d\lambda\right)  }{\left(  \lambda-\phi\left(  z\right)  \right)  \left(
\lambda-\phi\left(  x^{-1}\right)  \right)  }+\dfrac{m\left(  x^{-1}\right)
-m\left(  x\right)  }{\phi\left(  z\right)  -\phi\left(  x\right)  }\text{,}%
\end{align*}
which yields $\operatorname{Im}\dfrac{m\left(  z\right)  -m\left(  x\right)
}{\phi\left(  z\right)  -\phi\left(  x\right)  }<0$ if $x>0$, and
$\operatorname{Im}d_{x}m\left(  z\right)  >0$. The rest of the cases can be
computed similarly and one can show $\operatorname{Im}d_{x}m\left(  z\right)
>0$ in any case.\medskip
\end{proof}

\noindent\textbf{Acknowledgement} \ \textit{The authors would like to
appreciate Professor Yiqian Wang for his helpful discussions. This work is
supported by the NSF of China (Grants 12271245).}


\begin{thebibliography}{9}                                                                                                %


\bibitem {k}S. Kotani: \textit{Construction of KdV flow: a unified approach},
Peking Math. Jour. \textbf{6(2)} (2023), 469-558.

\bibitem {or}D. C. Ong - C. Remling: \textit{Generalized Toda flows}, Trans.
Amer. Math. Soc. \textbf{371} (2019), 5069-5081.

\bibitem {r}C. Remling: \textit{Toda maps, cocycles, and canonical systems},
J. Spectr. Theory \textbf{9}(4) (2019), 1327--1365.

\bibitem {sa}M. Sato: \textit{Soliton equations as dynamical systems on
infinite dimensional Grassmann manifold}, Suriken Koukyuroku \textbf{439}
(1981), 30-46 (http://www.kurims.kyoto-u.ac.jp/en/publi-01.html).

\bibitem {sew}G. Segal - G. Wilson: \textit{Loop groups and equation of KdV
type}, Publ. IHES, \textbf{61} (1985), 5-65.

\bibitem {si}B. Simon: \textit{Notes on Infinite Determinant of Hilbert Space
Operators}, Advances in Math. \textbf{24} (1977), 244-273.

\bibitem {t}G. Teschl: \textit{Jacobi operators and completely integrable
nonlinear lattices}, Math. Surveys and Monographs, Amer. Math. Soc.
\textbf{72}(2000).

\bibitem {v}J. Verdera: $L^{2}$\textit{-boundedness of the Cauchy integral and
Menger curvature}, Contemp. Math. \textbf{277} (2002), 139-158.
\end{thebibliography}
\end{document}